\documentclass[a4paper,10pt, final]{amsart}
\usepackage[utf8x]{inputenc}
\usepackage{amsmath, amsthm, amsfonts,amssymb,epsfig,enumerate, wrapfig, tikz , color, tensor,comment, todo, stmaryrd, caption, graphicx}
\usepackage{setspace}
\usetikzlibrary{decorations.markings}
\newcounter{res}[section]
\numberwithin{res}{section}
\newtheorem{thm}[res]{Theorem}
\newtheorem{lem}[res]{Lemma}
\newtheorem{claim}[res]{Claim}
\newtheorem{prop}[res]{Proposition}
\newtheorem{cor}[res]{Corollary}
\theoremstyle{definition}
\newtheorem{notation}[res]{Notation}
\newtheorem{dfn}[res]{Definition}
\newtheorem{propdfn}[res]{Proposition-Definition}
\newtheorem{rmk}[res]{Remark}
\newtheorem{exa}[res]{Example}

\newtheorem{cjc}[res]{Conjecture}

\newcommand{\C}{\ensuremath{\mathcal{C}}}

\newcommand{\G}{\ensuremath{\mathcal{G}}}

\newcommand{\id}{\ensuremath{\mathrm{id}}}
\newcommand{\NN}{\mathbb N} 
\newcommand{\ZZ}{\mathbb Z} 
\newcommand{\CC}{\mathbb C} 
\newcommand{\QQ}{\mathbb Q}
\newcommand{\RR}{\mathbb R} 

\newcommand{\F}{\mathcal{F}}

\newcommand{\Ker}{\mathop{\mathrm{Ker}}\nolimits}

\newcommand{\rk}{\mathop{\mathrm{rk}}\nolimits}
\newcommand{\eps}{\varepsilon}

\def \m {^{-1}}  
  
\newcommand{\To}{\longrightarrow}

\newcommand{\END}[1]{\mathop{\mathrm{END}{(#1)}}\nolimits}

\renewcommand {\epsilon}{\varepsilon}

\newcommand{\trinomial}[4]{\ensuremath{
    \begin{pmatrix}
     #1 \\ #2\quad  #3 \quad #4      
    \end{pmatrix}
}}

\renewcommand\marginpar[1]{}

\newcommand\kup[1]{\ensuremath\left\langle #1 \right\rangle}

\let\oldtodo\todo
\renewcommand{\todo}[1]{{\color{blue} #1}\oldtodo[]{#1}}

\newcommand{\imagesfolder}{.}
\newcommand{\ie}{i.~e.~}

\newcommand{\sll}{\ensuremath{\mathfrak{sl}}}
\newcommand{\Usl}{\ensuremath{U_q(\mathfrak{sl}_3)}}
\newcommand{\NB}[1]{$\vcenter{\hbox{\tikz[scale=0.6]{#1}}}$}
\newcommand{\dops}[1]{\fill[red] #1 circle (0.2); \node[white] at #1 {$\pmb{+}$};}
\newcommand{\doms}[1]{\fill[blue!70!white] #1 circle (0.2); \node[white] at #1 {$\pmb{-}$};}
\newcommand{\dpm}{\begin{scope}[scale= 0.6, decoration={markings, mark=at position 0.5 with {\arrow{>}}},postaction={decorate}]
\draw[postaction={decorate}] (-1,-1) .. controls +(0, +1) and +(0,+1) .. (+1, -1); \end{scope}}
\newcommand{\dmp}{\begin{scope}[scale= 0.6, decoration={markings, mark=at position 0.5 with {\arrow{>}}},postaction={decorate}]
\draw[postaction={decorate}] (1,-1) .. controls +(0, +1) and +(0,+1) .. (-1, -1); \end{scope}}
\newcommand{\hppm}{
\begin{scope}[scale= 0.6, decoration={markings, mark=at position 0.5 with {\arrow{>}}},postaction={decorate}]
     \draw[postaction={decorate}] (0,1) -- (0,0);
     \draw[postaction={decorate}] (-1,-1) .. controls +(0, +0.5) and +(-0.50,-0.5) .. (0, 0);
     \draw[postaction={decorate}] (+1,-1) .. controls +(0, +0.5) and +(0.5,-0.5) .. (0, 0);
   \end{scope}
}
\newcommand{\hmmp}{
\begin{scope}[scale =0.6, decoration={markings, mark=at position 0.5 with {\arrow{<}}},postaction={decorate}]
     \draw[postaction={decorate}] (0,1) -- (0,0);
     \draw[postaction={decorate}] (-1,-1) .. controls +(0, +0.5) and +(-0.50,-0.5) .. (0, 0);
     \draw[postaction={decorate}] (+1,-1) .. controls +(0, +0.5) and +(0.5,-0.5) .. (0, 0);
   \end{scope}
}

\newcommand{\websquare}[1][0.6]{\vcenter{\hbox{\!
 \begin{tikzpicture}[scale={#1},decoration={markings, mark=at
     position 0.5 with {\arrow{>}}},postaction={decorate}]
      \draw[postaction = {decorate}] (0.5,0.5)--(1,1);
      \draw[postaction = {decorate}] (2.5,2.5)--(2,2);
      \draw[postaction = {decorate}] (2,1)--(1,1);
      \draw[postaction = {decorate}] (2,1)--(2,2);
      \draw[postaction = {decorate}] (2,1)--(2.5,0.5);
      \draw[postaction = {decorate}] (1,2)--(2,2);
      \draw[postaction = {decorate}] (1,2)--(1,1);
      \draw[postaction = {decorate}] (1,2)--(0.5,2.5);
 \end{tikzpicture}}}}
\newcommand{\webtwovert}[1][0.6]{
\vcenter{\hbox{\!
 \begin{tikzpicture}[scale={#1},decoration={markings, mark=at
    position 0.5 with {\arrow{>}}},postaction={decorate}]
      \draw[postaction= {decorate}] (0.5,0.5) .. controls (1,1) and (1,2) .. (0.5, 2.5);
      \draw[postaction= {decorate}] (2.5,2.5) .. controls (2,2) and (2,1) .. (2.5, 0.5);
 \end{tikzpicture}}}}

\newcommand{\webtwohori}[1][0.6]{\vcenter{\hbox{\!
 \begin{tikzpicture}[scale={#1},decoration={markings, mark=at
     position 0.5 with {\arrow{>}}},postaction={decorate}]
      \draw[postaction= {decorate}] (0.5,0.5) .. controls (1,1) and (2,1) .. (2.5, 0.5);
      \draw[postaction= {decorate}] (2.5,2.5) .. controls (2,2) and (1,2) .. (0.5, 2.5);
 \end{tikzpicture}}}}

\newcommand{\webvert}[1][0.6]{\vcenter{\hbox{\!
 \begin{tikzpicture}[scale={#1},decoration={markings, mark=at
     position 0.5 with {\arrow{>}}},postaction={decorate}]
      \draw[postaction= {decorate}] (0.5,0.5) -- (0.5, 2.5);
  \end{tikzpicture}}}}

\newcommand{\webbigon}[1][0.6]{\vcenter{\hbox{\!
 \begin{tikzpicture}[scale={#1},decoration={markings, mark=at
     position 0.5 with {\arrow{>}}},postaction={decorate}]
      \draw[postaction= {decorate}] (0.5,0.5) -- (0.5, 1);
      \draw[postaction= {decorate}] (0.5,2) -- (0.5, 2.5);
      \draw[postaction= {decorate}] (0.5,2) ..controls (0,1.5) and (0,1.5).. (0.5, 1);
      \draw[postaction= {decorate}] (0.5,2) ..controls (1,1.5) and (1,1.5).. (0.5, 1);
  \end{tikzpicture}}}}

\newcommand{\webcircle}[1][0.6]{\vcenter{\hbox{\!
 \begin{tikzpicture}[scale={#1},decoration={markings, mark=at
     position 0.5 with {\arrow{>}}},postaction={decorate}]
      \draw[postaction= {decorate}] (1.5,1.5) circle (1cm);
 \end{tikzpicture}}}}
\newcommand{\webcirclereverse}[1][0.6]{\vcenter{\hbox{\!
 \begin{tikzpicture}[scale={#1},decoration={markings, mark=at
     position 0.5 with {\arrow{<}}},postaction={decorate}]
      \draw[postaction= {decorate}] (1.5,1.5) circle (1cm);
 \end{tikzpicture}}}}

\newcommand{\linktwovertcirclebase}[1][0.6]{
\vcenter{\hbox{\!
 \begin{tikzpicture}[scale={#1},decoration={markings, mark=at
     position 0.5 with {\arrow{>}}},postaction={decorate}]
     \draw[very thin] (1.5,1.5) circle (1.414);
      \draw[->] (0.5,0.5) .. controls (1,1) and (1,2) .. (0.5, 2.5);
      \draw[->] (2.5,0.5) .. controls (2,1) and (2,2) .. (2.5, 2.5);
      \fill[red] (1.5,0.086) circle (3pt);      
 \end{tikzpicture}}}}

\newcommand{\linknegcrosscirclebase}[1][0.6]{\vcenter{\hbox{\!
 \begin{tikzpicture}[scale=#1]
     \draw[very thin] (1.5,1.5) circle (1.414);
      \draw[->] (1.3, 1.7) -- (0.5, 2.5);
      \draw[-] (2.5,0.5) -- (1.7, 1.3);
      \draw[->] (0.5,0.5) -- (2.5, 2.5);
      \fill[red] (1.5,0.086) circle (3pt);
 \end{tikzpicture}}}}

\newcommand{\linkposcrosscirclebase}[1][0.6]{\vcenter{\hbox{\!
 \begin{tikzpicture}[scale=#1]
     \draw[very thin] (1.5,1.5) circle (1.414);
      \draw[->] (1.7,1.7) -- (2.5, 2.5);
      \draw[-] (0.5,0.5) -- (1.3, 1.3);
      \draw[->] (2.5,0.5) -- (0.5, 2.5);
      \fill[red] (1.5,0.086) circle (3pt);
 \end{tikzpicture}}}}

\newcommand{\webIvertcirclebase}[1][0.6]{
\vcenter{\hbox{\!
 \begin{tikzpicture}[scale={#1},decoration={markings, mark=at
     position 0.5 with {\arrow{>}}},postaction={decorate}]
     \draw[very thin] (1.5,1.5) circle (1.414);
      \draw[>-] (0.5,0.5) -- (1.5, 1);
      \draw[>-] (2.5,0.5) -- (1.5, 1);
      \draw[postaction={decorate}] (1.5,2) -- (1.5,1);
      \draw[->] (1.5,2)-- (0.5,2.5);
      \draw[->] (1.5,2)-- (2.5,2.5);
      \fill[red] (1.5,0.086) circle (3pt);
 \end{tikzpicture}}}}

\newcommand{\foamcup}[1][0.6]{
\vcenter{\hbox{\!
 \begin{tikzpicture}[scale={#1},decoration={markings, mark=at
     position 0.5 with {\arrow{>}}},postaction={decorate}]
      \draw (0,0) ellipse (0.8 and 0.4 );
      \draw (-0.8,0) arc (-180:0:0.8);
 \end{tikzpicture}}}}

\newcommand{\foamcap}[1][0.6]{
\vcenter{\hbox{\!
 \begin{tikzpicture}[scale={#1},decoration={markings, mark=at
     position 0.5 with {\arrow{>}}},postaction={decorate}]
      \draw[dotted] (-0.8,0) arc (180:0:0.8 and 0.4 );
      \draw (-0.8,0) arc (-180:0:0.8 and 0.4);
      \draw (-0.8,0) arc (180:0:0.8);
 \end{tikzpicture}}}}

\newcommand{\foamdig}[1][0.6]{
\vcenter{\hbox{\!
 \begin{tikzpicture}[scale={#1},decoration={markings, mark=at
     position 0.5 with {\arrow{>}}},postaction={decorate}]
\draw (-2,0) -- +(4,0);
\draw (1,3) -- +(1,0);
\draw (-1,3) ..controls +(1,0.5) and +(-1,0.5).. +(2,0);
\draw (-1,3) ..controls +(1,-0.5) and +(-1,-0.5).. +(2,0);
\draw (-2,0) -- +(0,3);
\draw (-2,3) -- +(1,0);
\draw[postaction={decorate}] (1,3) .. controls +(0,-1.5) and +(0,-1.5).. +(-2,0);
\draw (2,0) -- +(0,3);
 \end{tikzpicture}}}}

\newcommand{\foamdigun}[1][0.6]{
\vcenter{\hbox{\!
 \begin{tikzpicture}[scale={#1},decoration={markings, mark=at
     position 0.5 with {\arrow{>}}},postaction={decorate}]
\draw (-2,0) -- +(4,0);
\draw (1,3) -- +(1,0);
\draw (-1,3) ..controls +(1,0.5) and +(-1,0.5).. +(2,0);
\draw (-1,3) ..controls +(1,-0.5) and +(-1,-0.5).. +(2,0);
\draw (-2,0) -- +(0,3);
\draw (-2,3) -- +(1,0);
\draw (1,3) .. controls +(0,-1.5) and +(0,-1.5).. +(-2,0);
\draw (2,0) -- +(0,3);
 \end{tikzpicture}}}}

\newcommand{\foamdigd}[1][0.6]{
\vcenter{\hbox{\!
 \begin{tikzpicture}[scale={#1},decoration={markings, mark=at
     position 0.5 with {\arrow{>}}},postaction={decorate}]
\draw (-2,0) -- +(4,0);
\draw (1,3) -- +(1,0);
\draw (-1,3) ..controls +(1,0.5) and +(-1,0.5).. +(2,0);
\draw (-1,3) ..controls +(1,-0.5) and +(-1,-0.5).. +(2,0);
\draw (-2,0) -- +(0,3);
\draw (-2,3) -- +(1,0);
\draw[postaction={decorate}] (1,3) .. controls +(0,-1.5) and +(0,-1.5).. +(-2,0);
\draw (2,0) -- +(0,3);
\fill (0,3) circle (4pt and 4pt);
 \end{tikzpicture}}}}

\newcommand{\foamgidd}[1][0.6]{
\vcenter{\hbox{\!
 \begin{tikzpicture}[scale={#1},decoration={markings, mark=at
     position 0.5 with {\arrow{>}}},postaction={decorate}]
\draw (-2,0) -- +(1,0);
\draw (1,0) -- +(1,0);
\draw[dashed] (-1,0) ..controls +(1,0.5) and +(-1,0.5).. +(2,0);
\draw (-1,0) ..controls +(1,-0.5) and +(-1,-0.5).. +(2,0);
\draw (-2,3) -- +(4,0);
\draw (1,3) -- +(1,0);
\draw (-2,0) -- +(0,3);
\draw (-2,3) -- +(1,0);
\draw[postaction={decorate}] (-1,0) .. controls +(0,1.5) and +(0,1.5).. +(2,0);
\draw (2,0) -- +(0,3);
\fill (0,0) circle (4pt and 4pt);
 \end{tikzpicture}}}}

\newcommand{\foamgid}[1][0.6]{
\vcenter{\hbox{\!
 \begin{tikzpicture}[scale={#1},decoration={markings, mark=at
     position 0.5 with {\arrow{>}}},postaction={decorate}]
\draw (-2,0) -- +(1,0);
\draw (1,0) -- +(1,0);
\draw[dashed] (-1,0) ..controls +(1,0.5) and +(-1,0.5).. +(2,0);
\draw (-1,0) ..controls +(1,-0.5) and +(-1,-0.5).. +(2,0);
\draw (-2,3) -- +(4,0);
\draw (1,3) -- +(1,0);
\draw (-2,0) -- +(0,3);
\draw (-2,3) -- +(1,0);
\draw[postaction={decorate}] (-1,0) .. controls +(0,1.5) and +(0,1.5).. +(2,0);
\draw (2,0) -- +(0,3);
 \end{tikzpicture}}}}

\newcommand{\foamgidun}[1][0.6]{
\vcenter{\hbox{\!
 \begin{tikzpicture}[scale={#1},decoration={markings, mark=at
     position 0.5 with {\arrow{>}}},postaction={decorate}]
\draw (-2,0) -- +(1,0);
\draw (1,0) -- +(1,0);
\draw[densely dotted] (-1,0) ..controls +(1,0.5) and +(-1,0.5).. +(2,0);
\draw (-1,0) ..controls +(1,-0.5) and +(-1,-0.5).. +(2,0);
\draw (-2,3) -- +(4,0);
\draw (1,3) -- +(1,0);
\draw (-2,0) -- +(0,3);
\draw (-2,3) -- +(1,0);
\draw (-1,0) .. controls +(0,1.5) and +(0,1.5).. +(2,0);
\draw (2,0) -- +(0,3);
 \end{tikzpicture}}}}

\newcommand{\foamunzip}[1][0.6]{
\vcenter{\hbox{\!
\begin{tikzpicture}[scale =#1, decoration={markings, mark=at
     position 0.5 with {\arrow{>}}},postaction={decorate}]
\begin{scope}[scale=0.5]
\coordinate (A2) at (1,5);
\coordinate (B2) at (4,3);
\coordinate (C2) at (-2,0);
\coordinate (D2) at (1,-2);
\coordinate (E2) at (0,-3);
\coordinate (F2) at (2,0);
\coordinate (A1) at (1,2);
\coordinate (B1) at (4,0);
\coordinate (C1) at (-2,-3);
\coordinate (D1) at (1,-5);
\coordinate (G2) at (1.5,1.8);
\draw (A1) -- (F2) -- (B1);
\draw (C1) -- (E2) -- (D1);
\draw (E2) -- (F2);
\draw (C2) .. controls +(1,1) and +(0,-1.5) .. (A2);
\draw (D2) .. controls +(0,1.5) and +(-1,-1) .. (B2);
\draw (A1) -- (A2);
\draw (B1) -- (B2);
\draw (C1) -- (C2);
\draw (D1) -- (D2);
\draw (E2) .. controls +(0,1) and  +(-0.4,0.3).. (G2).. controls +(0.4,-0.3) and +(0,0.5) .. (F2);
 \draw[densely dotted] (A2) .. controls +(0,0) and +( -0.5, +0.5) ..  (G2) .. controls +(0.5,-0.5) and +(0,0) .. (B2);
+(0,0).. (D2);
\fill[opacity=0.4, color= gray] (E2) .. controls +(0,1) and  +(-0.4,0.3).. (G2).. controls +(0.4,-0.3) and +(0,+0.5) .. (F2) --(E2);
\end{scope}
\end{tikzpicture}}}}

\newcommand{\foamzip}[1][0.6]{
\vcenter{\hbox{\!
\begin{tikzpicture}[scale =#1, decoration={markings, mark=at
     position 0.5 with {\arrow{>}}},postaction={decorate}]
\begin{scope}[xscale=-0.5, yscale=-0.5]
\coordinate (A2) at (1,5);
\coordinate (B2) at (4,3);
\coordinate (C2) at (-2,0);
\coordinate (D2) at (1,-2);
\coordinate (E2) at (0,-3);
\coordinate (F2) at (2,0);
\coordinate (A1) at (1,2);
\coordinate (B1) at (4,0);
\coordinate (C1) at (-2,-3);
\coordinate (D1) at (1,-5);
\coordinate (G2) at (1.5,1.8);
\draw (A1) -- (F2) -- (B1);
\draw (C1) -- (E2) -- (D1);
\draw (E2) -- (F2);
\draw (C2) .. controls +(1,1) and +(0,-1.5) .. (A2);
\draw (D2) .. controls +(0,1.5) and +(-1,-1) .. (B2);
\draw (A1) -- (A2);
\draw (B1) -- (B2);
\draw (C1) -- (C2);
\draw (D1) -- (D2);
\draw (E2) .. controls +(0,1) and  +(-0.4,0.3).. (G2).. controls +(0.4,-0.3) and +(0,0.5) .. (F2);
 \draw[densely dotted] (A2) .. controls +(0,0) and +( -0.5, +0.5) ..  (G2) .. controls +(0.5,-0.5) and +(0,0) .. (B2);
+(0,0).. (D2);
\fill[opacity=0.4, color= gray] (E2) .. controls +(0,1) and  +(-0.4,0.3).. (G2).. controls +(0.4,-0.3) and +(0,+0.5) .. (F2) --(E2);
\end{scope}
\end{tikzpicture}}}}

\newcommand{\websourcebase}[1][0.6]{
\vcenter{\hbox{\!
 \begin{tikzpicture}[scale={#1},decoration={markings, mark=at
     position 0.5 with {\arrow{>}}},postaction={decorate}]
     \draw[very thin] (0,0) circle (1cm);
\draw[postaction={decorate}] (0,0) --(90:1);
\draw[postaction={decorate}] (0,0) --(-30:1);
\draw[postaction={decorate}] (0,0) --(210:1);
\fill[red] (-90:1) circle (2pt);
 \end{tikzpicture}}}}

\newcommand{\websinkbase}[1][0.6]{
\vcenter{\hbox{\!
 \begin{tikzpicture}[scale={#1},decoration={markings, mark=at
     position 0.5 with {\arrow{<}}},postaction={decorate}]
     \draw[very thin] (0,0) circle (1cm);
\draw[postaction={decorate}] (0,0) --(90:1);
\draw[postaction={decorate}] (0,0) --(-30:1);
\draw[postaction={decorate}] (0,0) --(210:1);
\fill[red] (-90:1) circle (2pt);
 \end{tikzpicture}}}}

\newcommand{\foamsaddle}[1][0.6]{
\vcenter{\hbox{\!
\begin{tikzpicture}[scale =#1, decoration={markings, mark=at
     position 0.5 with {\arrow{>}}},postaction={decorate}]
\begin{scope}[xscale=0.5, yscale=0.5]
\coordinate (A2) at (1,5);
\coordinate (B2) at (4,3);
\coordinate (C2) at (-2,0);
\coordinate (D2) at (1,-2);
\coordinate (E2) at (0,-3);
\coordinate (F2) at (1.5,0.5);
\coordinate (A1) at (1,2);
\coordinate (B1) at (4,0);
\coordinate (C1) at (-2,-3);
\coordinate (D1) at (1,-5);
\coordinate (G2) at (1.2,0);
\draw (A1) .. controls +(0,0) and +(-0.4,0.3).. (F2) .. controls +(0.4,-0.30) and +(0,0) .. (B1);
\draw (C1) .. controls +(0,0) and +(-0.4,0.30).. (E2) .. controls +(0.4,-0.3) and +(0,0) .. (D1);
\draw (C2) .. controls +(1,1) and +(0,-1.5) .. (A2) coordinate[midway] (Z1) {};
\draw (D2) .. controls +(0,1.5) and +(-1,-1) .. (B2) coordinate [midway] (Z2) {};
\draw (A1) -- (A2);
\draw (B1) -- (B2);
\draw (C1) -- (C2);
\draw (D1) -- (D2);
\draw[densely dotted] (E2) .. controls +(0,0) and  +(-0.4,-0.40).. (G2).. controls +(0.4,0.4) and +(0,0) .. (F2);
 \draw[densely dotted] (Z1) .. controls +(0,0) and +( -0.5,0) ..  (G2) .. controls +(0.5,-0) and +(0,0) .. (Z2);
+(0,0).. (D2);
\end{scope}
\end{tikzpicture}}}}

\title{Categorification of the colored $\sll_3$-invariant}
\author{Louis-Hadrien Robert}
\address{Universit\"a{}t Hamburg \\ Bundesstra\ss{}e 55\\ 20146 Hamburg \\ Germany }
\email{louis-hadrien.robert@uni-hamburg.de}

\newcommand{\V}{\ensuremath{V}}

\begin{document}

\let\oldnu\nu
\renewcommand{\nu}{\eta}

\begin{abstract}
  We give  explicit resolutions of all finite dimensional, simple \(U_q(\mathfrak{sl_3})\)-modules. We use these resolutions to categorify the colored $\mathfrak{sl}_3$-invariant of framed links via a complex of complexes of graded $\ZZ$-modules.  
\end{abstract}
\maketitle
\tableofcontents

\section{Introduction}
\label{sec:introduction}

In this paper, we define a categorification of the colored $\sll_3$-invariant; that is the Reshetikhin-Turaev invariant for framed links associated with arbitrary finite dimensional $\Usl$-modules. 

A categorification of the (uncolored) $\sll_3$-invariant has been defined by Khovanov \cite{MR2100691}, the so-called $\sll_3$-homology. The construction is topological and involves webs (trivalent, bipartite, planar graphs) and foams (natural cobordisms between webs). 

Two categorifications of the colored Jones polynomial (the $\sll_2$-analogue of the colored $\sll_3$-invariant) are defined by Khovanov \cite{MR2124557}. One of them is developed by Beliakova and Wehrli \cite{MR2462446}. The main idea is to give an explicit resolution of every finite dimensional simple $U_q(\sll_2)$-module. The morphisms which appear in the resolution have a rather simple expression and hence can be interpreted topologically. Using this interpretation and the categorification of the (uncolored) Jones polynomial (\ie the Khovanov homology), one achieves the construction of a categorification of the colored Jones polynomial.

In this paper, we mimic the strategy of Khovanov in the $\sll_3$-case: we find a ``nice'' resolution for every finite-dimensional simple $\Usl$-module (see theorem~\ref{mainthm}) and we interpret this resolution in terms of diagrams and foam-cobordisms. Finally, we use the $\sll_3$-homology to defined a categorification of the colored $\sll_3$-invariant (see theorem~\ref{thm:main}):

\vspace{0.3cm}
\emph{
If $D$ is an oriented link diagram colored with $V_{\mathbf{m},\mathbf{n}}$ (a collection finite dimensional, simple $\Usl$-modules), then the 
isomorphism type of $\C_\bullet(D)$ depends only on $(\mathbf{m}$, $\mathbf{n})$ and on the framed oriented isotopy type of $D$ and we have:
\[
\chi(\C_\bullet(D)) = e^{i\pi\frac23(n_+-n_-)s(\mathbf{m}-\mathbf{n})}\kup{D},
\]
Where the exponential term is a re-normalization by a complex number of module 1 and $\kup{D}$ is the colored $\sll_3$-invariant.
}
\vspace{0.3cm}

Clark \cite{MR2482322} has proven that the $\sll_3$-homology is functorial for links cobordisms. But it is not known to be functorial (or even projectively functorial) for foam-cobordisms. In this paper, we bypass the lack of functoriality by an ad-hoc construction. 

Rose \cite{MR3029720} categorified the $\sll_3$-analogue of the Jones-Wenzl projectors. This yields as well a categorification of the colored $\sll_3$-invariant. The construction we give here is quite different: We do not use complex of infinite length because we do not need to categorify some rational coefficients. It does not mean that our construction is more computable since we intensively cable link diagrams, and this operation dramatically increases the number of crossings.

\subsection*{Organisation of the paper}
\label{sec:organisation-paper}

In a first part, we shortly present the two main ingredients needed to define the colored $\sll_3$-homology:
\begin{itemize}
\item The colored $\sll_3$-invariant in section \ref{sec:color-sll_3-invar-1}. For this purpose we recall basic facts about the finite dimensional $\Usl$-modules in section~\ref{sec:repr-sll_3}.
\item The uncolored $\sll_3$-homology in section~\ref{sec:sll_3-homology}. Our normalization differs non-trivially from the original definition since we consider framed links. 
\end{itemize}

In section~\ref{sec:tens-resol-os}, we construct an explicit resolution of every finite dimensional, simple $\Usl$-module of type $1$ by tensor products of the fundamental representation and its dual (theorem~\ref{mainthm}). From the algebraic point of view, this result is very easy, but it is the key to construct the colored $\sll_3$-homology.

In the last section, we construct the colored $\sll_3$-homology. To give an idea of the construction, we start by explaining it with the hypothesis that the $\sll_3$-homology is functorial with respect to foam-cobordism. Then we explain how to bypass this hypothesis. We use a trick due to Bar-Natan (see\cite{MR2174270}) to prove some equalities up to a sign. 

In appendix~\ref{sec:complexes-cubes}, we set up a general framework to deal with the sign indeterminacy introduced in section~\ref{sec:categ-color-sll_3}.

\subsection*{Acknowledgments}
\label{sec:acknowledgments}
The author wishes to thank Mikhail Khovanov, Scott Carter, David Clark and Matt Hogancamp for enlightening conversations; Christian Blanchet and François Costantino for their constant supports; and the Max-Planck-Institute f\"u{}r Mathematik for its hospitality.


\section{Preliminiaries}
\label{sec:preliminiaries}

\subsection{The quantum group $U_q(\sll_3)$ and its finite dimensional representations}
\label{sec:repr-sll_3}

\begin{dfn}
The algebra $\Usl$ is the unital associative $\CC\left( q^{\frac16}\right)$ algebra generated by the elements $K_i$, $K_i\m$, $E_i$ and $F_i$ for $i\in\{1,2\}$ subjected to the following relations (for $i$ and $j$ in $\{1,2\}$ with $i\neq j$):
\begin{align*}
  &K_iK_i\m =K_i\m K_i= 1,  & &K_iK_j =K_jK_i,  \\
&K_i E_i = q^2 E_i K_i,& &K_i F_i = q^{-2} F_i K_i, \\
&K_i E_j = q E_i K_i,& &K_i F_i = q^{-1} F_i K_i, \\
&(q-q\m)(E_iF_i - F_iE_i)= K_i-K_i\m  & &E_iF_j = F_jE_i\\ 
&{E_i}^2E_j - [2] E_iE_jE_i + E_j{E_i}^2 = 0 & & F_i^2F_j - [2] F_iF_jF_i + F_j{F_i}^2 = 0 
\end{align*}
\end{dfn}
\begin{prop}[{\cite[Chapter 4.2]{kassel97:_quant}}]
  The following data turn $\Usl$ into a Hopf algebra:
  \begin{align*}
     &\Delta(K_i^{\pm1}) = K_i^{\pm1}\otimes K_i^{\pm1} & &\eps(K_i^{\pm1}) =1 && S(K_i) = K_i\m \\
     &\Delta(E_i) = E_i \otimes 1 + K_i \otimes E_i & &\eps(E_i)=0 && S(E_i) = -K_i\m E_i \\
     &\Delta(E_i) = F_i \otimes K_i\m + 1\otimes F_i & &\eps(F_i)=0 && S(F_i)= -K_i F_i \\
  \end{align*}
Furthermore, $\Usl$ admits a (virtual) universal $R$-matrix in an appropriate completion of $\Usl\otimes \Usl$. All these facts together imply that the category $U_q(\sll_3)$-$\mathsf{mod}$ of finite dimensional $U_q(\sll_3)$-modules is monoidal, autonomous and braided. 
\end{prop}

\begin{propdfn}[{\cite[Theorem 7.1.1]{kassel97:_quant}}]
  Let $n_1$ and $n_2$ be two non-negative integers and $\epsilon_1$ and $\epsilon_2$ two signs (elements of $\{-1, +1 \}$). Then there is a unique (up to isomorphism) finite dimensional\footnote{The ground field is $\CC\left(q^{\frac16}\right)$.} simple $U_q(\sll_3)$-module $\V_{n_1, n_2, \epsilon_1,\epsilon_2}$ such that there exists a non-zero vector $v$ in $\V_{n_1, n_2,\epsilon_1,\epsilon_2}$ satisfying the following conditions:
  \begin{align*}
    E_1\cdot v = 0,\qquad E_2\cdot v = 0,\qquad K_1\cdot v= \epsilon_1q^{n_1}v,\qquad\textrm{and}\qquad K_2\cdot \epsilon_2q^{n_2}v.
  \end{align*}
Such a vector $v$ is called a \emph{highest weight vector} of weight\footnote{$(n_1, n_2)^\star$ is the linear form on $\CC\left(q^{\frac16}\right) K_1 \oplus \CC\left(q^{\frac16}\right) K_2$, such that $(n_1, n_2)^\star(K_1)=n_1$ and $(n_1, n_2)^\star(K_2)=n_2$.} $(n_1, n_2)^\star$. 
The modules $\V_{0,0,\epsilon_1,\epsilon_2}$ have dimension 1 and we have the following isomorphisms for all $(\epsilon_1,\epsilon_2)$ and all $(n_1, n_2)$:
\[
\V_{n_1,n_2,+1,+1}\otimes \V_{0,0,\epsilon_1,\epsilon_2} \simeq  \V_{n_1, n_2,\epsilon_1,\epsilon_2}\simeq \V_{0,0,\epsilon_1,\epsilon_2}  \otimes \V_{n_1,n_2,+1,+1}.
\] 
This is why from now on, we consider only the modules $(\V_{n_1,n_2,+1,+1})_{(n_1, n_2)\in \NN}$ (they are said to be \emph{of type 1}). We write $\V_{n_1,n_2}:= \V_{n_1,n_2, +1,+1}$, $\V_+:= \V_{1,0}$ and $\V_-:=\V_{0,1}$. The dual of $V_{n_1, n_2}$ is isomorphic to $V_{n_2, n_1}$.
It is convenient to take the following convention: if $n_1$ or $n_2$ is negative we set: $V_{n_1, n_2}=\{0\}$.
\end{propdfn}

We would like to introduce a few morphisms of $\Usl$-modules which will be useful later. We first need to introduce some bases for $V_+$ and $V_-$. Both are three dimensional, we set $V_+=\left< v^+_{-1}, v^+_{0}, v^+_{1}\right>$ and $V_- = \left<v^-_{-1}, v^-_{0}, v^-_{1} \right>$ and the structures of $\Usl$-modules are given by:
\begin{align*}
&K_1\cdot v^+_{-1} = v^+_{-1}, & & K_1\cdot v^+_{0} = q\m v^+_{0}, & & K_1\cdot v^+_{1} = q v^+_{1}, \\  
&K_2\cdot v^+_{-1} = q\m v^+_{-1}, & & K_2\cdot v^+_{0} =q v^+_{0}, & & K_2\cdot v^+_{1} = v^+_{1}, \\ 
&E_1\cdot v^+_{-1} = 0, & & E_1\cdot v^+_{0} = v^+_{1}, & & E_1\cdot v^+_{1} = 0, \\  
&E_2\cdot v^+_{-1} = v^+_{0}, & & E_2\cdot v^+_{0} =0, & & E_2\cdot v^+_{1} = 0, \\ 
&F_1\cdot v^+_{-1} = 0, & & F_1\cdot v^+_{0} = 0, & & F_1\cdot v^+_{1} = v^+_0, \\  
&F_2\cdot v^+_{-1} = 0, & & F_2\cdot v^+_{0} = v^+_{-1}, & & F_2\cdot v^+_{1} = 0, \\ 
\end{align*}
and
\begin{align*}
&K_1\cdot v^-_{-1} = q\m v^-_{-1} ,& & K_1\cdot v^-_{0} = q v^-_{0} ,& & K_1\cdot v^-_{1} = v^-_{1}, \\  
&K_2\cdot v^-_{-1} =  v^-_{-1} ,& & K_2\cdot v^-_{0} =q\m v^-_{0} ,& & K_2\cdot v^-_{1} = q v^-_{1}, \\ 
&E_1\cdot v^-_{-1} = v^-_0 ,& & E_1\cdot v^-_{0} = 0 ,& & E_1\cdot v^-_{1} = 0, \\  
&E_2\cdot v^-_{-1} = 0 ,& & E_2\cdot v^-_{0} =v^-_{1} ,& & E_2\cdot v^-_{1} = 0, \\ 
&F_1\cdot v^-_{-1} = 0 ,& & F_1\cdot v^-_{0} = v^-_{-1} ,& & F_1\cdot v^-_{1} = 0, \\  
&F_2\cdot v^-_{-1} = 0 ,& & F_2\cdot v^-_{0} = 0 ,& & F_2\cdot v^-_{1} = v^{-}_{0}, \\ 
\end{align*}

\begin{dfn}\label{dfn:Uslmorphisms}
  We consider 6 morphisms:
\begin{align*}
&b_{-+}: V_{0,0} \to V_-\otimes V_+, && b_{+-}: V_{0,0} \to V_+\otimes V_-, \\ 
&d_{-+}:  V_-\otimes V_+ \to V_{0,0}, && d_{-+}:  V_+\otimes V_- \to V_{0,0}, \\
&h_{++}^-: V_+ \otimes V_+ \to V_- &&h_{--}^+: V_- \otimes V_- \to V_+, \\
\end{align*} defined on the bases of $V_+$ and $V_-$ by (the images of all missing elementary tensors are meant to be equal to 0):
\begin{align*}
  b_{-+}(1)&= -q\m v^-_{-1}\otimes v^+_{1} + v^-_{0}\otimes v^+_{0} - q v^-_{1}\otimes v^+_{-1},\\
  b_{+-}(1)&= -q\m v^+_{-1}\otimes v^-_{1} + v^+_{0}\otimes v^-_{0} - q v^+_{+1}\otimes v^-_{-1},\\
\end{align*}
by:
\begin{align*}
  &d_{-+} (v^-_{-1}\otimes v^+_{1})= - q\m,   &&  d_{+-} (v^+_{-1}\otimes v^-_{1})= - q\m,  \\
  &d_{-+} (v^-_{0}\otimes v^+_{0}) =1,      &&    d_{+-} (v^+_{0}\otimes v^-_{0}) =1,       \\
  &d_{-+} (v^-_{1}\otimes v^+_{-1}) = -q,   &&    d_{+-} (v^+_{1}\otimes v^-_{-1}) = -q,    \\  
\end{align*}
and by:
\begin{align*}
  &h_{++}^-(v^+_{-1}\otimes v^+_{0}) = q^{-\frac12}v^-_{-1}, &   &h_{--}^+(v^-_{-1}\otimes v^-_{0}) = q^{-\frac12}  v^+_{-1}, \\
  &h_{++}^-(v^+_{0}\otimes v^+_{-1}) = - q^{\frac12}v^-_{-1}, &  &h_{--}^+(v^-_{0} \otimes v^-_{-1}) = - q^{\frac12}v^+_{-1},\\
  &h_{++}^-(v^+_{-1}\otimes v^+_{1}) =  q^{-\frac12}v^-_{0}, &  &h_{--}^+(v^-_{-1}\otimes v^-_{1}) =  q^{-\frac12} v^+_{0}, \\
  &h_{++}^-(v^+_{1}\otimes v^+_{-1}) = - q^{\frac12}v^-_{0}, &  &h_{--}^+(v^-_{1} \otimes v^-_{-1}) = - q^{\frac12}v^+_{0}, \\
  &h_{++}^-(v^+_{0} \otimes v^+_{1}) =  q^{-\frac12}v^-_{1},&  &h_{--}^+(v^-_{0} \otimes v^-_{1}) =  q^{-\frac12} v^+_{1},\\
  &h_{++}^-(v^+_{1} \otimes v^+_{0}) = - q^{\frac12}v^-_{1},&  &h_{--}^+(v^-_{1} \otimes v^-_{0}) = - q^{\frac12} v^+_{1}.\\
\end{align*}
\end{dfn}
\begin{prop}\label{prop:easyrelUsl}
The maps of definition~\ref{dfn:Uslmorphisms} are morphisms of $\Usl$-modules. The following relations hold:
  \begin{align*}
& (\id_{V_+} \otimes d_{-+}) \circ (b_{+-} \otimes \id_{V_+}) =  (d_{-+} \otimes \id_{V_+}) \circ (\id_{V_+} \otimes b_{+-}) = \id_{V_+}, \\
& (\id_{V_-} \otimes d_{+-}) \circ (b_{-+} \otimes \id_{V_-}) =  (d_{+-} \otimes \id_{V_-}) \circ (\id_{V_-} \otimes b_{-+}) = \id_{V_-} \\
& d_{+-}\circ(\id_{V_+} \otimes h_{++}^-) =    d_{-+}\circ(h_{++}^- \otimes \id_{V_+}),\\
& d_{-+}\circ(\id_{V_-} \otimes h_{--}^+) =    d_{+-}\circ(h_{--}^+ \otimes \id_{V_-}).
  \end{align*}
\end{prop}
\begin{proof}
  These are easy computations.
\end{proof}
\begin{rmk}\label{rmk:vpvmdual}
  The proposition~\ref{prop:easyrelUsl} tells in particular that $V_-$ is a (right and left) dual of $V_+$ and vice-versa, hence this fixes an isomoprhism between $V_+^*$ and $V_-$ and between $V_-^*$ and $V_+$.
\end{rmk}

\begin{prop}[Littlewood-Richardson (simple) rules]\label{prop:CGformula}
The category $\Usl$-$\mathsf{mod}$ is semi-simple and the following relations hold for $m$ and $n$ non-negative integers:
  \begin{align*}
   & V_+\otimes V_{m,n}  \simeq V_{m+1,n} \oplus V_{m-1, n+1} \oplus V_{m, n-1} \\
   & V_-\otimes V_{m,n}  \simeq V_{m,n+1} \oplus V_{m+1, n-1} \oplus V_{m-1, n} \\
   & V_{m,n} \oplus (V_{m-1,0}\otimes V_{0,n-1})
                         \simeq V_{m,0}\otimes V_{0,n} \\
   & V_{m,0} \oplus (V_{0,1} \otimes V_{m-2,0} )
                         \simeq V_+\otimes V_{m-1,0}\oplus V_{m-3,0} \\
   & V_{0,n} \oplus (V_{1,0} \otimes V_{0,n-2} )
                         \simeq V_-\otimes V_{0,n-1}\oplus V_{0,n-3} \\
  \end{align*}
\end{prop}
Thanks to these formulas one can express $V_{m,n}$ in terms of tensor product of $V_+$ and $V_-$. It is convenient to introduce the trinomial coefficients:
\begin{notation}
  Let $a$, $b$ and $c$ three integers, then we define:
  \[\trinomial{a+b+c}{a}{b}{c}=
  \begin{cases}
    \displaystyle{\frac{(a+b+c)!}{a!b!c!}} &\textrm{if $a$, $b$ and $c$ are non-negative,}\\
0& \textrm{else}.
  \end{cases}
  \]
If $a$, $b$ and $n$ are three integers, we define:
\[
 {\trinomial{n}{a}{b}{\blacksquare}} = \trinomial{n}{a}{b}{(n-a-b)}.
\]
\end{notation}

Thanks to proposition \ref{prop:CGformula} and an easy recursion, we have the following corollary:
\begin{cor} \label{cor:explicitformula}
Let $m$ and $n$ be two non-negative integers, then we have the following isomorphism:
\begin{align*}&V_{m,n} \simeq \bigoplus_{\substack{(i,j,k,l) \in \NN^4 \\ \delta \in \{0,1\} }} (-1)^{\delta+i+k}
\trinomial{m-\delta-i -2j}{i}{j}{\blacksquare} 
\cdot\trinomial{n-\delta-k -2l}{k}{l}{\blacksquare}  \\ & \hspace{3cm}   
V_{1,0}^{\otimes m-2i+k-3j-\delta}\otimes V_{0,1}^{{\otimes n+i-2k-3l-\delta}}.
\end{align*}
In this expression, the terms of the sum with minus signs are meant to be on the left-hand side without minus signs. One can as well see this expression as an equality in the Grothendieck ring of $\Usl$-$\mathsf{mod}$.
\end{cor}


\subsection{The colored $\sll_3$-invariant}
\label{sec:color-sll_3-invar-1}

In this paper, all links (and tangles) are oriented and otherwised specified endowed with a framing. All link diagrams are oriented. When a diagram represents a framed link, the framing is understood to be the blackboard framing. Diagrams are read from bottom to top. We consider as well colored link diagrams. The set of colors is  the set of iso-classes of finite dimensional representations of $U_q(\sll_3)$. 

The combinatorial data given by the orientations and the colors are not completely independent: If a component $l$ of a link $L$ (resp. link diagram $D$) is colored with $V$, and if $L'$ is the same link (resp. is the same link diagram) except that the orientation of $l$ is reversed and the color of $l$ is $V^*$ the dual representation of $V$, $L$ and $L'$ (resp. $D$ and $D'$) are considered to be the same.


We fix once for all a duality structure of $U_q(\sll_3)$-$\mathsf{mod}$ compatible with the maps $d$'s and $b$'s. We introduce now the graphical notation: an upward oriented vertical strand colored with $V$ an object of  $U_q(\sll_3)$-$\mathsf{mod}$ represents the identity of this objects. The same strand oriented downward represents the identity of the dual of this object. The diagrams
\[
\begin{tikzpicture}[scale = 0.8]
  \begin{scope}[decoration={markings, mark=at
     position 0.5 with {\arrow{>}}},postaction={decorate}]
   \begin{scope}
     \draw[postaction={decorate}] (-1,0) .. controls +(0, -1) and +(0,-1) .. (1, 0) node [midway, above] {$V$};
   \end{scope}
\begin{scope}[xshift=4cm]
     \draw[postaction={decorate}] (+1,0) .. controls +(0, -1) and +(0,-1) .. (-1, 0) node [midway, above] {$V$};
   \end{scope}
\begin{scope}[xshift=8cm]
     \draw[postaction={decorate}] (-1,-1) .. controls +(0, +1) and +(0,+1) .. (+1, -1) node [midway, below] {$V$};
   \end{scope}
\begin{scope}[xshift=12cm]
     \draw[postaction={decorate}] (+1,-1) .. controls +(0, +1) and +(0,+1) .. (-1, -1) node [midway, below] {$V$};
   \end{scope}
\end{scope}
\end{tikzpicture}
\]
represent the morphisms $\CC\left(q^{\frac16}\right) \to V\otimes V^*$, $\CC\left(q^{\frac16}\right) \to V^*\otimes V$, $V\otimes V^* \to \CC\left(q^{\frac16}\right)$ and $V^*\otimes V \to \CC\left(q^{\frac16}\right)$ given by the duality structure. We have: 
\[
\begin{tikzpicture}[scale = 0.8]
  \begin{scope}[decoration={markings, mark=at
     position 0.5 with {\arrow{>}}},postaction={decorate}]
   \begin{scope}
     \draw[postaction={decorate}] (-1,0) .. controls +(0, -1) and +(0,-1) .. (1, 0) node [midway, above] {$V$};
   \end{scope} 
\node at (2, -0.5) {$=$};
\begin{scope}[xshift=4cm]
     \draw[postaction={decorate}] (+1,0) .. controls +(0, -1) and +(0,-1) .. (-1, 0) node [midway, above] {$V^*$};
   \end{scope}
\node at (6.5, -0.5) {$\textrm{and}$};
\begin{scope}[xshift=9cm]
     \draw[postaction={decorate}] (-1,-1) .. controls +(0, +1) and +(0,+1) .. (+1, -1) node [midway, below] {$V$};
   \end{scope}
\node at (11, -0.5) {$=$};
\begin{scope}[xshift=13cm]
     \draw[postaction={decorate}] (+1,-1) .. controls +(0, +1) and +(0,+1) .. (-1, -1) node [midway, below] {$V^*$};
   \end{scope}
\node at (14.1, -0.5) {$\ .$};
\end{scope}
\end{tikzpicture}
\]

The diagrams 
\[
\begin{tikzpicture}[scale = 0.7]
  \begin{scope}[decoration={markings, mark=at
     position 0.5 with {\arrow{>}}},postaction={decorate}]
   \begin{scope}
     \draw[->] (-1,-1) .. controls +(0, 1) and +(0,-1) .. (1, 1) node [near start, left] {$V$};
     \fill [white] (0,0) circle (0.2);
     \draw[->] (1,-1) .. controls +(0, 1) and +(0,-1) .. (-1, 1) node [near start, right] {$W$};
   \end{scope} 
   \begin{scope}[xshift = 5cm]
     \draw[->] (1,-1) .. controls +(0, 1) and +(0,-1) .. (-1, 1) node [near start, right] {$W$};
     \fill [white] (0,0) circle (0.2);
     \draw[->] (-1,-1) .. controls +(0, 1) and +(0,-1) .. (1, 1) node [near start, left] {$V$};
   \end{scope} 
\end{scope}
\end{tikzpicture}
\]
represent the morphisms $V\otimes W \to W\otimes V$ induced by the braided structure.
The juxtaposition of two diagrams side by side represents the tensor product of the two associated morphisms. Stacking two (compatible) diagrams one onto the other, we obtain another diagram which is meant to be the composition of the two morphisms.
\begin{rmk}\label{rmk:tensorcable}
Note that the duality and the braiding being compatible with tensor product, we may replace a strand colored with $V\otimes W$ by two parallel strands, one colored with $V$, the other with $W$. 
For example we have:  
\[
\begin{tikzpicture}[scale = 0.8]
  \begin{scope}[decoration={markings, mark=at
     position 0.5 with {\arrow{>}}},postaction={decorate}]
   \begin{scope}
     \draw[->] (-1,-1) .. controls +(0, 1) and +(0,-1) .. (1, 1) node [near start, left] {$V$};
     \fill [white] (0,0) circle (0.2);
     \draw[->] (1,-1) .. controls +(0, 1) and +(0,-1) .. (-1, 1) node [near start, right] {$W_1\otimes W_2$};
   \end{scope} 
\node at  (2, 0) {$=$};
   \begin{scope}[xshift = 4cm]
     \draw[->] (-1,-1) .. controls +(0, 1) and +(0,-1) .. (1, 1) node [near start, left] {$V$};     
     \fill [white] (0,0) circle (0.3);
   \draw[->] (1.1,-1) .. controls +(0, 1.1) and +(0,-0.9) .. (-0.9, 1) node [near start, right] {$W_2$};
\draw[->] (0.9,-1) .. controls +(0, 0.9) and +(0,-1.1) .. (-1.1, 1) node [near start, left] {$W_1$};
\end{scope} 
\end{scope}
\node at (6.5,0) {$\textrm{and}$};
\begin{scope}[yshift = 0.5cm, xshift = 9cm, decoration={markings, mark=at
     position 0.5 with {\arrow{>}}},postaction={decorate}]
   \begin{scope}
     \draw[postaction={decorate}] (-1,0) .. controls +(0, -1) and +(0,-1) .. (1, 0) node [midway, above] {$V\otimes W$};
   \end{scope}
\node at  (2, -0.50) {$=$};
\begin{scope}[xshift=4cm]
     \draw[postaction={decorate}] (-1.1,0) .. controls +(0, -1.1) and +(0,-1.1) .. (1.1, 0) node [midway, below] {$W$};
     \draw[postaction={decorate}] (-0.9,0) .. controls +(0, -0.9) and +(0,-0.9) .. (0.9, 0) node [midway, above] {$V$};
   \end{scope}
\end{scope}
\node at (14.1, -0.1) {$.$};
\end{tikzpicture}
\]

\end{rmk}

If $D$ is a ``well-presented'' colored tangle diagram\footnote{In Ohtsuki \cite{ohtsuki02:_quant}, such diagrams are called \emph{sliced}.}, we denote by $\kup{D}$ the interpretation of $D$ as a morphism of $U_q(\sll_3)$-modules.

\begin{thm}[\cite{MR1036112}, see as well~\cite{ohtsuki02:_quant}]
  The map $\kup{D}$ depends only on the framed isotopy type of $D$. Furthermore, if $L$ is a framed link, $q^s\kup{L}$ is a Laurent polynomial (where $s$ is a certain\footnote{This number is computable, but we do not need its expression.} real number depending on $L$).
\end{thm}

In the sequence, we refer to this invariant to the \emph{colored $\sll_3$-invariant}.


\begin{prop} \label{prop:directsum}
  The colored $\sll_3$-invariant is multilinear with respect to direct sum of representations:
    Let $T_1$, $T_2$ and $T_3$ be the same colored oriented tangles except for the color of their $i$th connected components: it is colored with $V$ in $T_1$, with $W$ in $T_2$ and with $V\oplus W$ in $T_3$. Then we have: $\kup{T_1} + \kup{T_2} = \kup{T_3}$. 
 \end{prop}

Note that remark~\ref{rmk:tensorcable} implies that $\kup{\cdot}$ relates tensor products and cablings. 

When a tangle or link diagram has no color, all components are meant to be colored with $V_+$ and we use the identifications of the duals of $V_+$ and $V_-$ with $V_-$ and $V_+$ given by proposition~\ref{prop:easyrelUsl}. This means in particular that we have:
\[
\begin{tikzpicture}[scale = 0.72]
  \begin{scope}[decoration={markings, mark=at
     position 0.5 with {\arrow{>}}},postaction={decorate}]
   \begin{scope}
     \draw[postaction={decorate}] (-1,0) .. controls +(0, -1) and +(0,-1) .. (1, 0);
   \end{scope} 
\node at (2, -0.5) {$= b_{-+},$};
\begin{scope}[xshift=4.5cm]
     \draw[postaction={decorate}] (+1,0) .. controls +(0, -1) and +(0,-1) .. (-1, 0);
   \end{scope}
\node at (6.5, -0.5) {$= b_{+-},$};
\begin{scope}[xshift=9cm]
     \draw[postaction={decorate}] (-1,-1) .. controls +(0, +1) and +(0,+1) .. (+1, -1);
   \end{scope}
\node at (11, -0.5) {$=d_{+-}$};
\begin{scope}[xshift=13.5cm]
     \draw[postaction={decorate}] (+1,-1) .. controls +(0, +1) and +(0,+1) .. (-1, -1);
   \end{scope}
\node at (15.5, -0.5) {$=d_{-+}.$};
\end{scope}
\end{tikzpicture}
\]

We represent the morphisms $h_{++}^-$ and $h_{--}^+$ with the following two diagrams:
\[
\begin{tikzpicture}[scale = 0.72]
  \begin{scope}[decoration={markings, mark=at
     position 0.5 with {\arrow{>}}},postaction={decorate}]
   \begin{scope}
     \draw[postaction={decorate}] (0,1) -- (0,0);
     \draw[postaction={decorate}] (-1,-1) .. controls +(0, +0.5) and +(-0.50,-0.5) .. (0, 0);
     \draw[postaction={decorate}] (+1,-1) .. controls +(0, +0.5) and +(0.5,-0.5) .. (0, 0);
\node at (2, 0) {$= h_{++}^-,$};
   \end{scope} 
\begin{scope}[xshift=6cm]
     \draw[postaction={decorate}] (0,0) -- (0,1);
     \draw[postaction={decorate}] (0,0) .. controls +(-0.50, -0.5) and +(0,0.5) .. (-1, -1);
     \draw[postaction={decorate}] (0,0) .. controls +(0.5, -0.5) and +(0,0.5) .. (+1, -1);
\node at (2, 0) {$= h_{--}^+.$};
   \end{scope} 
\end{scope}

\end{tikzpicture}
\]
\begin{rmk}
  \label{rmk:embbededfine}
Using the diagrams associated with the $b$'s the $d$'s and the $h$'s, we can construct all planar bipartite graph $w$ (the bipartition being given by sources and sinks) embedded in $\RR\times[0,1]$ with degree of degree 1 or 3 such that the intersection of $G$ with $\RR\times \{0,1\}$ is precisely the set of vertices of degree 1. To such an embedded graph one can associate a morphism of $\Usl$-modules. Proposition~\ref{prop:easyrelUsl} says that the associated morphisms depend only on the planar isotopy type of graphs.
  
\end{rmk}

\begin{prop}[\cite{MR1403861} (we have $-q_{\textrm{Kup}}^{\frac16} = q_{\textrm{Here}}^{\frac13}$)]
One can fix a normalization of the universal $R$-matrix such that the following relations hold:
  \begin{align*}
    & \begin{tikzpicture}[scale = 0.72]
  \begin{scope}[decoration={markings, mark=at
     position 0.5 with {\arrow{>}}},postaction={decorate}]
   \begin{scope}
     \draw[->] (-1,-1) .. controls +(0, 1) and +(0,-1) .. (1, 1);
     \fill [white] (0,0) circle (0.2);
     \draw[->] (1,-1) .. controls +(0, 1) and +(0,-1) .. (-1, 1);
   \end{scope} 
\node at (2, 0) {$= q^{-\frac23}\cdot$};
   \begin{scope}[xshift=4cm]
     \draw[->] (-1,-1) .. controls +(0.3, 1) and +(0.3,-1) .. (-1, 1);
     \fill [white] (0,0) circle (0.2);
     \draw[->] (1,-1) .. controls +(-0.3, 1) and +(-0.3,-1) .. (1, 1);{$$};
   \end{scope} 
\node at (6,0) {$- q^{\frac13}\cdot$};
   \begin{scope}[xshift = 8cm]
     \draw[postaction={decorate}] (1,-1) .. controls +(0, 0.5) and +(0.50,-0.5) .. (0, -0.3);
     \draw[postaction={decorate}] (-1,-1) .. controls +(0, 0.5) and +(-0.50,-0.5) .. (0, -0.3);
     \draw[postaction={decorate}] (0,0.3) .. controls +(0.5, 0.5) and +(0,-0.5) .. (1, 1);
     \draw[postaction={decorate}] (0,0.3) .. controls +(-0.5, 0.5) and +(0,-0.5) .. (-1,1);
     \draw[postaction={decorate}] (0,0.3) -- (0, -0.3);
   \end{scope} 
\end{scope}
\end{tikzpicture} \\
    & \begin{tikzpicture}[scale = 0.72]
  \begin{scope}[decoration={markings, mark=at
     position 0.5 with {\arrow{>}}},postaction={decorate}]
   \begin{scope}
     \draw[->] (1,-1) .. controls +(0, 1) and +(0,-1) .. (-1, 1);
     \fill [white] (0,0) circle (0.2);
     \draw[->] (-1,-1) .. controls +(0, 1) and +(0,-1) .. (1, 1);
   \end{scope} 
\node at (2, 0) {$= q^{\frac23}\cdot$};
   \begin{scope}[xshift=4cm]
     \draw[->] (-1,-1) .. controls +(0.3, 1) and +(0.3,-1) .. (-1, 1);
     \fill [white] (0,0) circle (0.2);
     \draw[->] (1,-1) .. controls +(-0.3, 1) and +(-0.3,-1) .. (1, 1);{$$};
   \end{scope} 
\node at (6,0) {$- q^{-\frac13}\cdot$};
   \begin{scope}[xshift = 8cm]
     \draw[postaction={decorate}] (1,-1) .. controls +(0, 0.5) and +(0.50,-0.5) .. (0, -0.3);
     \draw[postaction={decorate}] (-1,-1) .. controls +(0, 0.5) and +(-0.50,-0.5) .. (0, -0.3);
     \draw[postaction={decorate}] (0,0.3) .. controls +(0.5, 0.5) and +(0,-0.5) .. (1, 1);
     \draw[postaction={decorate}] (0,0.3) .. controls +(-0.5, 0.5) and +(0,-0.5) .. (-1,1);
     \draw[postaction={decorate}] (0,0.3) -- (0, -0.3);
   \end{scope} 
\end{scope}
\end{tikzpicture}
  \end{align*}
\end{prop}

We want to relate the colored $\sll_3$-invariant with the classical $\sll_3$-invariant (\ie the invariant for the representation $V_+$). 
\begin{notation}
  Let $k$ be a non-negative integer and  $\mathbf{a}$, $\mathbf{b}$ and $\mathbf{c}$ three $k$-tuples of integers. We define:
  \[\trinomial{\mathbf{a}+ \mathbf{b} + \mathbf{c} } {\mathbf{a}}{\mathbf{b}}{\mathbf{c}}
  =\prod_{i=1}^k \trinomial{a_i+ b_i + c_i}{a_i}{b_i}{c_i}.
\]
Similarly if $\mathbf{n}$, $\mathbf{a}$ and $\mathbf{b}$ are three $k$-tuples of integers, we define:
  \[\trinomial{\mathbf{n}}{\mathbf{a}}{\mathbf{b}}{\blacksquare}
  =\trinomial{\mathbf{n}}{\mathbf{a}}{\mathbf{b}}{\mathbf{n}-\mathbf{a}-\mathbf{b}}.
\]
If $\mathbf{a}$ is a $k$-tuple of integers, then we define $|\mathbf{a}|$ to be the sum of all its coordinate.
  If $L$ is a framed link with $k$ components, and if $(\mathbf{a},\mathbf{b})$ is a pair of $k$-tuples of integers, the link $L_{\mathbf{a}, \mathbf{b}}$ is obtained by cabling each component $l_i$ ($i\in\{1,k\}$) by $a_i+b_i$ strands, the first (leftmost) $a_i$ strands  orientated like $l_i$ and the last (rightmost) $b_i$ strands with the opposite orientation.
\end{notation}
 From remark~\ref{rmk:tensorcable}, proposition~\ref{prop:directsum} and corollary~\ref{cor:explicitformula}, we have:
 \begin{prop}\label{prop:formulakup} Let $L$ be an oriented framed link with $s$ components, and let $(\mathbf{m}, \mathbf{n})$ be a pair of $s$-tuples of non-negative integers. We denote by $(L, V_{\mathbf{m}, \mathbf{n}})$ the colored framed oriented link obtained by coloring each component $l_t$ of $L$ with $V_{m_t, n_t}$. We have the following equality:
   \begin{align*}
     \kup{(L, V_{\mathbf{m}, \mathbf{n}})} &= \sum_{\substack{(\mathbf{i},\mathbf{j},\mathbf{k},\mathbf{l}) \in (\NN^s)^4 \\ \pmb{\delta} \in \{0,1\}^s }} (-1)^{|\pmb{\delta}+\mathbf{i}+\mathbf{k}|}
\trinomial{\mathbf{m}-\pmb{\delta}-\mathbf{i} -2\mathbf{j}}{\mathbf{i}}{\mathbf{j}}{\blacksquare} 
\cdot\trinomial{\mathbf{n}-\pmb{\delta}-\mathbf{k} -2\mathbf{l}}{\mathbf{k}}{\mathbf{l}}{\blacksquare}  \\ &\hspace{3cm}
\cdot\kup{L_{\mathbf{m}-2\mathbf{i}+\mathbf{k}-3\mathbf{j}-\pmb{\delta}, \mathbf{n}+\mathbf{i}-2\mathbf{k}-3\mathbf{l}-\pmb{\delta}}}.\\
   \end{align*}
 \end{prop}
In this paper we give a categorification of this formula.


\subsection{Framed $\sll_3$-homology in a nutshell}
\label{sec:sll_3-homology}

In this section we recall briefly the $\sll_3$-homology introduced by \cite{MR2100691} (see  as well \cite{MR2457839}, \cite{MR2336253}, \cite{MR2443231}, \cite{MR2482322}, \cite{LewarkThese}, \cite{MR3248745}, \cite{LHR2} and \cite{LHRThese}). We adapt it to the context of framed links. 

\subsubsection{Webs and foams}
\label{sec:webs-foams}

\begin{dfn}
  \label{dfn:web}
  A \emph{closed web} $w$ is a planar, oriented, 3-regular multi-graph with eventually some oriented vertex-less loop components. The orientation satisfies that each vertex is either a source or a sink (this yields a bi-partition of $w$).
\end{dfn}
\begin{figure}[ht]
  \centering
  \begin{tikzpicture}[scale=0.7]
    \begin{scope}
   [yscale = {1}, xscale={1},decoration={markings, mark=at
     position 0.5 with {\arrow{>}}},postaction={decorate}]
\coordinate (A) at (0,0);
\coordinate (B) at (1,0);
\coordinate (C) at (2,-0.5);
\coordinate (D) at (3,0);
\coordinate (E) at (4,0);
\coordinate (F) at (5,0.5);
\coordinate (A1) at (0,1);
\coordinate (B1) at (1,1);
\coordinate (C1) at (2,1.5);
\coordinate (D1) at (3,1);
\coordinate (E1) at (4,1);
\coordinate (F1) at (4,2);
\coordinate (G) at (5,-0.5);
\coordinate (H) at (3.5,-1);

\draw[postaction=decorate] (A1) --(A);
\draw[postaction=decorate] (A1) .. controls +(-0.5,0)  and  +(-0.5,0).. (A);
\draw[postaction=decorate] (A1) -- (B1);
\draw[postaction=decorate] (B)-- (A);
\draw[postaction=decorate] (B)--(B1);
\draw[postaction=decorate] (B)--(C);
\draw[postaction=decorate] (C1)--(B1);
\draw[postaction=decorate] (C1)--(D1);
\draw[postaction=decorate] (C1)--(F1);
\draw[postaction=decorate] (D)--(C);
\draw[postaction=decorate] (D)--(D1);
\draw[postaction=decorate] (D)--(E);
\draw[postaction=decorate] (F) -- (G);
\draw[postaction=decorate] (F)-- (E);
\draw[postaction=decorate] (F) .. controls +(0,0.5) and  +(0.4,0.4) .. (F1);
\draw[postaction=decorate] (E1)--(F1);
\draw[postaction=decorate] (E1)--(D1);
\draw[postaction=decorate] (E1)--(E);
\draw[postaction=decorate] (H)--(C);
\draw[postaction=decorate] (H) .. controls +(0.5,0) and  +(0,-0.5) .. (G);
\draw[postaction=decorate] (H) .. controls +(0.2,0.3) and  +(-0.4,-0.1) .. (G);
\draw[postaction=decorate] (7,1) circle (0.5cm);
\end{scope}
  \end{tikzpicture}
  \caption{Example of a web}
  \label{fig:exweb}
\end{figure}
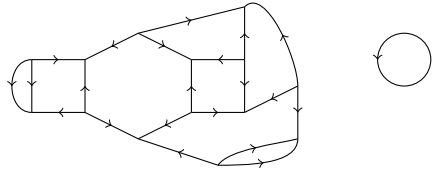

\begin{rmk}
  \label{rmk:web2sl3}
  Thanks to remark~\ref{rmk:embbededfine} any closed web may be interpreted as an endomorphism $\kup{w}$ of the $\Usl$-module $V_{0,0}$ which is simply an element of $\CC\left(q^{\frac16}\right)$. One can check that the following local relations hold:
\begin{align*}
   \kup{\websquare[0.4]} &= \kup{\webtwovert[0.4]} + \kup{\webtwohori[0.4]}, \\
   \kup{\webbigon[0.4]\,}  &= [2] \cdot \kup{\webvert[0.4]\,},\\
   \kup{\webcircle[0.4]} &= \kup{\webcirclereverse[0.4]} = [3],
  \end{align*}
where $[n]$ is by definition equal to $\frac{q^n-q^{-n}}{q- q\m}$.

The planarity of closed webs imposes that that every non-empty web contains at least one vertex-less loop, one digon or one square.
This shows that $\kup{w}$ can be computed completely combinatorialy  and that it is a Laurent polynomial symmetric in $q$ and $q\m$ with non-negative coefficients.
\end{rmk}

\begin{dfn}
  A \emph{pre-foam} is a smooth, oriented, compact surface $\Sigma$ (its connected components are called \emph{facets}) together with the following data~:
\begin{itemize}
\item A partition of the connected components of the boundary into cyclically ordered 3-sets and for each 3-set $(C_1,C_2,C_3)$, three orientation-preserving diffeomorphisms $\phi_1:C_2\to C_3$, $\phi_2:C_3\to C_1$ and $\phi_3:C_1\to C_2$ such that $\phi_3 \circ \phi_2 \circ \phi_1 = \mathrm{id}_{C_2}$.
\item A function from the set of facets to the set of non-negative integers (this gives the number of \emph{dots} on each facet).
\end{itemize}
The \emph{CW-complex associated with a pre-foam} is the 2-dimensional CW-complex $\Sigma$ quotiented by the diffeomorphisms so that the three circles of one 3-set are identified and become just one, called a \emph{singular circle}.
The \emph{degree} of a pre-foam $f$ is equal to $-2\chi(\Sigma')$ where $\chi$ is the Euler characteristic and $\Sigma'$ is the CW-complex associated with $f$ with the dots punctured out (a dot increases the degree by 2).
\end{dfn}
\begin{rmk}
  A CW-complex associated with a pre-foam has two local models. 
If a point $x$ is not on a singular circle, then it has a neighborhood diffeomorphic to a 2-dimensional disk, else it has a neighborhood diffeomorphic to a Y shape times an interval (see figure \ref{fig:yshape}).
  \begin{figure}[h]
    \centering
    \begin{tikzpicture}[scale=1]
      \input{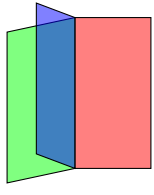}
    \end{tikzpicture}
    \caption{Singularities of a pre-foam}
    \label{fig:yshape}
  \end{figure}
\end{rmk}
\begin{dfn}
  A \emph{closed foam} is the image of an embedding of the CW-complex associated with a pre-foam such that the cyclic orders of the pre-foam are compatible with the left-hand rule in $\RR^3$ with respect to the orientations of the singular circles\footnote{We mean here that if, next to a singular circle, with the forefinger of the left hand we go from face 1 to face 2 to face 3 the thumb points to indicate the orientation of the singular circle (induced by orientations of facets). This is not quite canonical, physicists use more the right-hand rule, however this is the convention used in \cite{MR2100691}.}. The \emph{degree} of a closed foam is the degree of the underlying pre-foam. 
\end{dfn}

\begin{dfn}\label{dfn:wfoam}
  If $w$ is a closed web, a \emph{$w$-foam} $f$ is the intersection of a foam $f'$ with $\RR^2\times \RR_+$ such that there exits $\epsilon>0$, such that:
 \[\RR^2\times ]-\epsilon, 0]) \cap f = w\times ]-\epsilon, 0].\]  
The \emph{degree} of a $w$-foam $f$ is equal to $\chi(w)-2\chi(\Sigma)$ where $\Sigma$ is the underlying CW-complex associated with $f$ with the dots punctured out. 
\end{dfn}

\begin{propdfn}[\cite{MR2100691}]
  \label{pd:functor}
Let $w$ be a web. We define $\F(w)$ to be the $\ZZ$-module generated by isotopy classes (fixing $\RR^2 \times \{0\}$) of $w$-foams and subjected to the local relations:
\begin{align*}
  &\vcenter{\hbox{\begin{tikzpicture}[scale=0.3]
    \begin{scope}[xshift=0cm, yshift= 0cm, decoration={markings, mark=at
     position 0.5 with {\arrow{>}}},postaction={decorate}]
\draw (-2,0) -- +(1,0);
\draw (1,0) -- +(1,0);
\draw[dashed] (-1,0) .. controls +(1,0.5) and +(-1,0.5).. +(2,0);
\draw (-1,0) .. controls +(1,-0.5) and +(-1,-0.5).. +(2,0);
\draw (-2,3) -- +(1,0);
\draw (1,3) -- +(1,0);
\draw (-1,3) .. controls +(1,0.5) and +(-1,0.5).. +(2,0);
\draw (-1,3) .. controls +(1,-0.5) and +(-1,-0.5).. +(2,0);
\draw (-2,0) -- +(0,3);
\draw[postaction={decorate}] (-1,0) -- +(0,3);
\draw[postaction={decorate}] (1,3) -- +(0,-3);
\draw (2,0) -- +(0,3);
\end{scope}
\node at (3,1.5) {$=$};
\begin{scope}[xshift=6cm, yshift= 0cm, decoration={markings, mark=at
     position 0.5 with {\arrow{>}}},postaction={decorate}]
\draw (-2,0) -- +(1,0);
\draw (1,0) -- +(1,0);
\draw[dashed] (-1,0) ..controls +(1,0.5) and +(-1,0.5).. +(2,0);
\draw (-1,0) ..controls +(1,-0.5) and +(-1,-0.5).. +(2,0);
\draw (-2,3) -- +(1,0);
\draw (1,3) -- +(1,0);
\draw (-1,3) ..controls +(1,0.5) and +(-1,0.5).. +(2,0);
\draw (-1,3) ..controls +(1,-0.5) and +(-1,-0.5).. +(2,0);
\draw (-2,0) -- +(0,3);
\draw[postaction={decorate}] (-1,0) .. controls +(0,1.5) and +(0,1.5).. +(2,0);
\draw[postaction={decorate}] (1,3) .. controls +(0,-1.5) and +(0,-1.5).. +(-2,0);
\draw (2,0) -- +(0,3);
\fill (0,3) circle (4pt and 4pt);
\end{scope}
\node at (9,1.5) {$-$}; 
\begin{scope}[xshift=12cm, yshift= 0cm, decoration={markings, mark=at
     position 0.5 with {\arrow{>}}},postaction={decorate}]
\draw (-2,0) -- +(1,0);
\draw (1,0) -- +(1,0);
\draw[dashed] (-1,0).. controls +(1,0.5) and +(-1,0.5).. +(2,0);
\draw (-1,0) ..controls +(1,-0.5) and +(-1,-0.5).. +(2,0);
\draw (-2,3) -- +(1,0);
\draw (1,3) -- +(1,0);
\draw (-1,3).. controls +(1,0.5) and +(-1,0.5).. +(2,0);
\draw (-1,3).. controls +(1,-0.5) and +(-1,-0.5).. +(2,0);
\draw (-2,0) -- +(0,3);
\draw[postaction={decorate}] (-1,0) .. controls +(0,1.5) and +(0,1.5).. +(2,0);
\draw[postaction={decorate}] (1,3) .. controls +(0,-1.5) and +(0,-1.5).. +(-2,0);
\draw (2,0) -- +(0,3);
\fill (0,0) circle (4pt and 4pt);
\end{scope}
  \end{tikzpicture} }} & (\textrm{digon relation}) \\
  &\vcenter{\hbox{\begin{tikzpicture}[scale=0.3]
    \begin{scope}[xshift=0cm, yshift= 0cm]
\draw (0,0) -- (2,0);
\draw (2,0) -- (1,1);
\draw (1,1) -- (-1,1); 
\draw (-1,1) -- (0,0);
\draw (0,0) -- +(-0.5,-0.5);
\draw (2,0) -- +(0.7,-0.3);
\draw (1,1) -- +(0.5,0.5); 
\draw (-1,1) -- +(-0.7,0.3);
\draw (0,-4) -- (2,-4);
\draw[dashed] (2,-4) -- (1,-3);
\draw[dashed] (1,-3) -- (-1,-3); 
\draw (-1,-3) -- (-0.5,-3.5);
\draw[dashed] (-0.5,-3.5) -- (0,-4);
\draw (0,-4) -- ++(-0.5,-0.5) -- +(0,4);
\draw (2,-4) -- ++(0.7,-0.3)--+(0,4);
\draw[dashed] (1,-3) -- ++(0.5,0.5)--+(0,3);
\draw (1.5,0.5) -- (1.5,1.5); 
\draw (-1,-3) -- ++(-0.7,0.3)--+(0,4);
\draw (0,-4) -- +(0,4);
\draw (2,-4) -- +(0,4);
\draw[dashed] (1,-3) -- +(0,3);
\draw (1,0)-- +(0,1);
\draw (-1,-3) -- +(0,4);
\end{scope}
\node at (3.9,-2) {$=-$};
\begin{scope}[xshift=7cm, yshift= 0cm, decoration={markings, mark=at
     position 0.5 with {\arrow{>}}},postaction={decorate}]
\fill[gray, opacity =0.5] (-1,1) -- (0,0) arc (180:225:1) -- +(-1,1) arc (225:180:1)-- cycle;
\fill[gray,opacity =0.5] (1,1) -- (2,0) arc (0:-135:1) -- +(-1,1) arc(-135:0:1) --cycle;
\fill[gray,opacity=0.5] (-1,-3) -- (0,-4) arc (180:45:1) -- +(-1,1) arc (45:180:1)-- cycle;
\fill[gray,opacity=0.5] (1,-3) -- (2,-4) arc (0:45:1) -- +(-1,1) arc(45:0:1) --cycle;
\draw (0,0) -- (2,0);
\draw (2,0) -- (1,1);
\draw (1,1) -- (-1,1); 
\draw (-1,1) -- (0,0);
\draw (0,0) -- +(-0.5,-0.5);
\draw (2,0) -- +(0.7,-0.3);
\draw (1,1) -- +(0.5,0.5); 
\draw (-1,1) -- +(-0.7,0.3);
\draw (0,-4) -- (2,-4);
\draw[dashed] (2,-4) -- (1,-3);
\draw[dashed] (1,-3) -- (-1,-3); 
\draw (-1,-3) -- (-0.5,-3.5);
\draw[dashed] (-0.5,-3.5) -- (0,-4);
\draw (0,-4) -- ++(-0.5,-0.5) -- +(0,4);
\draw (2,-4) -- ++(0.7,-0.3)--+(0,4);
\draw[dashed] (1,-3) -- ++(0.15,0.15);
\draw (1,-3) + (0.15,0.15) --++(0.5,0.5) -- +(0,1.7);
\draw[dashed] (1.5,0.5) -- +(0,-1.3);
\draw (1.5,0.5) -- (1.5,1.5); 
\draw (-1,-3) -- ++(-0.7,0.3)--+(0,4);
\draw (0,-4) arc (180:0:1);
\draw (0,0) arc (-180:0:1);
\draw[dashed] (1,1) arc (0:-180:1);
\draw (-1,-3) arc (180:120:1);
\draw[dashed] (1,-3) arc (0:120:1);
\draw (0,-3)++(45:1) -- +(1,-1);
\draw (0,1)++(-135:1) -- +(1,-1);
\end{scope}
\node at (10.9,-2) {$-$};
\begin{scope}[xshift=14cm, yshift= 0cm, decoration={markings, mark=at
     position 0.5 with {\arrow{>}}},postaction={decorate}]
\fill[gray,opacity=0.5] (-1,1) ..controls +(0,-1.5) and +(-0.1,0).. (-0.5,-0.8) -- ++(2,0).. controls +(-0.1,0) and +(0,-1.5) .. (1,1);
\fill[gray,opacity=0.5] (0,0) ..controls +(0,-0.8) and +(0.1,0).. (-0.5,-0.8) -- ++(2,0).. controls +(0.1,0) and +(0,-0.8) .. (2,0);
\fill[gray,opacity=0.5] (-1,-3) ..controls +(0,0.8) and +(-0.1,0).. ++(0.5,0.8) -- ++(2,0).. controls +(-0.1,0) and +(0,0.8) .. (1,-3);
\fill[gray,opacity=0.5] (0,-4) ..controls +(0,1.5) and +(0.1,0).. (-0.5,-2.2) -- ++(2,0).. controls +(0.1,0) and +(0,1.5) .. (2,-4);
\draw (0,0) -- (2,0);
\draw (2,0) -- (1,1);
\draw (1,1) -- (-1,1); 
\draw (-1,1) -- (0,0);
\draw (0,0) -- +(-0.5,-0.5);
\draw (2,0) -- +(0.7,-0.3);
\draw (1,1) -- +(0.5,0.5); 
\draw (-1,1) -- +(-0.7,0.3);
\draw (0,-4) -- (2,-4);
\draw[dashed] (2,-4) -- (1,-3);
\draw[dashed] (1,-3) -- (-1,-3); 
\draw (-1,-3) -- (-0.5,-3.5);
\draw[dashed] (-0.5,-3.5) -- (0,-4);
\draw (0,-4) -- ++(-0.5,-0.5) -- +(0,4);
\draw (2,-4) -- ++(0.7,-0.3)--+(0,4);
\draw[dashed] (1,-3) -- ++(0.5,0.5)--+(0,0.3);
\draw (1.5,-2.2) -- (1.5,-0.8);
\draw[dashed] (1.5,-0.8) -- (1.5,0.5);
\draw (1.5,0.5)-- (1.5,1.5 );
\draw (1.5,0.5) -- (1.5,1.5); 
\draw (-1,-3) -- ++(-0.7,0.3)--+(0,4);
\draw[dashed] (1,-3) ..controls +(0,0.8) and +(-0.1,0) .. ++(0.5,0.8).. controls +(0.1,0) and +(0,1.5) .. (2,-4);
\draw (-1,-3) ..controls +(0,0.8) and +(-0.1,0) .. ++(0.5,0.8).. controls +(0.1,0) and +(0,1.5) .. (0,-4);
\draw (1,1) ..controls +(0,-1.5)  and +(-0.1,0) .. (1.5,-0.8).. controls +(0.1,0) and +(0,-0.8) .. (2,0);
\draw (-1,1) ..controls +(0,-1.5) and +(-0.1,0) .. (-0.5,-.8).. controls +(0.1,0) and +(0,-0.8) .. (0,0);
\draw (-0.5,-0.8)-- +(2,0);
\draw (-0.5,-2.2)-- +(2,0);
\end{scope}
  \end{tikzpicture}}}& (\textrm{square relation})\\
  &\vcenter{\hbox{\begin{tikzpicture}[scale=0.4]
    \begin{scope}[xshift=-0.5cm]
       \draw (0,-1) arc (270:90:0.5 and 1);
      \draw[dotted] (0,-1) arc (-90:90:0.5 and 1);
      \draw (3,0) ellipse (0.5 and 1);
      \draw (0,-1) -- (3,-1);
      \draw (0,1) -- (3,1);
\end{scope}
\node at (4.5,0) {$=-$};
\node at (10,0) {$-$};
\node at (15,0) {$-$};
\begin{scope}[xshift=6cm]
       \draw (0,-1) arc (270:90:0.5 and 1);
      \draw[dotted] (0,-1) arc (-90:90:0.5 and 1);
      \draw (3,0) ellipse (0.5 and 1);
      \draw (0,-1) .. controls +(1.5,0) and +(1.5,0) .. (0,1);
      \draw (3,-1) .. controls +(-1.5,0) and +(-1.5,0) .. (3,1);
      \filldraw (0.7,0.2) ellipse (2pt and 4pt);
      \filldraw (0.7,-0.2) ellipse (2pt and 4pt);
\end{scope}
\begin{scope}[xshift=11cm]
       \draw (0,-1) arc (270:90:0.5 and 1);
      \draw[dotted] (0,-1) arc (-90:90:0.5 and 1);
      \draw (3,0) ellipse (0.5 and 1);
      \draw (0,-1) .. controls +(1.5,0) and +(1.5,0) .. (0,1);
      \draw (3,-1) .. controls +(-1.5,0) and +(-1.5,0) .. (3,1);
      \filldraw (0.7,0) ellipse (2pt and 4pt);
      \filldraw (2.2,0) ellipse (2pt and 4pt);
\end{scope}
\begin{scope}[xshift=16cm]
       \draw (0,-1) arc (270:90:0.5 and 1);
      \draw[dotted] (0,-1) arc (-90:90:0.5 and 1);
      \draw (3,0) ellipse (0.5 and 1);
      \draw (0,-1) .. controls +(1.5,0) and +(1.5,0) .. (0,1);
      \draw (3,-1) .. controls +(-1.5,0) and +(-1.5,0) .. (3,1);
      \filldraw (2.2,-0.2) ellipse (2pt and 4pt);
      \filldraw (2.2,0.2) ellipse (2pt and 4pt);
\end{scope}

  \end{tikzpicture}}}\hspace{-2cm}& (\textrm{surgery})\\ &
\vcenter{\hbox{ \begin{tikzpicture}[scale=0.5]
     \begin{scope}
   [yscale = {1}, xscale={1},decoration={markings, mark=at
     position 0.5 with {\arrow{>}}},postaction={decorate}]
   \fill[color = gray] (0,0) ellipse (1 and 0.5);
  \draw (0,0) circle (1);
   \draw[postaction= decorate] (-1,0) arc (180:360:1 and 0.5);
   \draw[dotted] (-1,0) arc (180:0:1 and 0.5);
   \node (A) at (-2, 1.5) {$k \bullet$};
   \node (B) at (-2, 0.5) {$l \bullet$};
   \node (C) at (-2, -1.5) {$m \bullet$};
   \draw[very thin, ->]  (A) .. controls +(1, 0) and +(0,0.5) .. (0,1);
   \draw[very thin, ->]  (B) .. controls +(1, 0) and +(0,0.5) .. (0,0);
   \draw[very thin, ->]  (C) .. controls +(1, 0) and +(0,-0.5) .. (0,-1);
\end{scope}

  \end{tikzpicture}}}
=
   \begin{cases}
     1 &\textrm{if $(k,l,m) = (0,1,2), (1,2,0)$ or $(2,0,1)$, } \\
     -1 &\textrm{if $(k,l,m) = (0,2,1), (2,1,0)$ or $(1,0,2)$, } \\
     0 &\textrm{else.} \\
   \end{cases}\hspace{-6cm}
&     \begin{matrix}
      \\
       \\
      \textrm{(theta-foams evaluation)} \\
   \end{matrix}
\\ &
\vcenter{\hbox{
\begin{tikzpicture}[scale=0.5]
     \begin{scope}
   [xshift = 8cm, yscale = {1}, xscale={1},decoration={markings, mark=at
     position 0.5 with {\arrow{>}}},postaction={decorate}]
  \draw (0,0) circle (1);
   \draw[postaction= decorate] (-1,0) arc (180:360:1 and 0.5);
   \draw[dotted] (-1,0) arc (180:0:1 and 0.5);
   \node (C) at (-2, -1.5) {$k \bullet$};
 \draw[very thin, ->]  (C) .. controls +(1, 0) and +(0,-0.5) .. (0,-1);
 
\end{scope}

  \end{tikzpicture}
}}
  =
   \begin{cases}
     1 &\textrm{if $k=2$ } \\
         0 &\textrm{else. } \\
   \end{cases}& (\textrm{spheres evaluation})
\end{align*}
The $\ZZ$-module $\F(w)$ is free and its graded dimension is equal to $\kup{w}$. Furthermore, $\F$ extends to a functor from the category of foams (we believe that the definition of this category is clear from the context, see \cite{MR2100691} or \cite{LHRThese} for details) to the category of graded $\ZZ$-modules.
\end{propdfn}
\begin{rmk}
  \label{rmk:Frob}
  If we restrict $\F$ to collection of circles and surfaces, $\F$ becomes a 2-dimensional TQFT. The $\ZZ$-module associated to the circle has a natural structure of Frobenius algebra $\mathcal{A}$. It is isomorphic to $\ZZ[X]/X^3$.
\end{rmk}

\subsubsection{$\sll_3$-homology}
\label{sec:sll_3-homology-1}

\begin{dfn}
  \label{dfn:smoothings}
  Let $D$ be an oriented link diagram and $s$ a function from the set $X(D)$ of crossings of $D$ to $\{0,1\}$.
The \emph{$s$-smoothing of $D$}, denoted by $D^s$ is the web obtained by replacing each crossing of $D$ following the instructions:
\[
\begin{tikzpicture}[scale=0.7]
   \begin{scope}[yshift=2.75cm,scale=0.5, xshift = 1.5cm]
   \draw[->] (0.5,0.5) -- (2.5, 2.5);
   \draw[fill=white, color =white] (1.5,1.5) circle (3mm);
   \draw[dotted] (1.5,1.5) circle (1.414cm);
   \draw[->] (2.5,0.5) -- (0.5, 2.5);
   \node at (1.5, 3.6) {$x$ in $D$ is negative};
 \end{scope}
 \begin{scope}[yshift=-1.25cm, scale=0.5, xshift = 1.5cm,]
   \draw[->] (2.5,0.5) -- (0.5, 2.5);
   \draw[fill=white, color =white] (1.5,1.5) circle (3mm);
   \draw[dotted] (1.5,1.5) circle (1.414cm);
   \draw[->] (0.5,0.5) -- (2.5, 2.5);
   \node at (1.5, -0.6) {$x$ in $D$ is positive};
 \end{scope}
 \begin{scope}[yshift = 0.75cm, xshift=-5cm,decoration={markings, mark=at
     position 0.5 with {\arrow{>}}},postaction={decorate}, scale =0.5]
   \draw[dotted] (1.5,1.5) circle (1.414cm);
   \draw[->] (0.5,0.5) .. controls (1,1) and (1,2) .. (0.5, 2.5);
   \draw[->] (2.5,0.5) .. controls (2,1) and (2,2) .. (2.5, 2.5);
 \end{scope}
 \begin{scope}[yshift = 0.75cm, xshift = 6.5cm,decoration={markings, mark=at
     position 0.5 with {\arrow{>}}},postaction={decorate}, scale= 0.5]
   \draw[dotted] (1.5,1.5) circle (1.414cm);
      \draw[>-] (0.5,0.5) -- (1.5, 1);
      \draw[>-] (2.5,0.5) -- (1.5, 1);
      \draw[postaction={decorate}] (1.5,2) -- (1.5,1);
      \draw[->] (1.5,2)-- (0.5,2.5);
      \draw[->] (1.5,2)-- (2.5,2.5);
 \end{scope}
\draw[->] (0,3.5) -- (-3, 1.8) node [ midway, below, sloped] {$s(x)=0$};
\draw[->] (0,-0.5) -- (-3, 1.2) node [ midway, below, sloped] {$s(x)=1$};;
\draw[->] (3,3.5) -- (6, 1.7) node [ midway, below, sloped] {$s(x)=1$};;
\draw[->] (3,-0.5) -- (6, 1.3) node [ midway, below, sloped] {$s(x)=0$};;
\end{tikzpicture}
\]
Such a function $s$ is called a \emph{smoothing function for $D$}. We define $|s| = \#s\m(\{1\})$. Suppose additionally that $X(D)$ is endowed with a complete order.  If $x$ is a crossing, we set $k_x(s) = \#\{y\in X(D)\,|\, y<x \textrm{ and } s(y)=1 \}$.

If $x$ is a crossing of $D$ and $s$ and $s'$ two smoothing functions for $D$ such that $s(x)=0$ and $s'(x)=1$ and $s_{|X(D)\setminus\{x\}}=s'_{|X(D)\setminus\{x\}}$, we write $s\stackrel{x}{\To}s'$. We define a $(D^s, D^{s'})$-foam $f_{s\stackrel{x}{\To}s'}$ which is the identity foam outside a neighborhood of $x$ and around $x$ is given by the following pictures (this is to be read from bottom to top):
\[
\begin{tikzpicture}[scale= 0.5]
  \begin{scope}[scale =1, decoration={markings, mark=at
     position 0.5 with {\arrow{>}}},postaction={decorate}]
\begin{scope}[scale=0.5]
\coordinate (A2) at (1,5);
\coordinate (B2) at (4,3);
\coordinate (C2) at (-2,0);
\coordinate (D2) at (1,-2);
\coordinate (E2) at (0,-3);
\coordinate (F2) at (2,0);
\coordinate (A1) at (1,2);
\coordinate (B1) at (4,0);
\coordinate (C1) at (-2,-3);
\coordinate (D1) at (1,-5);
\coordinate (G2) at (1.5,1.8);
\node at (1, -6) {$f_{s\stackrel{x}{\longrightarrow}s'}$ when $x$ is positive.};
\draw (A1) -- (F2) -- (B1);
\draw (C1) -- (E2) -- (D1);
\draw (E2) -- (F2);
\draw (C2) .. controls +(1,1) and +(0,-1.5) .. (A2);
\draw (D2) .. controls +(0,1.5) and +(-1,-1) .. (B2);
\draw (A1) -- (A2);
\draw (B1) -- (B2);
\draw (C1) -- (C2);
\draw (D1) -- (D2);
\draw (E2) .. controls +(0,1) and  +(-0.4,0.3).. (G2).. controls +(0.4,-0.3) and +(0,0.5) .. (F2);
 \draw[densely dotted] (A2) .. controls +(0,0) and +( -0.5, +0.5) ..  (G2) .. controls +(0.5,-0.5) and +(0,0) .. (B2);
+(0,0).. (D2);
\fill[opacity=0.4, color= gray] (E2) .. controls +(0,1) and  +(-0.4,0.3).. (G2).. controls +(0.4,-0.3) and +(0,+0.5) .. (F2) --(E2);
\end{scope}
\end{scope}

\begin{scope}[scale =1, decoration={markings, mark=at
     position 0.5 with {\arrow{>}}},postaction={decorate}, xshift = 10cm]
\begin{scope}[xscale=-0.5, yscale=-0.5]
\coordinate (A2) at (1,5);
\coordinate (B2) at (4,3);
\coordinate (C2) at (-2,0);
\coordinate (D2) at (1,-2);
\coordinate (E2) at (0,-3);
\coordinate (F2) at (2,0);
\coordinate (A1) at (1,2);
\coordinate (B1) at (4,0);
\coordinate (C1) at (-2,-3);
\coordinate (D1) at (1,-5);
\coordinate (G2) at (1.5,1.8);
\draw (A1) -- (F2) -- (B1);
\draw (C1) -- (E2) -- (D1);
\draw (E2) -- (F2);
\draw (C2) .. controls +(1,1) and +(0,-1.5) .. (A2);
\draw (D2) .. controls +(0,1.5) and +(-1,-1) .. (B2);
\draw (A1) -- (A2);
\draw (B1) -- (B2);
\draw (C1) -- (C2);
\draw (D1) -- (D2);
\draw (E2) .. controls +(0,1) and  +(-0.4,0.3).. (G2).. controls +(0.4,-0.3) and +(0,0.5) .. (F2);
 \draw[densely dotted] (A2) .. controls +(0,0) and +( -0.5, +0.5) ..  (G2) .. controls +(0.5,-0.5) and +(0,0) .. (B2);
+(0,0).. (D2);
\fill[opacity=0.4, color= gray] (E2) .. controls +(0,1) and  +(-0.4,0.3).. (G2).. controls +(0.4,-0.3) and +(0,+0.5) .. (F2) --(E2);
\node at (-1, 6) {$f_{s\stackrel{x}{\longrightarrow}s'}$ when $x$ is negative.};
\end{scope}
\end{scope}
\end{tikzpicture}
\]
One can show that $\F(f_{s\stackrel{x}{\To}s'}): \F(D^s)\To\F(D^{s'})$ is an homogeneous map of degree $1$.
\end{dfn}

We are now ready to construct the $\sll_3$-homology for framed links:
\begin{dfn}
  \label{dfn:complexframed}
  Let $D$ be a link diagram with $n_+$ positive crossings and $n_-$ negative crossings. We fix a total order on $X(D)$. We define the complex $C(D)$ of graded $\ZZ$-module by:
\begin{center}  
\begin{align*}
    &C_i = \bigoplus_{\substack{\textrm{$s$ smoothing function for $D$} \\ i= |s| -\frac{1}{3}n_+-\frac{2}{3}n_-}} \F(D^s)q^{|s|-\frac{1}{3}n_+-\frac{2}{3}n_- }, \\
    &d_i: C_i \To C_{i+1}\quad \textrm{and} \quad d_i= \sum_{\substack{\textrm{$s$ smoothing function for $D$} \\i =  |s| -\frac{1}{3}n_+-\frac{2}{3}n_-}}\quad \sum _{s\stackrel{x}{\To}s'} (-1)^{k_x(s)}\F(f_{s\stackrel{x}{\To}s'}),
  \end{align*}
\end{center}
where $q^\bullet$ is the grading shift for the non-homological grading (the so called \emph{$q$-degree}).
\end{dfn}

\begin{rmk}
  \label{rmk:homdeginQ}
Two different orders on $X(D)$ yield isomorphic complexes.

  The homological grading of $C(D)$ is not in $\ZZ$ as usual but in $\QQ$ (the differential has degree 1). However, only finitely many objects $C_i(D)$ are not trivial. All the notions of chain maps, homotopy equivalence etc. remain exactly the same as in the classical case. The only thing which might need precision is the graded Euler characteristic. If $C= \bigoplus_{r\in \QQ} C_r$ is a such a complex we define:
\[
\chi_q(C)= \sum_{r\in \QQ} e^{i\pi r} \rk_q C_r.
\]
\end{rmk}

\begin{thm}[\cite{MR2100691}]
  \label{thm:sl3hom}
  The homotopy type of $C(D)$ depends only on the framed isotopy type of $D$.
\end{thm}

In the sequence the homology of $C(D)$ is denote by $H(D)$ or $H(L)$ if $L$ an oriented framed link represented by $D$.

\begin{thm}[\cite{MR2482322}]
  \label{thm:clarkeasy}
  $C$ extends to a functor from the category of framed link cobordisms to the category of complexes of graded $\ZZ$-modules up to homotopy.
\end{thm}
\begin{rmk}
  \label{rmk:extension2web}
  In his proof of theorem~\ref{thm:clarkeasy}, Clark actually shows\footnote{He actually shows that the original definition of $C(D)$ does not extend to knotted webs, because of some grading shifts that he carefully computes, but our normalization corrects this defects.} that one can canonically extends this construction to framed knotted webs. However he does not prove that this extension yields a functor on the category of framed knotted webs.  
\end{rmk}


\section{Tensor resolutions of simple finite dimensional $\Usl$-modules}
\label{sec:tens-resol-os}
In this section, we  define a graph $\Gamma_{m,n}$ associated with every pair of non-negative integers. In a second part, we ``interpret'' this graph to define a complex. Finally, we prove that is a resolution of the $\Usl$-module $V_{m,n}$.
\begin{dfn}\label{dfn:partition}
  Let $m$ and $n$ be two non-negative integers. We denote by $\eps(m,n)$ the sequence of $(m+n)$ dots, where the first $m$ dots are colored with $+$ and the last $n$ dots are colored with $-$. A \emph{compatible partition of $\eps(m,n)$} is a partition $\bigsqcup_i M_i$ of $\eps(m,n)$ such that:
  \begin{itemize}
  \item Each $M_i$ consists of consecutive dots; 
  \item Each $M_i$ is non empty and has cardinality at most three;
  \item If $M_i$ has three elements, all its elements have the same color ($+$ or $-$).
  \end{itemize}
The \emph{canonical partition of $\eps_{m,n}$} is the partition by singletons. 
The oriented graph $\Gamma_{m,n}$ is defined by:
\begin{itemize}
\item The set $V(\Gamma_{m,n})$ of vertices of $\Gamma_{m,n}$ consists of  all compatible partitions of $\eps(m,n)$.
\item There is a directed edge from $P_1$ to $P_2$ if and only if $P_2$ is obtained from $P_1$ by merging two sets of $P_1$.
\end{itemize}
If $(\mathbf{m}, \mathbf{n})$ is a pair of $k$-tuples of integers, $\Gamma_{\mathbf{m}, \mathbf{n}}$ is the Cartesian product of the graphs $\Gamma_{m_i, n_i}$ for $i =1, \dots k$.

\end{dfn}

\begin{exa}\label{exa:graph32}
  It is convenient to represent the compatible partitions by connecting dots belonging to the same set. The following picture represents the graph $\Gamma_{3,2}$. 
\[ 
\begin{tikzpicture}
  \begin{scope}[yscale=-1]
\newcommand{\dop}[1]{\fill[red] #1 circle (0.2); \node[white] at #1 {$\pmb{+}$};}
\newcommand{\dom}[1]{\fill[blue!70!white] #1 circle (0.2); \node[white] at #1 {$\pmb{-}$};}
\newcommand{\design}{ 
\begin{scope}[opacity =1]
\dop{(0,0)}
\dop{(1,0)}
\dop{(2,0)}
\dom{(3,0)}
\dom{(4,0)}
\draw[densely dotted, rounded corners] (-0.3, 0.3) -- (4.3, 0.3) -- (4.3, -0.3) -- (-0.3, -0.3) -- cycle;
\end{scope}
}
    \begin{scope}[scale=0.6]
      \node[opacity=0] (A00) at (2, 0) {$WWWW$};
      \design
    \end{scope}
    \begin{scope}[scale=0.6, yshift = -3.6cm, xshift= -6cm]
      \draw [very thick, gray] (0,0) -- +(1,0);
      \node[opacity=0] (A10) at (2, 0) {$WWWW$};
      \design
    \end{scope}
    \begin{scope}[scale=0.6, yshift = -2.9cm, xshift= -2cm]
      \draw [very thick, gray] (1,0) -- +(1,0);
      \node[opacity=0] (A11) at (2, 0) {$WWWW$};
      \design
    \end{scope}
    \begin{scope}[scale=0.6, yshift = -3.6cm, xshift= 2cm]
      \draw [very thick, gray] (2,0) -- +(1,0);
      \node[opacity=0] (A12) at (2, 0) {$WWWW$};
      \design
    \end{scope}
    \begin{scope}[scale=0.6, yshift = -2.9cm, xshift= 6cm]
      \draw [very thick, gray] (3,0) -- +(1,0);
      \node[opacity=0] (A13) at (2, 0) {$WWWW$};
      \design
    \end{scope}
\draw[->] (A00) -- (A10);
\draw[->] (A00) -- (A11);
\draw[->] (A00) -- (A12);
\draw[->] (A00) -- (A13);
    \begin{scope}[scale=0.6, yshift = -6.4cm, xshift= -6cm]
      \draw [very thick, gray] (0,0) -- +(1,0);
      \draw [very thick, gray] (1,0) -- +(1,0);
      \node[opacity=0] (A20) at (2, 0) {$WWWW$};
      \design
    \end{scope}
    \begin{scope}[scale=0.6, yshift = -7.1cm, xshift= -2cm]
      \draw [very thick, gray] (0,0) -- +(1,0);
      \draw [very thick, gray] (2,0) -- +(1,0);
      \node[opacity=0] (A21) at (2, 0) {$WWWW$};
      \design
    \end{scope}
    \begin{scope}[scale=0.6, yshift = -6.4cm, xshift= 2cm]
      \draw [very thick, gray] (0,0) -- +(1,0);
      \draw [very thick, gray] (3,0) -- +(1,0);
      \node[opacity=0] (A22) at (2, 0) {$WWWW$};
      \design
    \end{scope}
\draw[->] (A10) -- (A20);
\draw[->] (A10) -- (A21);
\draw[->] (A10) -- (A22);
    \begin{scope}[scale=0.6, yshift = -7.1cm, xshift= 6cm]
      \draw [very thick, gray] (1,0) -- +(1,0);
      \draw [very thick, gray] (3,0) -- +(1,0);
      \node[opacity=0] (A23) at (2, 0) {$WWWW$};
      \design
    \end{scope}
\draw[->] (A11) -- (A20);
\draw[->] (A12) -- (A21);
\draw[->] (A13) -- (A22);
\draw[->] (A11) -- (A23);
\draw[->] (A13) -- (A23);
    \begin{scope}[scale=0.6, yshift = -10cm, xshift= 0cm]
      \draw [very thick, gray] (0,0) -- +(1,0);
      \draw [very thick, gray] (1,0) -- +(1,0);
      \draw [very thick, gray] (3,0) -- +(1,0);
      \node[opacity=0] (A30) at (2, 0) {$WWWW$};
      \design
    \end{scope}
\draw[->] (A20) -- (A30);
\draw[->] (A22) -- (A30);
\draw[->] (A23) -- (A30);
\end{scope}
\end{tikzpicture}
\]
\end{exa}
\begin{rmk}
  \label{rmk:gradinggraphetc}
  As we can see on the example~\ref{exa:graph32}, the graph $\Gamma_{m,n}$ is naturally graded by the number of connected strands. Let us precise the normalization: an admissible partition $\sqcup_{i=1}^k M_i$ of $\eps(m,n)$ has degree $m+n-k$. With this setting the edges increase the degree by 1. 
Every admissible partition consists of sets of the form (we use the notation of example~\ref{exa:graph32}): 
\[ 
\begin{tikzpicture}
   \newcommand{\dop}[1]{\fill[red] #1 circle (0.2); \node[white] at #1 {$\pmb{+}$};}
 \newcommand{\dom}[1]{\fill[blue!70!white] #1 circle (0.2); \node[white] at #1 {$\pmb{-}$};}
\begin{scope}[scale=0.6]
\dop{(-1,0)}
\dom{(1,0)}
\draw [very thick, gray] (3,0) -- +(1,0);
\dop{(3,0)}
\dom{(4,0)}
\draw [very thick, gray] (6,0) -- +(1,0);
\dop{(6,0)}
\dop{(7,0)}
\draw [very thick, gray] (9,0) -- +(1,0);
\dom{(9,0)}
\dom{(10,0)}
\draw [very thick, gray] (12,0) -- +(2,0);
\dop{(12,0)}
\dop{(13,0)}
\dop{(14,0)}
\draw [very thick, gray] (16,0) -- +(2,0);
\dom{(16,0)}
\dom{(17,0)}
\dom{(18,0)}
\end{scope}

\end{tikzpicture}
\]
Here is a list of all type of edges (where$\vcenter{\hbox{\tikz[scale=0.6]{\fill[green!50!black] (0,0) circle (0.2); \node[white] at (0,0) {$?$};}}}$stands for$\vcenter{\hbox{\tikz[scale=0.6]{\fill[red] (0,0) circle (0.2); \node[white] at (0,0) {$\pmb{+}$};}}}\!\!,$ $\!\!\vcenter{\hbox{\tikz[scale=0.6]{\fill[blue!70!white] (0,0) circle (0.2); \node[white] at (0,0) {$\pmb{-}$};}}}$or ``no dot''):
\[ 
\begin{tikzpicture}
   \newcommand{\dop}[1]{\fill[red] #1 circle (0.2); \node[white] at #1 {$\pmb{+}$};}
 \newcommand{\dom}[1]{\fill[blue!70!white] #1 circle (0.2); \node[white] at #1 {$\pmb{-}$};}
 \newcommand{\dob}[1]{\fill[green!50!black] #1 circle (0.2); \node[white] at #1 {$?$};}
\begin{scope}
    \begin{scope}[scale=0.6]
      \node[opacity=0] (A00) at (0, 0) {$WWWW$};
      \dob{(-1.5,0)}
      \dop{(-0.5,0)}
      \dom{(0.5,0)}
      \dob{(1.5,0)}
    \end{scope}
    \begin{scope}[scale=0.6, yshift=1.5cm]
      \node[opacity=0] (A01) at (0, 0) {$WWWW$};
      \draw [very thick, gray] (-0.5,0) -- +(1,0);
      \dob{(-1.5,0)}
      \dop{(-0.5,0)}
      \dom{(0.5,0)}
      \dob{(1.5,0)}
    \end{scope}
\end{scope}
\begin{scope}[xshift= 3.2cm]
    \begin{scope}[scale=0.6]
      \node[opacity=0] (A10) at (0, 0) {$WWWW$};
      \dob{(-1.5,0)}
      \dom{(-0.5,0)}
      \dom{(0.5,0)}
      \dob{(1.5,0)}
    \end{scope}
    \begin{scope}[scale=0.6, yshift=1.5cm]
      \node[opacity=0] (A11) at (0, 0) {$WWWW$};
      \draw [very thick, gray] (-0.5,0) -- +(1,0);
      \dob{(-1.5,0)}
      \dom{(-0.5,0)}
      \dom{(0.5,0)}
      \dob{(1.5,0)}
    \end{scope}
\end{scope}
\begin{scope}[xshift= 6.4cm]
    \begin{scope}[scale=0.6]
      \node[opacity=0] (A20) at (0, 0) {$WWWW$};
      \dob{(-1.5,0)}
      \dop{(-0.5,0)}
      \dop{(0.5,0)}
      \dob{(1.5,0)}
    \end{scope}
    \begin{scope}[scale=0.6, yshift=1.5cm]
      \node[opacity=0] (A21) at (0, 0) {$WWWW$};
      \draw [very thick, gray] (-0.5,0) -- +(1,0);
      \dob{(-1.5,0)}
      \dop{(-0.5,0)}
      \dop{(0.5,0)}
      \dob{(1.5,0)}
    \end{scope}
\end{scope}
\begin{scope}[yshift = -0cm, xshift= 9.9cm]
    \begin{scope}[scale=0.6]
      \node[opacity=0] (A30) at (0, 0) {$WWWW$};
      \draw [very thick, gray] (-1,0) -- +(1,0);
      \dob{(-2,0)}
      \dom{(-1,0)}
      \dom{(0,0)}
      \dom{(1,0)}
      \dob{(2,0)}
    \end{scope}
    \begin{scope}[scale=0.6, yshift=1.5cm]
      \node[opacity=0] (A31) at (0, 0) {$WWWW$};
      \draw [very thick, gray] (-1,0) -- +(1,0);
      \draw [very thick, gray] (0,0) -- +(1,0);
      \dob{(-2,0)}
      \dom{(-1,0)}
      \dom{(0,0)}
      \dom{(1,0)}
      \dob{(2,0)}
    \end{scope}
\end{scope}
\begin{scope}[yshift = -2cm, xshift= 1.35cm]
    \begin{scope}[scale=0.6]
      \node[opacity=0] (A40) at (0, 0) {$WWWW$};
      \draw [very thick, gray] (0,0) -- +(1,0);
      \dob{(-2,0)}
      \dom{(-1,0)}
      \dom{(0,0)}
      \dom{(1,0)}
      \dob{(2,0)}
    \end{scope}
    \begin{scope}[scale=0.6, yshift=1.5cm]
      \node[opacity=0] (A41) at (0, 0) {$WWWW$};
      \draw [very thick, gray] (-1,0) -- +(1,0);
      \draw [very thick, gray] (0,0) -- +(1,0);
      \dob{(-2,0)}
      \dom{(-1,0)}
      \dom{(0,0)}
      \dom{(1,0)}
      \dob{(2,0)}
    \end{scope}
\end{scope}
\begin{scope}[yshift = -2cm, xshift= 4.95cm]
    \begin{scope}[scale=0.6]
      \node[opacity=0] (A50) at (0, 0) {$WWWW$};
      \draw [very thick, gray] (0,0) -- +(1,0);
      \dob{(-2,0)}
      \dop{(-1,0)}
      \dop{(0,0)}
      \dop{(1,0)}
      \dob{(2,0)}
    \end{scope}
    \begin{scope}[scale=0.6, yshift=1.5cm]
      \node[opacity=0] (A51) at (0, 0) {$WWWW$};
      \draw [very thick, gray] (-1,0) -- +(1,0);
      \draw [very thick, gray] (0,0) -- +(1,0);
      \dob{(-2,0)}
      \dop{(-1,0)}
      \dop{(0,0)}
      \dop{(1,0)}
      \dob{(2,0)}
    \end{scope}
\end{scope}
\begin{scope}[yshift = -2cm, xshift= 8.75cm]
    \begin{scope}[scale=0.6]
      \node[opacity=0] (A60) at (0, 0) {$WWWW$};
      \draw [very thick, gray] (-1,0) -- +(1,0);
      \dob{(-2,0)}
      \dop{(-1,0)}
      \dop{(0,0)}
      \dop{(1,0)}
      \dob{(2,0)}
    \end{scope}
    \begin{scope}[scale=0.6, yshift=1.5cm]
      \node[opacity=0] (A61) at (0, 0) {$WWWW$};
      \draw [very thick, gray] (-1,0) -- +(1,0);
      \draw [very thick, gray] (0,0) -- +(1,0);
      \dob{(-2,0)}
      \dop{(-1,0)}
      \dop{(0,0)}
      \dop{(1,0)}
      \dob{(2,0)}
    \end{scope}
\end{scope}
\draw[->] (A00) -- (A01);
\draw[->] (A10) -- (A11);
\draw[->] (A20) -- (A21);
\draw[->] (A30) -- (A31);
\draw[->] (A40) -- (A41);
\draw[->] (A50) -- (A51);
\draw[->] (A60) -- (A61);

\end{tikzpicture}
\]
\end{rmk}

The edges induce a poset structure on $\Gamma_{m,n}$. We fix some notations which will be useful in section~\ref{sec:cobordisms}.

\begin{notation}
  \label{not:lup}
  Let $P_1$ and $P_2$ two admissible partitions of $\epsilon_{m,n}$, we write $P_1\pitchfork P_2$ when the following conditions are satisfied:
  \begin{enumerate}
  \item\label{item:disjoint} The canonical partition is an upper lower bound for $\{P_1, P_2\}$. In the dots-strands setting, this means that $P_1$ and $P_2$ have no common strand. 
  \item\label{item:compatible} $\{P_1, P_2\}$ admits a lower upper bound (in our case, one can check if exists it is unique) denoted by $P_1\vee P_2$. In the dots-strands setting, this means that the superposition of $P_1$ and $P_2$ still corresponds to an admissible partition.
  \end{enumerate}
\end{notation}
\begin{exa}
  \label{exa:lup}
  For $m=3$ and $n=1$, we consider $P_1=\vcenter{\hbox{\tikz[scale=0.6]{   \draw [very thick, gray] (2,0) -- +(1,0); \dops{(0, 0)} \dops{(1, 0)} \dops{(2, 0)} \doms{(3, 0)} }}}\!\!,$ $P_2=\vcenter{\hbox{\tikz[scale=0.6]{  \draw [very thick, gray] (0,0) -- +(1,0); \dops{(0, 0)} \dops{(1, 0)} \dops{(2, 0)} \doms{(3, 0)} }}}\!$ and $P_3=\vcenter{\hbox{\tikz[scale=0.6]{  \draw [very thick, gray] (0,0) -- +(2,0); \dops{(0, 0)} \dops{(1, 0)} \dops{(2, 0)} \doms{(3, 0)} }}}\!\!.$
 We have:
  \begin{itemize}
  \item $P_1 \pitchfork P_2$ and $P_1\vee P_2=\vcenter{\hbox{\tikz[scale=0.6]{  \draw [very thick, gray] (0,0) -- +(1,0);  \draw [very thick, gray] (2,0) -- +(1,0); \dops{(0, 0)} \dops{(1, 0)} \dops{(2, 0)} \doms{(3, 0)} }}}\!\!.$
  \item $\neg (P_1 \pitchfork P_3)$ (condition~(\ref{item:compatible}) is not fulfilled),
  \item $\neg (P_2 \pitchfork P_3)$ (condition~(\ref{item:disjoint}) is not fulfilled),
  \end{itemize}
\end{exa}

We now give an algebraic interpretation of $\Gamma_{m,n}$:

\begin{dfn}
  \label{dfn:ctilda}
Let $m$ and $n$ be two non-negative integers. 
With every vertex $P$ of $\Gamma_{m,n}$ is associated a $\Usl$-module $\widetilde{C}_{m,n}(P)$, and with every edge $e$ of $\Gamma_{m,n}$ is associated a morphism $d_e$ of $\Usl$-module using the following dictionary:
\begin{center}
\begin{tabular}[!h]{|c||c|}
\hline  $\Gamma_{m,n}$ & $\left(\widetilde{C}_{m,n},d_e\right)$ \\ \hline \hline 
$\sqcup$ &  $\otimes$ \\ \hline
\NB{\dops{(0,0)}} & $V_+$\\ \hline
\NB{\doms{(0,0)}} & $V_-$\\ \hline
\NB{\draw [very thick, gray] (3,0) -- +(1,0);
\dops{(3,0)}
\doms{(4,0)}} & $V_{0,0}$ \\ \hline
\NB{\draw [very thick, gray] (6,0) -- +(1,0);
\dops{(6,0)}
\dops{(7,0)}}  & $V_-$ \\ \hline
\NB{\draw [very thick, gray] (9,0) -- +(1,0);
\doms{(9,0)}
\doms{(10,0)}} & $V_+$ \\ \hline
\NB{\draw [very thick, gray] (12,0) -- +(2,0);
\dops{(12,0)}
\dops{(13,0)}
\dops{(14,0)}} &$V_{0,0}$ \\
\hline
\end{tabular}

\end{center}
\end{dfn}

\begin{exa}
  \label{exa:Ctilda}
  Continuing example~\ref{exa:graph32}, we get the following diagram for $\left(\widetilde{C}_{3,2}, d_\bullet\right)$:
\[
\begin{tikzpicture}
  \begin{scope}[scale =1.33]
    \begin{scope}[scale=0.6]
      \node (A00) at (2, 0) {$ $};
      \node (B00) at (2, 0) {$V_+ \otimes V_+ \otimes V_+ \otimes V_-\otimes V_- $};
    \end{scope}
    \begin{scope}[scale=0.6, yshift = 3.5cm, xshift= -6cm]
      \node (A10) at (2, 0) {$ $};
      \node (B10) at (2, 0)  {$V_- \otimes V_+ \otimes V_-\otimes V_- $};
    \end{scope}
\draw[->] (B00) -- (B10) node [midway, below, sloped] {$\scriptstyle h_{++}^-\otimes \mathrm{id}_{V_+\otimes V_-^{\otimes 2}}$};
    \begin{scope}[scale=0.6, yshift = 3cm, xshift= -2cm]
      \node (A11) at (2, 0) {$ $};
      \node (B11) at (2, 0) {$V_+ \otimes V_- \otimes V_-\otimes V_- $};
    \end{scope}
    \begin{scope}[scale=0.6, yshift = 3.5cm, xshift= 2cm]
      \node (A12) at (2, 0) {$ $};
      \node (B12) at (2, 0) {$V_+ \otimes V_+ \otimes  V_- $};
    \end{scope}
    \begin{scope}[scale=0.6, yshift = 3cm, xshift= 6cm]
      \node (A13) at (2, 0) {$ $};
      \node (B13) at (2, 0) {$V_+ \otimes V_+ \otimes V_+ \otimes V_+ $};
    \end{scope}
\draw[->] (B00) -- (B11)node [midway, below, sloped] {$\scriptstyle\mathrm{id}_{V_+} \otimes h_{++}^-\otimes \mathrm{id}_{V_-^{\otimes 2}}\quad $};
\draw[->] (B00) -- (B12) node [midway, below, sloped] {$\scriptstyle\mathrm{id}_{V_+^{\otimes 2}}\otimes d_{+-}\otimes \mathrm{id}_{V_-} $};
\draw[->] (B00) -- (B13) node [midway, below, sloped] {$\scriptstyle\mathrm{id}_{V_+^{\otimes 3}}\otimes h_{--}^+$};
    \begin{scope}[scale=0.6, yshift = 6.5cm, xshift= -6cm]
      \node (A20) at (2, 0) {$ $};
      \node (B20) at (2, 0) {$ V_- \otimes V_-$};;
    \end{scope}
    \begin{scope}[scale=0.6, yshift = 7cm, xshift= -2cm]
      \node (A21) at (2, 0) {$ $};
      \node (B21) at (2, 0) {$V_- \otimes V_- $};
    \end{scope}
    \begin{scope}[scale=0.6, yshift = 6.5cm, xshift= 2cm]
      \node (A22) at (2, 0) {$ $};
      \node (B22) at (2, 0) {$V_-\otimes V_+ \otimes V_+ $};
    \end{scope}
    \begin{scope}[scale=0.6, yshift = 7cm, xshift= 6cm]
      \node (A23) at (2, 0) {$ $};
      \node (B23) at (2, 0) {$V_+ \otimes V_- \otimes V_+ $};;
        \end{scope}
\draw[->] (B10) -- (B20) node [midway,red] {$\bullet$};
\draw[->] (B10) -- (B21);
\draw[->] (B10) -- (B22);
\draw[->] (B11) -- (B20) ;
\draw[->] (B12) -- (B21)  node [midway,red] {$\bullet$};
\draw[->] (B13) -- (B22) node [midway,red] {$\bullet$};
\draw[->] (B11) -- (B23)  node [midway,red] {$\bullet$};
\draw[->] (B13) -- (B23);  

    \begin{scope}[scale=0.6, yshift = 10cm, xshift= 0cm]
      \node (B30) at (2, 0) {$V_+$};
    \end{scope}
\draw[->] (B20) -- (B30) node [midway, above, sloped] {$\scriptstyle h_{--}^+$};
\draw[->]   (B22) -- (B30) node [midway, above, sloped] {$\scriptstyle  d_{-+} \otimes \mathrm{id}_{V_+}$};
\draw[->] (B23) -- (B30) node [midway, above, sloped] {$\scriptstyle d_{+-} \otimes \mathrm{id}_{V_+} $}; 
  \end{scope}
\end{tikzpicture}
\]
For readability reasons, the middle arrows are not explicitly given, but one can of course obtain them from the dictionary. The meaning of the red dots will be explained later. 
\end{exa}

\begin{propdfn}
  \label{pd:Cmn}
  $\left(\widetilde{C}_{m,n}, d_\bullet\right)$ is a $S$-shaped pre-complex (all the squares commute, see appendix~\ref{sec:complexes-cubes} for details). Hence we can apply lemma~\ref{lem:precom2com} and obtain  
a complex $\left(C_{m,n}, d_\bullet\right)$ via a sign assignment (all of the possibilities yield isomorphic complexes). 
\end{propdfn}
\begin{proof}
  The natural hypercube to consider is $B_{m+n-1}$ with the following correspondence: if $P$ is an admissible partition we define the function $f_P$ (it is a face of dimension $0$):
\[
\begin{array}[!h]{rcrcl}
  f_P&:& [1,m+n-1]& \to& \ZZ/2\ZZ\\
&& i& \mapsto &
\begin{cases}
  1& \textrm{if $i$ and $i+1$ belong to the same set in $P$,} \\
  0& \textrm{else.}
\end{cases}
\end{array}
\]
The cubical set $S$ is clearly strong-inductive (the maximal faces are ``maximal'' admissible partition). The fact that all squares commute follows from proposition~\ref{prop:easyrelUsl}. All maps in $C_{m,n}$ being surjective the non-nullity condition is clearly fulfilled.
\end{proof}

An adequate sign assignement is depicted in example~\ref{exa:Ctilda}: the red dotted arrows are to be multiplied by $-1$ in order to construct the complex $C_{m,n}$.

\begin{thm}\label{mainthm}
The cohomology of the complex $C_{m,n}$ is  concentrated in degree 0 where it is isomorphic to the irreducible $\Usl$-module $V_{m,n}$.
\end{thm}

\begin{rmk}
  \label{rmk:thmfits} 
  The formula of corollary~\ref{cor:explicitformula} can be understood through the combinatoric of $\Gamma_{m,n}$.
\end{rmk}

\begin{proof}The proof is decomposed two main steps: first we compute $H^0(C_{m,n})$, and second we compute $H^i(C_{m,n})$ for $i>0$. For each of these steps we separate two cases: on the one hand, $m$ or $n$ equal $0$; on the other hand, $m$ and $n$ different from $0$. 

  \emph{1. Case $n=0$ (the case $m=0$ is similar).} The space $H^0(C_{m,0})$ is a subrepresentation of $V_+^{\otimes n}$. Remark that the map $h_{++}^-$ is the $q$-antisymetrization: $V_-$ is isomorphic as a $\Usl$-module to $V_+\wedge V_+$. Hence $H^0(C_{m,0})\simeq \mathrm{Sym}^m_q(V_+) \simeq V_{m,0}$.

\emph{Case $m, n \neq 0$.} We consider $d^0 = d'+ d''$ where $d'=\id_{V_+^{\otimes m-1}}\otimes d_{+-} \otimes \id_{V_-^{\otimes n-1}}$  and $d'$ is the sum of all the $\id_\bullet\otimes h_{\pm\pm}^{\mp}\otimes\id_\bullet $'s).  Since $d'$ and $d''$ have images in distinct direct summand, we have: $\Ker d^0 = \Ker(d'_{|\Ker d''})$. This implies (thanks to the previous discussion) that $H^0(C_{m,0})$ is a subrepresentation of $V_{m,0}\otimes V_{0,n}$. 
We know that $H^0(C_{m,n})$ is not trivial (it contains a highest vector of weight $(m,n)^\star$). Now, since $d'_{|\Ker d''}$ clearly maps $\mathrm{Sym}^m_q(V_+) \otimes \mathrm{Sym}^n_q(V_-)$ onto $\mathrm{Sym}^{m-1}_q(V_+) \otimes \mathrm{Sym}^{n-1}_q(V_-)$, we have $H^0(C_{m,n})\simeq V_{m,n}$ thanks to proposition~\ref{prop:CGformula}. 

\emph{2.} We prove that $H^{i}(C_{m,n})=0$ by induction on $m+n$. If $m+n$ is equal to $0$ or $1$, the statement is obvious. Let $N$ be an integer greater than or equal to $2$. We suppose that the statement holds for all pairs $(m,n)$ with $m+n<N$. We consider a pair $(m_0,n_0)$ of integers such that $m_0+n_0=N$. 

\emph{Case $m_0$ and $n_0$ different from $0$.} We have the following short exact sequence of complexes:
\[
0 \longrightarrow C_{m_0 -1, 0} \otimes C_{0, n_0 -1}\{+1\} \longrightarrow C_{m_0,n_0} \longrightarrow C_{m_0,0}\otimes C_{0,n_0} \longrightarrow 0,
\]
where $\{\bullet\}$ denotes an homological grading shift. This short exact sequence and the induction hypothesis implies that $H^i(C_{m_0,n_0})=0$ for $i>1$. We write the long exact sequence in small homological degrees:
\begin{align*}
&0 \longrightarrow H^0(C_{m_0,n_0}) \longrightarrow H^0(C_{m_0,0}\otimes C_{0,n_0})  \\ & \hspace{1.5cm}   \longrightarrow  
H^1(C_{m_0 -1, 0} \otimes C_{0, n_0 -1}\{+1\}) \longrightarrow H^1(C_{m_0,n_0}) \longrightarrow  0.
\end{align*}
The induction hypothesis gives:
\[ 
0\longrightarrow V_{m_0,n_0} \longrightarrow V_{m_0,0}\otimes V_{0,n_0} \longrightarrow V_{m_0 -1, 0} \otimes V_{0, n_0 -1} \longrightarrow H^1(C_{m_0,n_0}) \longrightarrow  0,
\]
and finally $H^1(C_{m_0,n_0})=0$ (see proposition~\ref{prop:CGformula}).

\emph{Case $n_0=0$ (the case $m_0=0$ is similar).} We have the following short exact sequence of complexes: 
\[
0\To C_{m_0-2,1}\{+1\} \longrightarrow C_{m_0,0} \longrightarrow C_{1,0}\otimes C_{m_0-1,0} \To 0.
\] 
This short exact sequence and the induction hypothesis implies that $H^i(C_{m_0,0})=0$ for $i>1$. Let us write the long exact sequence for small homological degrees:
\begin{align*}
   & 0 \longrightarrow H^0(C_{m_0,0}) \longrightarrow H^0(C_{1,0}\otimes C_{m_0-1,0}) \longrightarrow \\ &\hspace{1.5cm}  
H^1(C_{m_0-2,1}\{+1\}) \longrightarrow H^1(C_{m_0,0}) \longrightarrow 0.
\end{align*}
The induction hypothesis gives:
\[
0\To V_{m_0,0} \To V_{1,0} \otimes V_{m_0 -1,0} \To V_{m_0-2,1}  \To H^1(C_{m_0,0}) \longrightarrow 0.
\]
This gives $H^1(C_{m_0,0})=0$  thanks to  proposition~\ref{prop:CGformula}.
\end{proof}

\begin{rmk}
  \label{rmk:dequantifyversion}
  This result and its de-quantified proof are valid for $\sll_3$-modules. 
\end{rmk}


\section{Categorification of the colored $\sll_3$-invariant}
\label{sec:categ-color-sll_3}

\subsection{Overview in an ideal world}
Our strategy to categorify the colored $\sll_3$-invariant is to mimic the (first) categorification of the colored Jones polynomial due to Khovanov \cite{MR2124557}. We face a major issue: the required adequate cobordisms (foam-cobordisms) are not in the category of framed links, and there is no result of functoriality about them. To give the flavor and the ideas of our categorification, we suppose in this subsection that this not a problem: we give the construction assuming that the $\sll_3$-homology is functorial with respect to framed foams (see conjecture~\ref{cjc:foamycob}).  

\label{sec:overview}
\begin{dfn}
  \label{dfn:graph2cabling}
  Let $D$ be an oriented knot diagram, $m$ and $n$ two non-negative integers and $P$ an admissible partition of $\eps(m,n)$. Then $D_P$ is the oriented link diagram obtained by cabling $D$ with $k$ strands, $k$ being the number of sets of $P$ with $1$ or $2$ elements with the same sign. Each strand is naturally associated with such a set (see definition~\ref{dfn:partition}). Its orientation is the same as the original strand in $D$ if its corresponding set is of the form\NB{\dops{(0,0)}}or\NB{\draw [very thick, gray] (6,0) -- +(1,0); \doms{(6,0)} \doms{(7,0)}}and is the opposite orientation if the corresponding set is of the form\NB{\doms{(0,0)}}or\NB{\draw [very thick, gray] (6,0) -- +(1,0); \dops{(6,0)} \dops{(7,0)}}. We say that $D_P$ is the \emph{$P$-cabling of $D$}.
\end{dfn}
\begin{exa}
  \label{exa:Pcabling}
  If $P$ is equal to\,\NB{\draw [very thick, gray] (1,0) -- +(2,0); \draw [very thick, gray] (6,0) -- +(1,0); \dops{(0,0)} \dops{(1,0)} \dops{(2,0)} \dops{(3,0)} \dops{(4,0)} \doms{(5,0)} \doms{(6,0)} \doms{(7,0)} \doms{(8,0)}}\!\!, then we have:
\[
\begin{tikzpicture}[scale =0.7, decoration={markings, mark=at
     position 0.5 with {\arrow{>}}},postaction={decorate}]
  \draw[postaction={decorate}] (0,0) -- +(0,1);
\node at  (0, -0.5) {$D$};
  \begin{scope}[xshift= 4cm]
      \draw[postaction={decorate}] (0,0) -- +(0,1);
      \draw[postaction={decorate}] (1,0) -- +(0,1);
      \draw[postaction={decorate}] (2,1) -- +(0,-1);
      \draw[postaction={decorate}] (3,0) -- +(0,1);
      \draw[postaction={decorate}] (4,1) -- +(0,-1);
\node at  (2, -0.5) {$D_P$};
  \end{scope}
\end{tikzpicture}
\]
\end{exa}

\begin{cjc}
  \label{cjc:foamycob}
  The functor $\F$ defined in \ref{pd:functor} extends to knotted framed webs and framed foam-cobordisms.
\end{cjc}
\begin{rmk}
  \label{rmk:precisioncjc}
  One should carefully define what knotted framed webs and framed foam-cobordisms are.
\end{rmk}
In the rest of this subsection, we assume that the conjecture~\ref{cjc:foamycob} holds. We will give no proof here since, in the end, we want to bypass this conjecture.
\begin{dfn}
  \label{dfn:cobincomplex}
  Let $D$ be an oriented knot diagram, $m$ and $n$ be two negative integers and $e: P_1 \To P_2$ an edge in $\Gamma_{m,n}$. We associate with $e$ and $D$ a cobordism $f_D(e)$ between $D_{P_1}$ and $D_{P_2}$: It is the identity cobordism on all strands not concerned with the edge $e$ and on the strands concerned by $e$, we use the following table:
\begin{center}
  \begin{tabular}[ht]{|c||c|}
    \hline $e$ & $f_D(e)$ \\ \hline \hline
\NB{\dops{(0,0)} \doms{(1,0)} \draw[->](1.4,0) -- (2.6,0); \draw[very thick, gray](3,0) -- +(1,0); \dops{(3,0)} \doms{(4,0)} } & \NB{\dpm}$\times D: D_{P_1} \To D_{P_2}$ \\ \hline
\NB{\dops{(0,0)} \dops{(1,0)} \draw[->](1.4,0) -- (2.6,0); \draw[very thick, gray](3,0) -- +(1,0); \dops{(3,0)} \dops{(4,0)} } & \NB{\hppm}$\times D: D_{P_1} \To D_{P_2}$ \\ \hline
\NB{\doms{(0,0)} \doms{(1,0)} \draw[->](1.4,0) -- (2.6,0); \draw[very thick, gray](3,0) -- +(1,0); \doms{(3,0)} \doms{(4,0)} } & \NB{\hmmp}$\times D: D_{P_1} \To D_{P_2}$ \\ \hline
\NB{\draw[very thick, gray](0,0) -- +(1,0);\dops{(0,0)} \dops{(1,0)} \dops{(2,0)} \draw[->](2.4,0) -- (3.6,0); \draw[very thick, gray](4,0) -- +(2,0); \dops{(4,0)} \dops{(5,0)} \dops{(6,0)} } & \NB{\dmp} $\times D: D_{P_1} \To D_{P_2}$ \\ \hline
\NB{\draw[very thick, gray](1,0) -- +(1,0);\dops{(0,0)} \dops{(1,0)} \dops{(2,0)} \draw[->](2.4,0) -- (3.6,0); \draw[very thick, gray](4,0) -- +(2,0); \dops{(4,0)} \dops{(5,0)} \dops{(6,0)} } & \NB{\dpm}$\times D: D_{P_1} \To D_{P_2}$ \\ \hline
\NB{\draw[very thick, gray](0,0) -- +(1,0);\doms{(0,0)} \doms{(1,0)} \doms{(2,0)} \draw[->](2.4,0) -- (3.6,0); \draw[very thick, gray](4,0) -- +(2,0); \doms{(4,0)} \doms{(5,0)} \doms{(6,0)} } & \NB{\dpm}$\times D: D_{P_1} \To D_{P_2}$ \\ \hline
\NB{\draw[very thick, gray](1,0) -- +(1,0);\doms{(0,0)} \doms{(1,0)} \doms{(2,0)} \draw[->](2.4,0) -- (3.6,0); \draw[very thick, gray](4,0) -- +(2,0); \doms{(4,0)} \doms{(5,0)} \doms{(6,0)} } & \NB{\dmp}$\times D: D_{P_1} \To D_{P_2}$ \\ \hline
  \end{tabular}
\end{center}
\end{dfn}

\begin{rmk}
  Note that this is compatible with definition~\ref{dfn:graph2cabling}. These two definitions can be thought as a suspension of the graph $\Gamma_{m,n}$ along $D$.
\end{rmk}

Let us consider an oriented knot diagram $D$ and $m$ and $n$ two non-negative integers. We construct a complex $\C_\bullet(D, V_{m,n})$ using definitions~\ref{dfn:graph2cabling} and \ref{dfn:cobincomplex}: 
\begin{align*}
  \C_i(D, V_{m,n})&= \bigoplus_{\textrm{$P$ of degree $i$}} H(D_{P}) \quad \textrm{and}\\
  d_i&: \C_i\To \C_{i+1},\\
  d_i&=\sum_{\textrm{$P$ of degree $i$}} \sum_{P_1\stackrel{e}{\To} P_2}(-1)^\bullet \F(f_e(D))
\end{align*}
There is a sign indeterminacy in the definition of $d_i$: one should choose the signs to make sure that $d$ is indeed a differential. Before any modification, all squares commute. We want them to anti-commute. One can easily show that such a sign modification exists and that two adequate sign modifications yield isomorphic complexes (see appendix~\ref{sec:complexes-cubes}). The Euler characteristic of this complex is defined by:
\[
\chi(\C_\bullet(D, V_{m,n})) = \sum_{j\in \NN} (-1)^j \chi_q(\C_j) \in \CC[q, q\m]
\]
where $\chi_q(\C_j)$ is defined in remark~\ref{rmk:homdeginQ}.

One can of course play the same game with a link with $k$ components and a pair $(\mathbf{m}, \mathbf{n})$ of $k$-tuples of non-negative integers, by applying the same procedure with the graph $\Gamma_{\mathbf{m}, \mathbf{n}}$.

\begin{thm}[Assuming conjecture~\ref{cjc:foamycob} holds]
  \label{thm:coloredhomwithcjc}
  The isomorphism type of $\C_\bullet(D, V_{\mathbf{m},\mathbf{n}})$ depends only on $(\mathbf{m}, \mathbf{n})$ and on the framed oriented isotopy type of $D$ and we have:
\[
\chi(\C_\bullet(D, V_{\mathbf{m},\mathbf{n}})) = e^{i\pi\frac23(n_+-n_-)s(\mathbf{m}-\mathbf{n})}\kup{(D, V_{\mathbf{m},\mathbf{n}})},
\]
where $(D, V_{\mathbf{m},\mathbf{n}})$ denotes the diagram $D$ colored with the $\Usl$-modules $V_{\mathbf{m},\mathbf{n}}$; $n_+$ and $n_-$ are respectively the number of positive and negative crossings\footnote{Remark that $(n_+-n_-)$ is a framed link invariant.} in $D$; and  $s(\mathbf{m}- \mathbf{n})$ is the sum over all coordinate of $(\mathbf{m}-\mathbf{n})$.
\end{thm}

\begin{proof}[Sketch of the proof]
  The links being framed, the operation of cabling is well-defined. This proves that each $\C_i(D, V_{m,n})$ is a link invariant. Thanks to conjecture~\ref{cjc:foamycob}, the morphisms $\F(f_e(D))$ are well-defined, hence  $\C_\bullet(D, V_{\mathbf{m},\mathbf{n}})$ depends only on $(\mathbf{m}$, $\mathbf{n})$ and on the framed oriented isotopy type of $D$.
The formula on the Euler characteristic follows directly from proposition~\ref{prop:formulakup} and on the definition of the graphs $\Gamma_{\bullet,\bullet}$. 
\end{proof}

We want to set a definition of  $\C_\bullet(D, V_{\mathbf{m},\mathbf{n}})$ bypassing conjecture~\ref{cjc:foamycob} in such a way that theorem~\ref{thm:coloredhomwithcjc} still holds. The main issue is of course to have a proper definition of the morphisms $\F(f_e(D))$. Here is a list of the main steps:
\begin{itemize}
\item Give a new definition of $f_e(D)$ in terms of movies (we have to make some arbitrary choices). See section~\ref{sec:two-foamy-cobordisms}. 
\item Extend $\F$ to movies. See section~\ref{sec:an-aval-categ}
\item Prove that the dependence of $\F(f_e(D))$ in the choices made in the first step is only a multiplication by a sign. See~\ref{sec:invariance-result}.
\item Prove some commutativity property (up to a sign) of some squares to replace isotopies of foams in the ideal world of foam-cobordism to some isotopies. See section~\ref{sec:invariance-result}
\item Prove some non-nullity property in order to apply lemma \ref{lem:precom2com} and get rid of the sign ambiguities. See section~\ref{sec:invariance-result}
\end{itemize}

We finally conclude in section~\ref{sec:color-sll_3-homol-1}.


\subsection{The framework of canopoleis}
\label{sec:canopoleis}

As Bar-Natan explains in \cite{MR2174270}, canoploeis is the right notion to deal with the ``locality'' of the definition of knot homology. It has been used by  Morrison and Nieh \cite{MR2457839}, Clark \cite{MR2482322} and Lewark\footnote{Lewark sees canopoleis as a special case of an extended notion of planar algebra.} \cite{LewarkThese} to present (and compute, see\cite{MR3248745}) the $\sll_3$-homology. We present shortly this notion in order to fix  notations.

\begin{dfn}
  We consider a 2-manifold $M= B_0 \setminus (\mathring{B}_1 \cup \dots \cup \mathring{B}_n)$ where $B_0$ is the standard closed disk in $\RR^2$ and the other $B_i$'s are small closed disks in the interior of $B_0$. The connected components of $\partial M$ are labelled from $0$ to $n$ (with the convention that the 0th connected component is $\partial B_0^2$), furthermore they all carry a base point that we may called $*$. Let $\gamma$ be a compact oriented smooth $1$-sub-manifold of $M$ which does not meet the base points (with the standard conventions that $\partial \gamma \subset \partial M$ and $\gamma\setminus \partial\gamma \subset M\setminus\partial M$). Then the data given by $M$, $\gamma$ and the base points is called an \emph{input-diagram with $n$ inputs} is denoted $(M,\gamma)$. The input-diagrams are regarded up to planar isotopies.  If $(M,\gamma)$ is an input-diagram with $n$ inputs, then the isotopy type of $\gamma\cap B_i$ is canonically encoded by a sequence of signs (start with the base point and scan $\partial B_i$ in the conterclockwise way). We denote $\epsilon_0,\dots,\epsilon_n$ these sequences of signs.
\end{dfn}

\begin{figure}[ht]
  \centering
  \begin{tikzpicture}[scale=0.5]
    \input{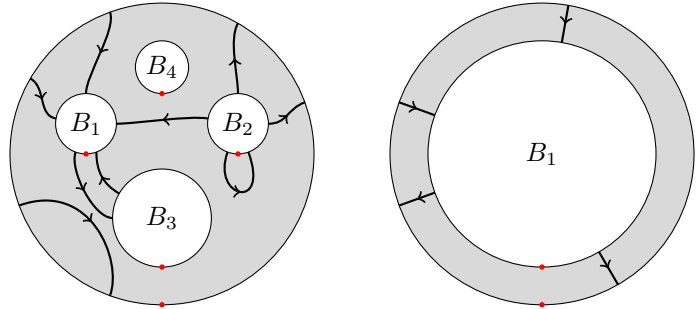}
  \end{tikzpicture}
  \caption{Two examples of input-diagrams. The base points are depicted by red dots. The input-diagram on the left-hand side has 4 inputs and we have: $\epsilon_0 = (+,+,-,-,-,+)$, $\epsilon_1= (-,-,-,-,+)$, $\epsilon_2=(-,+,+,+,+)$, $\epsilon_3=(+,-)$ and $\epsilon_4=()$. The input-diagram on the right-hand side is $\textrm{sun}_{(+,-,-,+)}$.
}
  \label{fig:exainput}
\end{figure}

\begin{dfn}
  If $(M,\gamma)$ and $(M',\gamma')$ are two input-diagrams with $n$ inputs and $n'$ inputs respectively, and if for some $i$, $\epsilon_i= \epsilon'_0$. Then one might glue $(M',\gamma')$ to $(M,\gamma)$ by identifying $\partial B_i$ and $\partial B'_0$. This yields a new input-diagram with $n+n'-1$ inputs  called the \emph{$i$th composed of $(M,\gamma)$ and $(M',\gamma')$} and denoted by $(M,\gamma)\circ_i (M',\gamma')$ (the labels of boundary components of this new input-diagram are obtained by inserting the components corresponding to $B'_1, \dots B'_{n'}$ between $\partial B_{i-1}$ and $\partial B_{i+1}$). See figure~\ref{fig:exainput}

Let $\epsilon$ be a finite sequence of signs, we denote by $\textrm{sun}_\epsilon$ the input-diagram with one input and $\epsilon_0=\epsilon_1=\epsilon$ consisting of an annulus with appropriately oriented radial strands (see figure~\ref{fig:exainput2}).
\end{dfn}

\begin{figure}[ht]
  \centering
  \begin{tikzpicture}[scale=0.6]
    \input{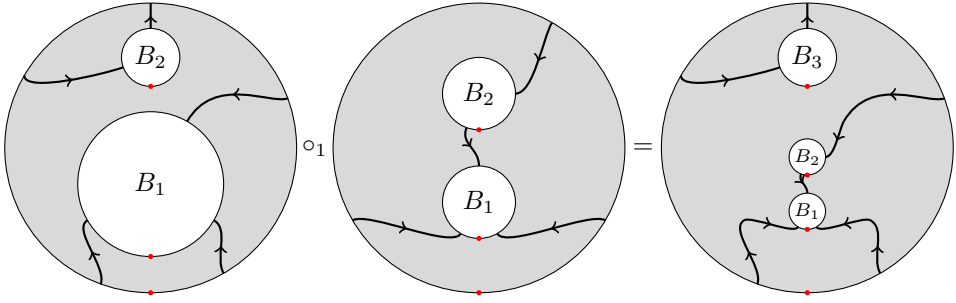}
  \end{tikzpicture}
  \caption{Composition of input-diagrams.}
  \label{fig:exainput2}
\end{figure}
\begin{rmk}
  The input-diagrams $\textrm{sun}_\bullet$ behave like identities:
for every input-diagram $(M,\gamma)$, we have:
\[(M,\gamma) \circ_i \textrm{sun}_{\epsilon_i} = (M,\gamma)= \textrm{sun}_{\epsilon_0}\circ_1 (M,\gamma).\]
\end{rmk}

\begin{dfn}
Let $I$ be a subset of the set of finite sequences of signs.
An \emph{$I$-canopolis $\mathcal{P}$} consists of:
\begin{itemize}
\item For each element $\epsilon$ in $I$, a category $\mathcal{P}_\epsilon$.
\item For each input-diagram $(M,\gamma)$ with $n$ inputs, a functor:
\[
\mathcal{P}_{M,\gamma}: \prod_{i=1}^n \mathcal{P}_{\epsilon_i} \to \mathcal{P}_{\epsilon_0}.
\]
\end{itemize}
Furthermore, this data should satisfy:
\begin{itemize}
\item $\mathcal{P}_{\textrm{sun}_\epsilon}$ is isomorphic to the identity functor of $\mathcal{P}_\epsilon$ for all $\epsilon$ in $I$.
\item For any compatible input-diagrams $(M,\gamma)$ and $(M,\gamma')$:
\[
\mathcal{P}_{(M,\gamma)\circ_i (M',\gamma')}= \mathcal{P}_{(M,\gamma)} \circ(\id_{\prod_{k=1}^{i-1} \mathcal{P}_{\epsilon_k}}, \mathcal{P}_{(M',\gamma')},\id_{\prod_{k=i+1}^{n} \mathcal{P}_{\epsilon_k}})
\]
\end{itemize}  
\end{dfn}

\begin{dfn}
Suppose that  $\mathcal{P}^1$ is an $I_1$-canopolis and  $\mathcal{P}^2$ is an $I_2$-canopolis and that $I_1$ is a subset of $I_2$. Then a \emph{morphism of canopoleis} from $\mathcal{P}^1$ to $\mathcal{P}^2$ is a collection of functors $(\phi_\epsilon: \mathcal{P}^1_\epsilon \to \mathcal{P}^2_\epsilon)_{\epsilon\in I_1}$ such that, for every input-diagram $(M,\gamma)$ with $n$ inputs, we have: 
\[\mathcal{P}^2_{M,\gamma} \circ (\phi_{\epsilon_1} ,\cdots , \phi_{\epsilon_n}) = \phi_{\epsilon_0} \circ \mathcal{P}^1_{M,\gamma}
\]
\end{dfn}

\begin{dfn}
  Let $\mathcal{P}$ be an $I$-canopolis and  $(D_j)_{j\in J}$ be a collection of objects in $\bigcup_{\epsilon \in I} \mathcal{P}_\epsilon$. We say that $(D_j)_{j\in J}$ \emph{objects-generates} $\mathcal{P}$ if for every $\epsilon$ in $I$ and every objects $D$ in $\mathcal{P}_\epsilon$, there exists an input-diagram $(M,\gamma)$ and an $n$-uple $(D_{j_1}, \dots, D_{j_n})$ of objects in $(D_j)_{j\in J}$ such that:
\[
D=\mathcal{P}_{M,\gamma}(D_{j_1}\times \dots \times D_{j_n}).
\]
If $(f_k)_{k\in K}$ is a collection of morphisms in $\bigcup_{\epsilon \in I} \mathcal{P}_\epsilon$. We say that $(f_k)_{k\in K}$ \emph{morphisms-generates} $\mathcal{P}$ if for every $\epsilon$ in $I$,
the morphisms of $\mathcal{P}_\epsilon$, are generated by morphisms of the form:
\[
\mathcal{P}_{M,\gamma}(f_{k_1}\times \dots \times f_{k_n}),
\]
where $(M,\gamma)$ is an input-diagram and $(f_{k_1}, \dots, f_{k_n})$ is an $n$-uple of elements in $(f_k)_{k\in K}$
\end{dfn}


\subsection{A morphism of canopoleis}
\label{sec:an-aval-categ}

In this subsection we describe a morphism of canopoleis. We first define the two canopoleis concerned with this morphism.

In what follows, we consider $I_3$ the set of finite sequences of signs which sum up to $0$ modulo $3$. 

\subsubsection{Two canopoleis}
\label{sec:two-canopoleis}

\begin{dfn}
Let $B_0$ be the standard disk in $\RR^2$ with a base point $*$ on its boundary and $\epsilon$ be an element of $I_3$ of length $k$. Fix $\mathbf{x}$  a set of $k$ distinct points on $\partial B_0 \setminus\{*\}$. With each point is associated a sign given by $\epsilon$ (scan $\partial B_0$ counterclockwise starting from $*$).
We define $\mathcal{W}(\epsilon)$ to be the following category:
  \begin{itemize}
  \item objects are pairs $(w,m)$ where $w$ is a web embedded in $B_0$ whose boundary is precisely $\mathbf{x}$ (out-going (resp. in-going) edges correspond to $+$ (resp. $-$)); and $m$ is an integer (the ``$q$-grading-shift of $(w,m)$''), it is denoted by $w\cdot q^m$.
  \item the space of  morphisms from $w_0\cdot q^{m_0}$ to $w_1\cdot q^{m_1}$ is the graded $\ZZ$ module generated by foams $f$ in $B_0\times [0,1]$ such that: 
    \begin{itemize}
    \item $f\cap B_0\times \{0\} = -w_0$ (the ``$-$'' means that the orientation of $w_0$ is reversed),
    \item $f\cap B_0\times \{1\} = w_1$,
    \item $f\cap \partial B_0\times [0,1] = \mathbf{x}\times [0,1]$;
\end{itemize}
subjected to ambient isotopy and the local relations of proposition~\ref{pd:functor}. 
\item The degree of a foam $f$ is given by: $\chi(w_0)+\chi(w_1)-2\chi(\Sigma) + m_1 -m_0$ where $\Sigma$ is the underlying $CW$-complex, where the dots are punctured out.
  \end{itemize}
This category is clearly pre-additive and graded. We denote by $\mathcal{KW}(\epsilon)$ the graded category of bounded complexes\footnote{We consider complexes with homological grading in $\QQ$, see remark \ref{rmk:homdeginQ}.} of $\mathrm{Mat}(\mathcal{W}(\epsilon))$, the additive completion of $\mathcal{W}(\epsilon)$. 
This yields an $I_3$-canopolis $\mathcal{KW}$, where $\mathcal{KW}_{\epsilon} = \mathcal{KW}(\epsilon)$. The functors associated with input-diagrams are given by gluing foams vertically. 
\end{dfn}

\begin{rmk}\label{rmk:dimHOM}
  From proposition~\ref{pd:functor} one can derive that\footnote{We use $\mathrm{HOM}$ instead of $\mathrm{Hom}$ to stress that we consider morphisms of arbitrary degrees.} $\mathrm{HOM}_{\mathcal{KW}(\epsilon)}(w_0\cdot q^{m_0}, w_1\cdot q^{m_1})$ is a free $\ZZ$-module of dimension $q^{|\epsilon|+m_1 -m_0}\kup{(-w_0)_\epsilon w_1}$, where $|\epsilon|$ is the number of element of $\epsilon$ and $(-w_0)_\epsilon w_1$ is the plane web obtained by gluing $(-w_0)$ and $w_1$ along $\epsilon$. This  canopolis already appears in \cite{MR2457839} and \cite{MR2482322}, where it is denoted by $\mathrm{\mathbf{Kob}}(\mathfrak{su}_3)$.
\end{rmk}

\begin{dfn}
  \label{dfn:movie}
  A \emph{movie} is a sequence of knotted web diagrams in the standard disk $B_0$ such that any two consecutive diagrams are the same except in a small ball $B$ where their difference is in the following list (these are called \emph{elementary movies}):
\[ 
      \begin{tikzpicture}[scale=0.4]
        \begin{scope}
\draw[draw=  gray, line width = 2mm] (0,0) --  (8,0); 
\draw [draw = gray, line width = 2mm] (8,4) -- (0,4); 
\draw[draw=  gray, very thick] (0,0) -- (0,4);
\draw[draw=  gray, very thick] (4,0) -- (4,4);
\draw[draw=  gray, very thick] (8,0) -- (8,4);
\draw[draw= white, dotted, line width =1.2mm] (0,0) -- (8,0);
\draw[draw= white, dotted, line width = 1.2mm] (0,4) -- (8,4);
\end{scope}
\begin{scope}[xshift= 2cm, yshift=2cm]
  \filldraw[dotted, fill=gray!50!white] (0,0) circle (1.5cm);
  \filldraw[dotted, fill = white] (0, 0) circle (1cm);
\end{scope}
\begin{scope}[xshift= 6cm, yshift=2cm]
  \filldraw[dotted, fill=gray!50!white] (0,0) circle (1.5cm);
  \filldraw[dotted, fill = white] (0, 0) circle (1cm);
  \draw (0, 0) circle (0.5cm);
\end{scope}

      \end{tikzpicture} \hspace{1cm}
      \begin{tikzpicture}[scale=0.4]
        \begin{scope}
\draw[draw=  gray, line width = 2mm] (0,0) --  (8,0); 
\draw [draw = gray, line width = 2mm] (8,4) -- (0,4); 
\draw[draw=  gray, very thick] (0,0) -- (0,4);
\draw[draw=  gray, very thick] (4,0) -- (4,4);
\draw[draw=  gray, very thick] (8,0) -- (8,4);
\draw[draw= white, dotted, line width =1.2mm] (0,0) -- (8,0);
\draw[draw= white, dotted, line width = 1.2mm] (0,4) -- (8,4);
\end{scope}
\begin{scope}[xshift= 2cm, yshift=2cm]
  \filldraw[dotted, fill=gray!50!white] (0,0) circle (1.5cm);
  \filldraw[dotted, fill = white] (0, 0) circle (1cm);
\end{scope}
\begin{scope}[xshift= 6cm, yshift=2cm]
  \filldraw[dotted, fill=gray!50!white] (0,0) circle (1.5cm);
  \filldraw[dotted, fill = white] (0, 0) circle (1cm);
  \draw (0, 0) circle (0.5cm);
\end{scope}

      \end{tikzpicture} \hspace{1cm}
      \begin{tikzpicture}[scale=0.4]
        \begin{scope}
\draw[draw=  gray, line width = 2mm] (0,0) --  (8,0); 
\draw [draw = gray, line width = 2mm] (8,4) -- (0,4); 
\draw[draw=  gray, very thick] (0,0) -- (0,4);
\draw[draw=  gray, very thick] (4,0) -- (4,4);
\draw[draw=  gray, very thick] (8,0) -- (8,4);
\draw[draw= white, dotted, line width =1.2mm] (0,0) -- (8,0);
\draw[draw= white, dotted, line width = 1.2mm] (0,4) -- (8,4);
\end{scope}
\begin{scope}[xshift= 6cm, yshift=2cm]
  \filldraw[dotted, fill=gray!50!white] (0,0) circle (1.5cm);
  \filldraw[dotted, fill = white] (0, 0) circle (1cm);
  \draw[->] (45:1) .. controls (0,0) and (0,0) .. (-45:1);
  \draw[->] (-135:1) .. controls (0,0) and (0,0) .. (135:1);
\end{scope}
\begin{scope}[xshift= 2cm, yshift=2cm]
  \filldraw[dotted, fill=gray!50!white] (0,0) circle (1.5cm);
  \filldraw[dotted, fill = white] (0, 0) circle (1cm);
  \draw[->] (45:1) .. controls (0,0) and (0,0) .. (135:1);
  \draw[->] (-135:1) .. controls (0,0) and (0,0) .. (-45:1);
\end{scope}
      \end{tikzpicture} \vspace{0.2cm} \] \[
      \begin{tikzpicture}[scale=0.4]
        \begin{scope}
\draw[draw=  gray, line width = 2mm] (0,0) --  (8,0); 
\draw [draw = gray, line width = 2mm] (8,4) -- (0,4); 
\draw[draw=  gray, very thick] (0,0) -- (0,4);
\draw[draw=  gray, very thick] (4,0) -- (4,4);
\draw[draw=  gray, very thick] (8,0) -- (8,4);
\draw[draw= white, dotted, line width =1.2mm] (0,0) -- (8,0);
\draw[draw= white, dotted, line width = 1.2mm] (0,4) -- (8,4);
\end{scope}
\begin{scope}[xshift= 2cm, yshift=2cm]
  \filldraw[dotted, fill=gray!50!white] (0,0) circle (1.5cm);
  \filldraw[dotted, fill = white] (0, 0) circle (1cm);
\draw[->] (60:1).. controls +(-120:0.5) and +(120:0.5) .. (-60:1);
\draw[->] (120:1).. controls +(-60:0.5) and +(60:0.5) .. (-120:1);
\end{scope}
\begin{scope}[xshift= 6cm, yshift=2cm]
  \filldraw[dotted, fill=gray!50!white] (0,0) circle (1.5cm);
  \filldraw[dotted, fill = white] (0, 0) circle (1cm);
\draw[->] (60:1) -- (90:0.5) -- (90:-0.5) -- (60:-1);
\draw (90:0.5) -- (120:1);
\draw[->] (90:-0.5) -- (120:-1);
\end{scope}
      \end{tikzpicture} \hspace{1cm}
      \begin{tikzpicture}[scale=0.4]
        \begin{scope}
\draw[draw=  gray, line width = 2mm] (0,0) --  (8,0); 
\draw [draw = gray, line width = 2mm] (8,4) -- (0,4); 
\draw[draw=  gray, very thick] (0,0) -- (0,4);
\draw[draw=  gray, very thick] (4,0) -- (4,4);
\draw[draw=  gray, very thick] (8,0) -- (8,4);
\draw[draw= white, dotted, line width =1.2mm] (0,0) -- (8,0);
\draw[draw= white, dotted, line width = 1.2mm] (0,4) -- (8,4);
\end{scope}
\begin{scope}[xshift= 6cm, yshift=2cm]
  \filldraw[dotted, fill=gray!50!white] (0,0) circle (1.5cm);
  \filldraw[dotted, fill = white] (0, 0) circle (1cm);
\draw[->] (60:1).. controls +(-120:0.5) and +(120:0.5) .. (-60:1);
\draw[->] (120:1).. controls +(-60:0.5) and +(60:0.5) .. (-120:1);
\end{scope}
\begin{scope}[xshift= 2cm, yshift=2cm]
  \filldraw[dotted, fill=gray!50!white] (0,0) circle (1.5cm);
  \filldraw[dotted, fill = white] (0, 0) circle (1cm);
\draw[->] (60:1) -- (90:0.5) -- (90:-0.5) -- (60:-1);
\draw (90:0.5) -- (120:1);
\draw[->] (90:-0.5) -- (120:-1);
\end{scope}
      \end{tikzpicture} \hspace{1cm}
      \begin{tikzpicture}[scale=0.4]
        \begin{scope}
\draw[draw=  gray, line width = 2mm] (0,0) --  (8,0); 
\draw [draw = gray, line width = 2mm] (8,4) -- (0,4); 
\draw[draw=  gray, very thick] (0,0) -- (0,4);
\draw[draw=  gray, very thick] (4,0) -- (4,4);
\draw[draw=  gray, very thick] (8,0) -- (8,4);
\draw[draw= white, dotted, line width =1.2mm] (0,0) -- (8,0);
\draw[draw= white, dotted, line width = 1.2mm] (0,4) -- (8,4);
\end{scope}
\begin{scope}[xshift= 2cm, yshift=2cm, rotate=90]
  \filldraw[dotted, fill=gray!50!white] (0,0) circle (1.5cm);
  \filldraw[dotted, fill = white] (0, 0) circle (1cm);
\draw (90:1) -- (-90:1);
\end{scope}
\begin{scope}[xshift= 6cm, yshift=2cm, rotate=90]
  \filldraw[dotted, fill=gray!50!white] (0,0) circle (1.5cm);
  \filldraw[dotted, fill = white] (0, 0) circle (1cm);
\draw (90:1) -- (90:0.5).. controls (-0.5, 0) and (-0.5, 0) .. (90: -0.5)-- (90:-1);
\draw (90:0.5).. controls (0.5, 0) and (0.5, 0).. (90: -0.5);
\end{scope}
      \end{tikzpicture}  \vspace{0.2cm} \] \[
      \begin{tikzpicture}[scale=0.4]
        \begin{scope}
\draw[draw=  gray, line width = 2mm] (0,0) --  (8,0); 
\draw [draw = gray, line width = 2mm] (8,4) -- (0,4); 
\draw[draw=  gray, very thick] (0,0) -- (0,4);
\draw[draw=  gray, very thick] (4,0) -- (4,4);
\draw[draw=  gray, very thick] (8,0) -- (8,4);
\draw[draw= white, dotted, line width =1.2mm] (0,0) -- (8,0);
\draw[draw= white, dotted, line width = 1.2mm] (0,4) -- (8,4);
\end{scope}
\begin{scope}[xshift= 6cm, yshift=2cm, rotate=90]
  \filldraw[dotted, fill=gray!50!white] (0,0) circle (1.5cm);
  \filldraw[dotted, fill = white] (0, 0) circle (1cm);
\draw (90:1) -- (-90:1);
\end{scope}
\begin{scope}[xshift= 2cm, yshift=2cm, rotate=90]
  \filldraw[dotted, fill=gray!50!white] (0,0) circle (1.5cm);
  \filldraw[dotted, fill = white] (0, 0) circle (1cm);
\draw (90:1) -- (90:0.5).. controls (-0.5, 0) and (-0.5, 0) .. (90: -0.5)-- (90:-1);
\draw (90:0.5).. controls (0.5, 0) and (0.5, 0).. (90: -0.5);
\end{scope}
      \end{tikzpicture} \hspace{1cm}
      \begin{tikzpicture}[scale=0.4]
        \begin{scope}
\draw[draw=  gray, line width = 2mm] (0,0) --  (8,0); 
\draw [draw = gray, line width = 2mm] (8,4) -- (0,4); 
\draw[draw=  gray, very thick] (0,0) -- (0,4);
\draw[draw=  gray, very thick] (4,0) -- (4,4);
\draw[draw=  gray, very thick] (8,0) -- (8,4);
\draw[draw= white, dotted, line width =1.2mm] (0,0) -- (8,0);
\draw[draw= white, dotted, line width = 1.2mm] (0,4) -- (8,4);
\end{scope}
\begin{scope}[xshift= 2cm, yshift=2cm]
  \filldraw[dotted, fill=gray!50!white] (0,0) circle (1.5cm);
  \filldraw[dotted, fill = white] (0, 0) circle (1cm);
  \draw (90:1) .. controls (0,0) and +(0,-0.50) .. (0.5,0.5).. controls +(-0,+0.2) and +(0.15,0.3) ..  (0.2,0.4);
  \draw (-90:1) .. controls (0,0) and +(0,+0.50) .. (0.5,-0.5).. controls +(-0,-0.2) and +(0.15,-0.3) ..  (0.2,-0.4);
  \draw (0.1,0.2) .. controls +(-0.1, -0.1) and +(-0.1, 0.1) .. (0.1, -0.2);
\end{scope}
\begin{scope}[xshift= 6cm, yshift=2cm]
  \filldraw[dotted, fill=gray!50!white] (0,0) circle (1.5cm);
  \filldraw[dotted, fill = white] (0, 0) circle (1cm);
  \draw (90:1) .. controls (0,0) and (0,0) .. (-90:1);
\end{scope}
      \end{tikzpicture} \hspace{1cm}
      \begin{tikzpicture}[scale=0.4]
        \begin{scope}
\draw[draw=  gray, line width = 2mm] (0,0) --  (8,0); 
\draw [draw = gray, line width = 2mm] (8,4) -- (0,4); 
\draw[draw=  gray, very thick] (0,0) -- (0,4);
\draw[draw=  gray, very thick] (4,0) -- (4,4);
\draw[draw=  gray, very thick] (8,0) -- (8,4);
\draw[draw= white, dotted, line width =1.2mm] (0,0) -- (8,0);
\draw[draw= white, dotted, line width = 1.2mm] (0,4) -- (8,4);
\end{scope}
\begin{scope}[xshift= 2cm, yshift=2cm]
  \filldraw[dotted, fill=gray!50!white] (0,0) circle (1.5cm);
  \filldraw[dotted, fill = white] (0, 0) circle (1cm);
  \draw (90:1) .. controls (0,0) and +(0,-0.50) .. (0.5,0.5).. controls +(-0,+0.2) and +(0.15,0.3) ..  (0.2,0.4);
  \draw (-90:1) .. controls (0,0) and +(0,+0.50) .. (0.5,-0.5).. controls +(-0,-0.2) and +(0.15,-0.3) ..  (0.2,-0.4);
  \draw (0.1,0.2) .. controls +(-0.1, -0.1) and +(-0.1, 0.1) .. (0.1, -0.2);
\end{scope}
\begin{scope}[xshift= 6cm, yshift=2cm]
  \filldraw[dotted, fill=gray!50!white] (0,0) circle (1.5cm);
  \filldraw[dotted, fill = white] (0, 0) circle (1cm);
  \draw (90:1) .. controls (0,0) and (0,0) .. (-90:1);
\end{scope}
      \end{tikzpicture}  \vspace{0.2cm} \] \[
      \begin{tikzpicture}[scale=0.4]
        \begin{scope}
\draw[draw=  gray, line width = 2mm] (0,0) --  (8,0); 
\draw [draw = gray, line width = 2mm] (8,4) -- (0,4); 
\draw[draw=  gray, very thick] (0,0) -- (0,4);
\draw[draw=  gray, very thick] (4,0) -- (4,4);
\draw[draw=  gray, very thick] (8,0) -- (8,4);
\draw[draw= white, dotted, line width =1.2mm] (0,0) -- (8,0);
\draw[draw= white, dotted, line width = 1.2mm] (0,4) -- (8,4);
\end{scope}
\begin{scope}[xshift= 2cm, yshift=2cm]
  \filldraw[dotted, fill=gray!50!white] (0,0) circle (1.5cm);
  \filldraw[dotted, fill = white] (0, 0) circle (1cm);
  \draw (90:1) .. controls (0,0.5) and +(-0.1,+0.3) .. (0.1,0.4);
  \draw (0.2, 0.2) .. controls +(0.1, -0.2) and +(0, -0.2) .. (0.5, 0.5) .. controls +(0, 0.4) and +(0, 0.2).. (0,0);
  \draw (0.2, -0.2) .. controls +(0.1, 0.2) and +(0, 0.2) .. (0.5, -0.5) .. controls +(0, -0.4) and +(0, -0.2).. (0,0);
  \draw (-90:1) .. controls (0,-0.5) and +(-0.1,-0.3) .. (0.1,-0.4);
\end{scope}
\begin{scope}[xshift= 6cm, yshift=2cm]
  \filldraw[dotted, fill=gray!50!white] (0,0) circle (1.5cm);
  \filldraw[dotted, fill = white] (0, 0) circle (1cm);
  \draw (90:1) .. controls (0,0) and (0,0) .. (-90:1);
\end{scope}
      \end{tikzpicture} \hspace{1cm}
      \begin{tikzpicture}[scale=0.4]
        \begin{scope}
\draw[draw=  gray, line width = 2mm] (0,0) --  (8,0); 
\draw [draw = gray, line width = 2mm] (8,4) -- (0,4); 
\draw[draw=  gray, very thick] (0,0) -- (0,4);
\draw[draw=  gray, very thick] (4,0) -- (4,4);
\draw[draw=  gray, very thick] (8,0) -- (8,4);
\draw[draw= white, dotted, line width =1.2mm] (0,0) -- (8,0);
\draw[draw= white, dotted, line width = 1.2mm] (0,4) -- (8,4);
\end{scope}
\begin{scope}[xshift= 6cm, yshift=2cm]
  \filldraw[dotted, fill=gray!50!white] (0,0) circle (1.5cm);
  \filldraw[dotted, fill = white] (0, 0) circle (1cm);
  \draw (90:1) .. controls (0,0.5) and +(-0.1,+0.3) .. (0.1,0.4);
  \draw (0.2, 0.2) .. controls +(0.1, -0.2) and +(0, -0.2) .. (0.5, 0.5) .. controls +(0, 0.4) and +(0, 0.2).. (0,0);
  \draw (0.2, -0.2) .. controls +(0.1, 0.2) and +(0, 0.2) .. (0.5, -0.5) .. controls +(0, -0.4) and +(0, -0.2).. (0,0);
  \draw (-90:1) .. controls (0,-0.5) and +(-0.1,-0.3) .. (0.1,-0.4);
\end{scope}
\begin{scope}[xshift= 2cm, yshift=2cm]
  \filldraw[dotted, fill=gray!50!white] (0,0) circle (1.5cm);
  \filldraw[dotted, fill = white] (0, 0) circle (1cm);
  \draw (90:1) .. controls (0,0) and (0,0) .. (-90:1);
\end{scope}
      \end{tikzpicture} \hspace{1cm}
      \begin{tikzpicture}[scale=0.4]
        \begin{scope}
\draw[draw=  gray, line width = 2mm] (0,0) --  (8,0); 
\draw [draw = gray, line width = 2mm] (8,4) -- (0,4); 
\draw[draw=  gray, very thick] (0,0) -- (0,4);
\draw[draw=  gray, very thick] (4,0) -- (4,4);
\draw[draw=  gray, very thick] (8,0) -- (8,4);
\draw[draw= white, dotted, line width =1.2mm] (0,0) -- (8,0);
\draw[draw= white, dotted, line width = 1.2mm] (0,4) -- (8,4);
\end{scope}
\begin{scope}[xshift= 2cm, yshift=2cm]
  \filldraw[dotted, fill=gray!50!white] (0,0) circle (1.5cm);
  \filldraw[dotted, fill = white] (0, 0) circle (1cm);
\draw (60:1).. controls +(-120:0.5) and +(120:0.5) .. (-60:1);
\draw (120:1).. controls +(-60:0.5) and +(60:0.5) .. (-120:1);
\end{scope}
\begin{scope}[xshift= 6cm, yshift=2cm]
  \filldraw[dotted, fill=gray!50!white] (0,0) circle (1.5cm);
  \filldraw[dotted, fill = white] (0, 0) circle (1cm);
\draw (60:1).. controls +(-135:1.5) and +(135:1.5) .. (-60:1);
\fill[white] (0, 0.36) circle (0.1);
\fill[white] (0, -0.36) circle (0.1);
\draw (120:1).. controls +(-45:1.5) and +(45:1.5) .. (-120:1);
\end{scope}
      \end{tikzpicture}  \vspace{0.2cm} \] \[
      \begin{tikzpicture}[scale=0.4]
        \begin{scope}
\draw[draw=  gray, line width = 2mm] (0,0) --  (8,0); 
\draw [draw = gray, line width = 2mm] (8,4) -- (0,4); 
\draw[draw=  gray, very thick] (0,0) -- (0,4);
\draw[draw=  gray, very thick] (4,0) -- (4,4);
\draw[draw=  gray, very thick] (8,0) -- (8,4);
\draw[draw= white, dotted, line width =1.2mm] (0,0) -- (8,0);
\draw[draw= white, dotted, line width = 1.2mm] (0,4) -- (8,4);
\end{scope}
\begin{scope}[xshift= 6cm, yshift=2cm]
  \filldraw[dotted, fill=gray!50!white] (0,0) circle (1.5cm);
  \filldraw[dotted, fill = white] (0, 0) circle (1cm);
\draw (60:1).. controls +(-120:0.5) and +(120:0.5) .. (-60:1);
\draw (120:1).. controls +(-60:0.5) and +(60:0.5) .. (-120:1);
\end{scope}
\begin{scope}[xshift= 2cm, yshift=2cm]
  \filldraw[dotted, fill=gray!50!white] (0,0) circle (1.5cm);
  \filldraw[dotted, fill = white] (0, 0) circle (1cm);
\draw (60:1).. controls +(-135:1.5) and +(135:1.5) .. (-60:1);
\fill[white] (0, 0.36) circle (0.1);
\fill[white] (0, -0.36) circle (0.1);
\draw (120:1).. controls +(-45:1.5) and +(45:1.5) .. (-120:1);
\end{scope}
      \end{tikzpicture} \hspace{1cm}
      \begin{tikzpicture}[scale=0.4]
        \begin{scope}
\draw[draw=  gray, line width = 2mm] (0,0) --  (8,0); 
\draw [draw = gray, line width = 2mm] (8,4) -- (0,4); 
\draw[draw=  gray, very thick] (0,0) -- (0,4);
\draw[draw=  gray, very thick] (4,0) -- (4,4);
\draw[draw=  gray, very thick] (8,0) -- (8,4);
\draw[draw= white, dotted, line width =1.2mm] (0,0) -- (8,0);
\draw[draw= white, dotted, line width = 1.2mm] (0,4) -- (8,4);
\end{scope}
\begin{scope}[xshift= 2cm, yshift=2cm]
  \filldraw[dotted, fill=gray!50!white] (0,0) circle (1.5cm);
  \filldraw[dotted, fill = white] (0, 0) circle (1cm);
\draw (60:1) -- (-120:1);
\fill[white] (0, 0) circle (0.1);
\draw (-60:1) -- (120:1);
\fill[white] (-0.2, -0.33) circle (0.1);
\fill[white] (0.2, -0.33) circle (0.1);
\draw (180:1).. controls +(-40:0.7) and +(-140:0.7) .. (0:1);
\end{scope}
\begin{scope}[xshift= 6cm, yshift=2cm]
  \filldraw[dotted, fill=gray!50!white] (0,0) circle (1.5cm);
  \filldraw[dotted, fill = white] (0, 0) circle (1cm);
\draw (60:1) -- (-120:1);
\fill[white] (0, 0) circle (0.1);
\draw (-60:1) -- (120:1);
\fill[white] (-0.2, +0.33) circle (0.1);
\fill[white] (0.2, +0.33) circle (0.1);
\draw (180:1).. controls +(+40:0.7) and +(+140:0.7) .. (0:1);
\end{scope}
      \end{tikzpicture} \hspace{1cm}
      \begin{tikzpicture}[scale=0.4]
        \begin{scope}
\draw[draw=  gray, line width = 2mm] (0,0) --  (8,0); 
\draw [draw = gray, line width = 2mm] (8,4) -- (0,4); 
\draw[draw=  gray, very thick] (0,0) -- (0,4);
\draw[draw=  gray, very thick] (4,0) -- (4,4);
\draw[draw=  gray, very thick] (8,0) -- (8,4);
\draw[draw= white, dotted, line width =1.2mm] (0,0) -- (8,0);
\draw[draw= white, dotted, line width = 1.2mm] (0,4) -- (8,4);
\end{scope}
\begin{scope}[xshift= 2cm, yshift=2cm]
  \filldraw[dotted, fill=gray!50!white] (0,0) circle (1.5cm);
  \filldraw[dotted, fill = white] (0, 0) circle (1cm);
\draw (30:1) -- (00,0);
\draw (150:1) -- (00,0);
\draw (-90:1) -- (00,0);
\fill[white] (-0, -0.35) circle (0.1);
\draw (180:1).. controls +(-40:0.7) and +(-140:0.7) .. (0:1);
\end{scope}
\begin{scope}[xshift= 6cm, yshift=2cm]
  \filldraw[dotted, fill=gray!50!white] (0,0) circle (1.5cm);
  \filldraw[dotted, fill = white] (0, 0) circle (1cm);
\draw (30:1) -- (00,0);
\draw (150:1) -- (00,0);
\draw (-90:1) -- (00,0);
\fill[white] (0.45, 0.25) circle (0.1);
\fill[white] (-0.45, 0.25) circle (0.1);
\draw (180:1).. controls +(+40:0.7) and +(+140:0.7) .. (0:1);
\end{scope}
      \end{tikzpicture}  \vspace{0.2cm} \] \[
      \begin{tikzpicture}[scale=0.4]
        \begin{scope}
\draw[draw=  gray, line width = 2mm] (0,0) --  (8,0); 
\draw [draw = gray, line width = 2mm] (8,4) -- (0,4); 
\draw[draw=  gray, very thick] (0,0) -- (0,4);
\draw[draw=  gray, very thick] (4,0) -- (4,4);
\draw[draw=  gray, very thick] (8,0) -- (8,4);
\draw[draw= white, dotted, line width =1.2mm] (0,0) -- (8,0);
\draw[draw= white, dotted, line width = 1.2mm] (0,4) -- (8,4);
\end{scope}
\begin{scope}[xshift= 6cm, yshift=2cm]
  \filldraw[dotted, fill=gray!50!white] (0,0) circle (1.5cm);
  \filldraw[dotted, fill = white] (0, 0) circle (1cm);
\draw (30:1) -- (00,0);
\draw (150:1) -- (00,0);
\draw (-90:1) -- (00,0);
\fill[white] (-0, -0.35) circle (0.1);
\draw (180:1).. controls +(-40:0.7) and +(-140:0.7) .. (0:1);
\end{scope}
\begin{scope}[xshift= 2cm, yshift=2cm]
  \filldraw[dotted, fill=gray!50!white] (0,0) circle (1.5cm);
  \filldraw[dotted, fill = white] (0, 0) circle (1cm);
\draw (30:1) -- (00,0);
\draw (150:1) -- (00,0);
\draw (-90:1) -- (00,0);
\fill[white] (0.45, 0.25) circle (0.1);
\fill[white] (-0.45, 0.25) circle (0.1);
\draw (180:1).. controls +(+40:0.7) and +(+140:0.7) .. (0:1);
\end{scope}
      \end{tikzpicture} \hspace{1cm}
      \begin{tikzpicture}[scale=0.4]
        \begin{scope}
\draw[draw=  gray, line width = 2mm] (0,0) --  (8,0); 
\draw [draw = gray, line width = 2mm] (8,4) -- (0,4); 
\draw[draw=  gray, very thick] (0,0) -- (0,4);
\draw[draw=  gray, very thick] (4,0) -- (4,4);
\draw[draw=  gray, very thick] (8,0) -- (8,4);
\draw[draw= white, dotted, line width =1.2mm] (0,0) -- (8,0);
\draw[draw= white, dotted, line width = 1.2mm] (0,4) -- (8,4);
\end{scope}
\begin{scope}[xshift= 2cm, yshift=2cm]
  \filldraw[dotted, fill=gray!50!white] (0,0) circle (1.5cm);
  \filldraw[dotted, fill = white] (0, 0) circle (1cm);
\draw (180:1).. controls +(-40:0.7) and +(-140:0.7) .. (0:1);
\fill[white] (-0, -0.35) circle (0.1);
\draw (30:1) -- (00,0);
\draw (150:1) -- (00,0);
\draw (-90:1) -- (00,0);
\end{scope}
\begin{scope}[xshift= 6cm, yshift=2cm]
  \filldraw[dotted, fill=gray!50!white] (0,0) circle (1.5cm);
  \filldraw[dotted, fill = white] (0, 0) circle (1cm);
\draw (180:1).. controls +(+40:0.7) and +(+140:0.7) .. (0:1);
\fill[white] (0.45, 0.25) circle (0.1);
\fill[white] (-0.45, 0.25) circle (0.1);
\draw (30:1) -- (00,0);
\draw (150:1) -- (00,0);
\draw (-90:1) -- (00,0);
\end{scope}
      \end{tikzpicture} \hspace{1cm}
      \begin{tikzpicture}[scale=0.4]
        \begin{scope}
\draw[draw=  gray, line width = 2mm] (0,0) --  (8,0); 
\draw [draw = gray, line width = 2mm] (8,4) -- (0,4); 
\draw[draw=  gray, very thick] (0,0) -- (0,4);
\draw[draw=  gray, very thick] (4,0) -- (4,4);
\draw[draw=  gray, very thick] (8,0) -- (8,4);
\draw[draw= white, dotted, line width =1.2mm] (0,0) -- (8,0);
\draw[draw= white, dotted, line width = 1.2mm] (0,4) -- (8,4);
\end{scope}
\begin{scope}[xshift= 6cm, yshift=2cm]
  \filldraw[dotted, fill=gray!50!white] (0,0) circle (1.5cm);
  \filldraw[dotted, fill = white] (0, 0) circle (1cm);
\draw (180:1).. controls +(-40:0.7) and +(-140:0.7) .. (0:1);
\fill[white] (-0, -0.35) circle (0.1);
\draw (30:1) -- (00,0);
\draw (150:1) -- (00,0);
\draw (-90:1) -- (00,0);
\end{scope}
\begin{scope}[xshift= 2cm, yshift=2cm]
  \filldraw[dotted, fill=gray!50!white] (0,0) circle (1.5cm);
  \filldraw[dotted, fill = white] (0, 0) circle (1cm);
\draw (180:1).. controls +(+40:0.7) and +(+140:0.7) .. (0:1);
\fill[white] (0.45, 0.25) circle (0.1);
\fill[white] (-0.45, 0.25) circle (0.1);
\draw (30:1) -- (00,0);
\draw (150:1) -- (00,0);
\draw (-90:1) -- (00,0);
\end{scope}
      \end{tikzpicture} \]
Movies are considered up to the relation ``far modifications commute'':  
\[
 \vspace{0.2cm}
 \begin{tikzpicture}[scale= 0.6]
\begin{scope}
\draw[draw=  gray, line width = 2mm] (0,0) --  (12,0); 
\draw [draw = gray, line width = 2mm] (12,4) -- (0,4); 
\draw[draw=  gray, very thick] (0,0) -- (0,4);
\draw[draw=  gray, very thick] (4,0) -- (4,4);
\draw[draw=  gray, very thick] (8,0) -- (8,4);
\draw[draw=  gray, very thick] (12,0) -- (12,4);
\draw[draw= white, dotted, line width =1.2mm] (0,0) -- (12,0);
\draw[draw= white, dotted, line width = 1.2mm] (0,4) -- (12,4);
\end{scope}
\begin{scope}[xshift= 2cm, yshift=2cm]
  \filldraw[dotted, fill=gray!50!white] (0,0) circle (1.5cm);
  \filldraw[dotted, fill = white] (-0.75, 0) circle (0.5cm);
  \filldraw[dotted, fill = white] (0.75, 0) circle (0.5cm);
  \node at (-0.75,0) {$D_1$};
  \node at (0.75,0) {$D_2$};
\end{scope}
\begin{scope}[xshift= 6cm, yshift=2cm]
  \filldraw[dotted, fill=gray!50!white] (0,0) circle (1.5cm);
  \filldraw[dotted, fill = white] (-0.75, 0) circle (0.5cm);
  \filldraw[dotted, fill = white] (0.75, 0) circle (0.5cm);
  \node at (-0.75,0) {$D'_1$};
  \node at (0.75,0) {$D_2$};
\end{scope}
\begin{scope}[xshift= 10cm, yshift=2cm]
  \filldraw[dotted, fill=gray!50!white] (0,0) circle (1.5cm);
  \filldraw[dotted, fill = white] (-0.75, 0) circle (0.5cm);
  \filldraw[dotted, fill = white] (0.75, 0) circle (0.5cm);
  \node at (-0.75,0) {$D'_1$};
  \node at (0.75,0) {$D'_2$};
\end{scope}
\node at (6, -1) {$\simeq$};
\begin{scope}[yshift=-6cm]
\begin{scope}
\draw[draw=  gray, line width = 2mm] (0,0) --  (12,0); 
\draw[draw = gray, line width = 2mm] (12,4) -- (0,4); 
\draw[draw=  gray, very thick] (0,0) -- (0,4);
\draw[draw=  gray, very thick] (4,0) -- (4,4);
\draw[draw=  gray, very thick] (8,0) -- (8,4);
\draw[draw=  gray, very thick] (12,0) -- (12,4);
\draw[draw= white, dotted, line width =1.2mm] (0,0) -- (12,0);
\draw[draw= white, dotted, line width = 1.2mm] (0,4) -- (12,4);
\end{scope}
\begin{scope}[xshift= 2cm, yshift=2cm]
  \filldraw[dotted, fill=gray!50!white] (0,0) circle (1.5cm);
  \filldraw[dotted, fill = white] (-0.75, 0) circle (0.5cm);
  \filldraw[dotted, fill = white] (0.75, 0) circle (0.5cm);
  \node at (-0.75,0) {$D_1$};
  \node at (0.75,0) {$D_2$};
\end{scope}
\begin{scope}[xshift= 6cm, yshift=2cm]
  \filldraw[dotted, fill=gray!50!white] (0,0) circle (1.5cm);
  \filldraw[dotted, fill = white] (-0.75, 0) circle (0.5cm);
  \filldraw[dotted, fill = white] (0.75, 0) circle (0.5cm);
  \node at (-0.75,0) {$D_1$};
  \node at (0.75,0) {$D'_2$};
\end{scope}
\begin{scope}[xshift= 10cm, yshift=2cm]
  \filldraw[dotted, fill=gray!50!white] (0,0) circle (1.5cm);
  \filldraw[dotted, fill = white] (-0.75, 0) circle (0.5cm);
  \filldraw[dotted, fill = white] (0.75, 0) circle (0.5cm);
  \node at (-0.75,0) {$D'_1$};
  \node at (0.75,0) {$D'_2$};
\end{scope}
\end{scope}

 \end{tikzpicture}
 \]
\end{dfn}

\begin{rmk}
  Links cobordisms can be presented via movies and movie moves (see \cite{MR1905687} and \cite{MR2462446} for the framed case).
\end{rmk}

\begin{dfn}
  \label{dfn:canomovie}
Let $B_0$ be the standard disk in $\RR^2$ with a base point $*$ on its boundary and $\epsilon$ be an element of $I_3$ of length $k$. Fix $\mathbf{x}$  a set of distinct $k$ points on $\partial B_0\setminus\{*\}$. With each point is associated a sign given by $\epsilon$ (scan $\partial B_0$ counterclockwise starting from $*$).
We define $\mathcal{TD}(\epsilon)$ to be the following category:
\begin{itemize}
\item Objects are knotted web diagrams in $B_0$ with boundary precisely equal to $\mathbf{x}$.
\item A morphism from $w_0$ to $w_1$ is a movie starting with $w_0$ and ending with $w_1$.
\item The composition is the concatenation of movies. The movie of length one is the identity.
\end{itemize}
This yields an $I_3$-canopolis $\mathcal{TD}$, where $\mathcal{TD}_{\epsilon} = \mathcal{TD}(\epsilon)$. The functors associated with input-diagrams are given by gluing movies (as we allow only one local transformation between two consecutive frames, one should actually proceed carefully but the relation ``far modifications commutes'' ensures that it is well-defined).
\end{dfn}

\begin{rmk}
  \label{rmk:generators}
  The canopolis $\mathcal{TD}$ is clearly objects-generated by:
\[
\begin{tikzpicture}[scale=0.5]
  \begin{scope}[decoration={markings, mark=at
     position 0.5 with {\arrow{>}}},postaction={decorate}]
\draw[very thin] (0,0) circle (1cm);
\draw[postaction={decorate}] (0,0) --(90:1);
\draw[postaction={decorate}] (0,0) --(-30:1);
\draw[postaction={decorate}] (0,0) --(210:1);
\fill[red] (-90:1) circle (2pt);
\node at (1,-0.2) {$\,\ ,$};
\end{scope}

\begin{scope}[xshift= 3cm, decoration={markings, mark=at
     position 0.5 with {\arrow{<}}},postaction={decorate}]
\draw[very thin] (0,0) circle (1cm);
\draw[postaction={decorate}] (0,0) --(90:1);
\draw[postaction={decorate}] (0,0) --(-30:1);
\draw[postaction={decorate}] (0,0) --(210:1);
\fill[red] (-90:1) circle (2pt);
\node at (1, -0.20) {$\ \,,$};
\end{scope}

\begin{scope}[xshift= 6cm, decoration={markings, mark=at
     position 0.5 with {\arrow{<}}},postaction={decorate}]
\draw[very thin] (0,0) circle (1cm);
\draw[->] (-135:1) --(45:1);
\fill[white] (0,0) circle (0.2);
\draw[->] (-45:1) --( 135:1);
\fill[red] (-90:1) circle (2pt);
\end{scope}

\node at (8,0) {\textrm{and}};

\begin{scope}[xshift= 10cm, decoration={markings, mark=at
     position 0.5 with {\arrow{<}}},postaction={decorate}]
\draw[very thin] (0,0) circle (1cm);
\draw[->] (-45:1) --(135:1);
\fill[white] (0,0) circle (0.2);
\draw[->] (-135:1) --( 45:1);
\fill[red] (-90:1) circle (2pt);
\node at (1, -0.2) {$\ \,.$};
\end{scope}
\end{tikzpicture}
\]
It is morphisms-generated by the elementary movies movie given in definition~\ref{dfn:movie}.
\end{rmk}

\subsubsection{The morphism}
\label{sec:morphism}

We now want to define a morphism $\mathcal{G}:\mathcal{TD} \To \mathcal{KW}$. All the ingredients are in \cite{MR2482322}, we do not repeat all the definitions, but rather give the flavor of this morphism and explain the needed modifications. 

\emph{The morphism $\mathcal{G}$ on objects.} Thanks to remark~\ref{rmk:generators}, we only need to give the value of $\G$ on four objects. We set:
\[
\G\left(\linkposcrosscirclebase[0.4]\right)=  \linktwovertcirclebase[0.4]\cdot q^{-\frac23}\left[-\frac23\right]\stackrel{\foamzip[0.2]}{\xrightarrow{\hspace{1.5cm}}} \webIvertcirclebase[0.4]\cdot q^{\frac{1}3}\left[\frac13\right], \]\[
\G\left(\linknegcrosscirclebase[0.4]\right)=  \webIvertcirclebase[0.4]\cdot q^{-\frac13}\left[-\frac13\right]\stackrel{\foamunzip[0.2]}{\xrightarrow{\hspace{1.5cm}}} \linktwovertcirclebase[0.4]\cdot q^{\frac23}\left[\frac23\right], \] \[
\G\left(\websinkbase[0.56]\right) =\websinkbase[0.56] [0] \quad  \textrm{and} \quad
\G\left(\websourcebase[0.56]\right) =\websourcebase[0.56] [0],
\]
where $[\bullet]$ indicates the homological degree. 
\emph{Definition of $\mathcal{G}$ on morphism.} A brief description of the image of morphisms by $\G$ is given in figure~\ref{fig:description-g}.
\noindent
\begin{minipage}{\textwidth}
  \centering
\begin{tabular}[p]{|@{}p{0.3cm}|c|p{2.45cm}||c|p{2.45cm}|@{}p{0.3cm}|}
\hline
&$m$ & $\G(m)$  &$m$ & $\G(m)$ & \\ \hline 
A1 &      $\vcenter{\hbox{\begin{tikzpicture}[scale=0.3]
        \input{\imagesfolder/c3_moviecap}
      \end{tikzpicture}}}$ & $\foamcap[0.3]$ (cap)&
      $\vcenter{\hbox{\begin{tikzpicture}[scale=0.3]
        \input{\imagesfolder/c3_moviecup}
      \end{tikzpicture}}}$ & $\foamcup[0.3]$ (cup)&B1 \\ \hline
A2 &      $\vcenter{\hbox{\begin{tikzpicture}[scale=0.3]
        \input{\imagesfolder/c3_moviesaddle}
      \end{tikzpicture}}}$ & $\foamsaddle[0.2]$ (saddle)&
      $\vcenter{\hbox{\begin{tikzpicture}[scale=0.3]
        \input{\imagesfolder/c3_zipa}
      \end{tikzpicture}}}$ &$\foamzip[0.2]$ (zip) &B2\\ \hline 
A3 &      $\vcenter{\hbox{\begin{tikzpicture}[scale=0.3]
        \input{\imagesfolder/c3_zipb}
      \end{tikzpicture}}}$ & $\foamunzip[0.2]$ (unzip)&
      $\vcenter{\hbox{\begin{tikzpicture}[scale=0.3]
        \input{\imagesfolder/c3_diga}
      \end{tikzpicture}}}$ &$\foamdigun[0.19]$ {\scriptsize (digon cup)} &B3 \\ \hline
A4 &      $\vcenter{\hbox{\begin{tikzpicture}[scale=0.3]
        \input{\imagesfolder/c3_digb}
      \end{tikzpicture}}}$ & $\foamgidun[0.19]$ {\scriptsize (digon cap)}&
      $\vcenter{\hbox{\begin{tikzpicture}[scale=0.3]
        \input{\imagesfolder/c3_R1a}
      \end{tikzpicture}}}$ & {\tiny \vspace{-0.7cm} Homotopy equivalence\footnotemark[1] , composition of R1a and R1b from \cite{MR2482322}, inverse of A5. \vspace{-0.5cm}\par} &B4 \\
\hline
\end{tabular}
  \captionof{figure}{Description of $\G$.}
  \label{fig:description-g}
\footnotetext{$^1$ With our normalization, R1a and R1b have to be shifted both in $q$ and in homological grading, however the composition of the two is degree preserving. }
\footnotetext{$^2$ In \cite{MR2482322}, there are some degree shifts, with our normalization, there is no.}
\end{minipage}


\begin{rmk}
  \label{rmk:freecano} The canopolis $\mathcal{TD}$ does not see the topology of diagrams (for example we did not impose that movies associated with Reidemeister moves correspond to isomorphisms). The topology will be recovered by saying that the morphism $\G$ is invariant under certain moves. This is the purpose of the next section.
\end{rmk}


\subsection{Invariance results}
\label{sec:invariance-result}

\begin{dfn}\label{dfn:framed-equivalent}
  Two knotted web  diagrams $w_1$ and $w_2$ of $\mathcal{KW}_\epsilon$ are \emph{framed equivalent} if there exists a morphism from $w_1$ to $w_2$ which is a composition of elementary movies of type A5-A9 and B4-B9.
\end{dfn}

\begin{rmk}
  We do not claim that this is a ``good'' definition of framed equivalence for webs. The ``good'' definition should be geometric and deal with normal vector field and ambient isotopy equivalence. We do not claim that this ``cheap'' definition is equivalent to the geometric one. 
\end{rmk}
From the definition of $\G$  we have the obvious lemma:
\begin{lem}\label{lem:framedequivalence}
  Let $w_1$ and $w_2$ are two knotted web diagrams. If $w_1$ and $w_2$ are framed equivalent then the two complexes $\G(w_1)$ and $\G(w_2)$ are homotopy equivalent. 
\end{lem} 

Let $w_1$ and $w_2$ be two tangle diagrams and $\Sigma$ a framed cobordism (in the classical sense, i. e. $\Sigma$ is a surface) between $w_1$ and $w_2$. Then $\Sigma$ (or a surface isotopic to $\Sigma$) can be  presented by a movie. We have the following theorem:

\begin{thm}[Clark's functoriality theorem, \cite{MR2482322}]
  \label{thm:clark} With the same notations, $\G(\Sigma)$ depends only on the isotopy type of $\Sigma$.
\end{thm}

\begin{rmk}
 \label{rmk:noncanonical}
  The isomorphism between $H(\G(w_1))$ and $H(\G(w_2))$ in lemma~\ref{lem:framedequivalence} is \emph{a priori} not canonical, however we can interpret the Clark's functoriality theorem by saying that if $w_1$ and $w_2$ are link diagrams, then $\G(w_1)$ and $\G(w_2)$ are canonically homotopy equivalent.
\end{rmk} 

\begin{lem}[{\cite[Lemma 3.12]{MR2482322}}]\label{lem:zipcrossing}
  The images by $\G$ of the following two morphisms are equal:
\begin{center}
\vspace{0.5cm}
$\vcenter{\hbox{\begin{tikzpicture}[scale=0.42]
      \begin{scope}
\draw[draw=  gray, line width = 2mm] (0,0) --  (12,0); 
\draw [draw = gray, line width = 2mm] (0,4) -- (12,4); 
\draw[draw=  gray, very thick] (0,0) -- (0,4);
\draw[draw=  gray, very thick] (4,0) -- (4,4);
\draw[draw=  gray, very thick] (8,0) -- (8,4);
\draw[draw=  gray, very thick] (12,0) -- (12,4);
\draw[draw= white, dotted, line width =1.2mm] (0,0) -- (12,0);
\draw[draw= white, dotted, line width = 1.2mm] (0,4) -- (12,4);
\end{scope}
\begin{scope}[xshift= 2cm, yshift=2cm]
  \filldraw[dotted, fill=gray!50!white] (0,0) circle (1.5cm);
  \filldraw[dotted, fill = white] (0, 0) circle (1cm);
\draw (180:1) .. controls (0,1) and (0,1) .. (0:1);
\fill[white] (-0.38, 0.56) circle (0.1);
\fill[white] (0.38, 0.56) circle (0.1);
\draw (60:1).. controls +(-120:0.5) and +(120:0.5) .. (-60:1);
\draw (120:1).. controls +(-60:0.5) and +(60:0.5) .. (-120:1);
\end{scope}
\begin{scope}[xshift= 6cm, yshift=2cm]
  \filldraw[dotted, fill=gray!50!white] (0,0) circle (1.5cm);
  \filldraw[dotted, fill = white] (0, 0) circle (1cm);
\draw (180:1) .. controls (0,1) and (0,1) .. (0:1);
\fill[white] (-0.25, 0.67) circle (0.1);
\fill[white] (0.25, 0.67) circle (0.1);
\draw (60:1) -- (90:0.5) -- (90:-0.5) -- (60:-1);
\draw (90:0.5) -- (120:1);
\draw (90:-0.5) -- (120:-1);
\end{scope}
\begin{scope}[xshift= 10cm, yshift=2cm]
  \filldraw[dotted, fill=gray!50!white] (0,0) circle (1.5cm);
  \filldraw[dotted, fill = white] (0, 0) circle (1cm);
\draw (180:1) .. controls (0,0) and (0,0) .. (0:1);
\fill[white] (0, 0) circle (0.1);
\draw (60:1) -- (90:0.5) -- (90:-0.5) -- (60:-1);
\draw (90:0.5) -- (120:1);
\draw (90:-0.5) -- (120:-1);
\end{scope} 
    \end{tikzpicture}}}$\quad and \quad
$\vcenter{\hbox{ \begin{tikzpicture}[scale=0.42]
      \begin{scope}
\draw[draw=  gray, line width = 2mm] (0,0) --  (12,0); 
\draw [draw = gray, line width = 2mm] (0,4) -- (12,4); 
\draw[draw=  gray, very thick] (0,0) -- (0,4);
\draw[draw=  gray, very thick] (4,0) -- (4,4);
\draw[draw=  gray, very thick] (8,0) -- (8,4);
\draw[draw=  gray, very thick] (12,0) -- (12,4);
\draw[draw= white, dotted, line width =1.2mm] (0,0) -- (12,0);
\draw[draw= white, dotted, line width = 1.2mm] (0,4) -- (12,4);
\end{scope}
\begin{scope}[xshift= 2cm, yshift=2cm]
  \filldraw[dotted, fill=gray!50!white] (0,0) circle (1.5cm);
  \filldraw[dotted, fill = white] (0, 0) circle (1cm);
\draw (180:1) .. controls (0,-1) and (0,-1) .. (0:1);
\fill[white] (-0.38, -0.56) circle (0.1);
\fill[white] (0.38, -0.56) circle (0.1);
\draw (60:1).. controls +(-120:0.5) and +(120:0.5) .. (-60:1);
\draw (120:1).. controls +(-60:0.5) and +(60:0.5) .. (-120:1);
\end{scope}
\begin{scope}[xshift= 6cm, yshift=2cm]
  \filldraw[dotted, fill=gray!50!white] (0,0) circle (1.5cm);
  \filldraw[dotted, fill = white] (0, 0) circle (1cm);
\draw (180:1) .. controls (0,-1) and (0,-1) .. (0:1);
\fill[white] (-0.25, -0.67) circle (0.1);
\fill[white] (0.25, -0.67) circle (0.1);
\draw (60:1) -- (90:0.5) -- (90:-0.5) -- (60:-1);
\draw (90:0.5) -- (120:1);
\draw (90:-0.5) -- (120:-1);
\end{scope}
\begin{scope}[xshift= 10cm, yshift=2cm]
  \filldraw[dotted, fill=gray!50!white] (0,0) circle (1.5cm);
  \filldraw[dotted, fill = white] (0, 0) circle (1cm);
\draw (180:1) .. controls (0,0) and (0,0) .. (0:1);
\fill[white] (0, 0) circle (0.1);
\draw (60:1) -- (90:0.5) -- (90:-0.5) -- (60:-1);
\draw (90:0.5) -- (120:1);
\draw (90:-0.5) -- (120:-1);
\end{scope}
    \end{tikzpicture}}}$.
\vspace{0.5cm}
  \end{center}
The images by $\G$ of the following two morphisms are equal:
  \begin{center}
\vspace{0.5cm}
$\vcenter{\hbox{ \begin{tikzpicture}[scale=0.42]
      \begin{scope}
\draw[draw=  gray, line width = 2mm] (0,0) --  (12,0); 
\draw [draw = gray, line width = 2mm] (0,4) -- (12,4); 
\draw[draw=  gray, very thick] (0,0) -- (0,4);
\draw[draw=  gray, very thick] (4,0) -- (4,4);
\draw[draw=  gray, very thick] (8,0) -- (8,4);
\draw[draw=  gray, very thick] (12,0) -- (12,4);
\draw[draw= white, dotted, line width =1.2mm] (0,0) -- (12,0);
\draw[draw= white, dotted, line width = 1.2mm] (0,4) -- (12,4);
\end{scope}
\begin{scope}[xshift= 2cm, yshift=2cm]
  \filldraw[dotted, fill=gray!50!white] (0,0) circle (1.5cm);
  \filldraw[dotted, fill = white] (0, 0) circle (1cm);
\draw (60:1).. controls +(-120:0.5) and +(120:0.5) .. (-60:1);
\draw (120:1).. controls +(-60:0.5) and +(60:0.5) .. (-120:1);
\fill[white] (-0.38, 0.56) circle (0.1);
\fill[white] (0.38, 0.56) circle (0.1);
\draw (180:1) .. controls (0,1) and (0,1) .. (0:1);
\end{scope}
\begin{scope}[xshift= 6cm, yshift=2cm]
  \filldraw[dotted, fill=gray!50!white] (0,0) circle (1.5cm);
  \filldraw[dotted, fill = white] (0, 0) circle (1cm);
\draw (60:1) -- (90:0.5) -- (90:-0.5) -- (60:-1);
\draw (90:0.5) -- (120:1);
\draw (90:-0.5) -- (120:-1);
\fill[white] (-0.25, 0.67) circle (0.1);
\fill[white] (0.25, 0.67) circle (0.1);
\draw (180:1) .. controls (0,1) and (0,1) .. (0:1);
\end{scope}
\begin{scope}[xshift= 10cm, yshift=2cm]
  \filldraw[dotted, fill=gray!50!white] (0,0) circle (1.5cm);
  \filldraw[dotted, fill = white] (0, 0) circle (1cm);
\draw (60:1) -- (90:0.5) -- (90:-0.5) -- (60:-1);
\draw (90:0.5) -- (120:1);
\draw (90:-0.5) -- (120:-1);
\fill[white] (0, 0) circle (0.1);
\draw (180:1) .. controls (0,0) and (0,0) .. (0:1);
\end{scope} 
    \end{tikzpicture}}}$\quad and \quad
$\vcenter{\hbox{\begin{tikzpicture}[scale=0.42]
      \begin{scope}
\draw[draw=  gray, line width = 2mm] (0,0) --  (12,0); 
\draw [draw = gray, line width = 2mm] (0,4) -- (12,4); 
\draw[draw=  gray, very thick] (0,0) -- (0,4);
\draw[draw=  gray, very thick] (4,0) -- (4,4);
\draw[draw=  gray, very thick] (8,0) -- (8,4);
\draw[draw=  gray, very thick] (12,0) -- (12,4);
\draw[draw= white, dotted, line width =1.2mm] (0,0) -- (12,0);
\draw[draw= white, dotted, line width = 1.2mm] (0,4) -- (12,4);
\end{scope}
\begin{scope}[xshift= 2cm, yshift=2cm]
  \filldraw[dotted, fill=gray!50!white] (0,0) circle (1.5cm);
  \filldraw[dotted, fill = white] (0, 0) circle (1cm);
\draw (60:1).. controls +(-120:0.5) and +(120:0.5) .. (-60:1);
\draw (120:1).. controls +(-60:0.5) and +(60:0.5) .. (-120:1);
\fill[white] (-0.38, -0.56) circle (0.1);
\fill[white] (0.38, -0.56) circle (0.1);
\draw (180:1) .. controls (0,-1) and (0,-1) .. (0:1);
\end{scope}
\begin{scope}[xshift= 6cm, yshift=2cm]
  \filldraw[dotted, fill=gray!50!white] (0,0) circle (1.5cm);
  \filldraw[dotted, fill = white] (0, 0) circle (1cm);
\draw (60:1) -- (90:0.5) -- (90:-0.5) -- (60:-1);
\draw (90:0.5) -- (120:1);
\draw (90:-0.5) -- (120:-1);
\fill[white] (-0.25, -0.67) circle (0.1);
\fill[white] (0.25, -0.67) circle (0.1);
\draw (180:1) .. controls (0,-1) and (0,-1) .. (0:1);
\end{scope}
\begin{scope}[xshift= 10cm, yshift=2cm]
  \filldraw[dotted, fill=gray!50!white] (0,0) circle (1.5cm);
  \filldraw[dotted, fill = white] (0, 0) circle (1cm);
\draw (60:1) -- (90:0.5) -- (90:-0.5) -- (60:-1);
\draw (90:0.5) -- (120:1);
\draw (90:-0.5) -- (120:-1);
\fill[white] (0, 0) circle (0.1);
\draw (180:1) .. controls (0,0) and (0,0) .. (0:1);
\end{scope}
    \end{tikzpicture}}}$.
\vspace{0.5cm}
  \end{center}
\end{lem}

Exactly the same argument provides the following lemma:

\begin{lem}\label{lem:zipdigon}
  The images by $\G$ of the following two morphisms are equal:
\begin{center}
\vspace{0.5cm}
$\vcenter{\hbox{\begin{tikzpicture}[scale=0.42]
      \begin{scope}
\draw[draw=  gray, line width = 2mm] (0,0) --  (12,0); 
\draw [draw = gray, line width = 2mm] (0,4) -- (12,4); 
\draw[draw=  gray, very thick] (0,0) -- (0,4);
\draw[draw=  gray, very thick] (4,0) -- (4,4);
\draw[draw=  gray, very thick] (8,0) -- (8,4);
\draw[draw=  gray, very thick] (12,0) -- (12,4);
\draw[draw= white, dotted, line width =1.2mm] (0,0) -- (12,0);
\draw[draw= white, dotted, line width = 1.2mm] (0,4) -- (12,4);
\end{scope}
\begin{scope}[xshift= 2cm, yshift=2cm]
  \filldraw[dotted, fill=gray!50!white] (0,0) circle (1.5cm);
  \filldraw[dotted, fill = white] (0, 0) circle (1cm);
\draw (180:1) --(0:1);
\fill[white] (-0.38, 0) circle (0.1);
\fill[white] (0.38, 0) circle (0.1);
\draw (90:1) -- (90:0.5).. controls (180:0.5) and (180:0.5) .. (90:-0.5) -- (90:-1);
\draw (90:0.5).. controls (180:-0.5) and (180:-0.5) .. (90:-0.5);

\end{scope}
\begin{scope}[xshift= 6cm, yshift=2cm]
  \filldraw[dotted, fill=gray!50!white] (0,0) circle (1.5cm);
  \filldraw[dotted, fill = white] (0, 0) circle (1cm);
\draw (180:1) .. controls (0,1) and (0,1) .. (0:1);
\fill[white] (0, 0.8) circle (0.1);
\draw (90:1) -- (90:0.5).. controls (180:0.5) and (180:0.5) .. (90:-0.5) -- (90:-1);
\draw (90:0.5).. controls (180:-0.5) and (180:-0.5) .. (90:-0.5);
\end{scope}
\begin{scope}[xshift= 10cm, yshift=2cm]
  \filldraw[dotted, fill=gray!50!white] (0,0) circle (1.5cm);
  \filldraw[dotted, fill = white] (0, 0) circle (1cm);
\draw (180:1) .. controls (0,0) and (0,0) .. (0:1);
\fill[white] (0, 0) circle (0.1);
\draw (90:-1) -- (90:1);
\end{scope} 
    \end{tikzpicture}}}$\quad and \quad
$\vcenter{\hbox{ \begin{tikzpicture}[scale=0.42]
      \begin{scope}
\draw[draw=  gray, line width = 2mm] (0,0) --  (12,0); 
\draw [draw = gray, line width = 2mm] (0,4) -- (12,4); 
\draw[draw=  gray, very thick] (0,0) -- (0,4);
\draw[draw=  gray, very thick] (4,0) -- (4,4);
\draw[draw=  gray, very thick] (8,0) -- (8,4);
\draw[draw=  gray, very thick] (12,0) -- (12,4);
\draw[draw= white, dotted, line width =1.2mm] (0,0) -- (12,0);
\draw[draw= white, dotted, line width = 1.2mm] (0,4) -- (12,4);
\end{scope}
\begin{scope}[xshift= 2cm, yshift=2cm]
  \filldraw[dotted, fill=gray!50!white] (0,0) circle (1.5cm);
  \filldraw[dotted, fill = white] (0, 0) circle (1cm);
\draw (180:1) --(0:1);
\fill[white] (-0.38, 0) circle (0.1);
\fill[white] (0.38, 0) circle (0.1);
\draw (90:1) -- (90:0.5).. controls (180:0.5) and (180:0.5) .. (90:-0.5) -- (90:-1);
\draw (90:0.5).. controls (180:-0.5) and (180:-0.5) .. (90:-0.5);
\end{scope}
\begin{scope}[xshift= 6cm, yshift=2cm]
  \filldraw[dotted, fill=gray!50!white] (0,0) circle (1.5cm);
  \filldraw[dotted, fill = white] (0, 0) circle (1cm);
\draw (180:1) .. controls (0,-1) and (0,-1) .. (0:1);
\fill[white] (0, -0.75) circle (0.1);
\draw (90:1) -- (90:0.5).. controls (180:0.5) and (180:0.5) .. (90:-0.5) -- (90:-1);
\draw (90:0.5).. controls (180:-0.5) and (180:-0.5) .. (90:-0.5);
\end{scope}
\begin{scope}[xshift= 10cm, yshift=2cm]
  \filldraw[dotted, fill=gray!50!white] (0,0) circle (1.5cm);
  \filldraw[dotted, fill = white] (0, 0) circle (1cm);
\draw (180:1) .. controls (0,0) and (0,0) .. (0:1);
\fill[white] (0, 0) circle (0.1);
\draw (90:-1) -- (90:1);

\end{scope}
    \end{tikzpicture}}}$.
\vspace{0.5cm}
  \end{center}
The images by $\G$ of the following two morphisms are equal:
  \begin{center}
\vspace{0.5cm}
$\vcenter{\hbox{ \begin{tikzpicture}[scale=0.42]
      \begin{scope}
\draw[draw=  gray, line width = 2mm] (0,0) --  (12,0); 
\draw [draw = gray, line width = 2mm] (0,4) -- (12,4); 
\draw[draw=  gray, very thick] (0,0) -- (0,4);
\draw[draw=  gray, very thick] (4,0) -- (4,4);
\draw[draw=  gray, very thick] (8,0) -- (8,4);
\draw[draw=  gray, very thick] (12,0) -- (12,4);
\draw[draw= white, dotted, line width =1.2mm] (0,0) -- (12,0);
\draw[draw= white, dotted, line width = 1.2mm] (0,4) -- (12,4);
\end{scope}
\begin{scope}[xshift= 2cm, yshift=2cm]
  \filldraw[dotted, fill=gray!50!white] (0,0) circle (1.5cm);
  \filldraw[dotted, fill = white] (0, 0) circle (1cm);
\draw (90:1) -- (90:0.5).. controls (180:0.5) and (180:0.5) .. (90:-0.5) -- (90:-1);
\draw (90:0.5).. controls (180:-0.5) and (180:-0.5) .. (90:-0.5);
\fill[white] (-0.38, 0) circle (0.1);
\fill[white] (0.38, 0) circle (0.1);
\draw (180:1) --(0:1);
\end{scope}
\begin{scope}[xshift= 6cm, yshift=2cm]
  \filldraw[dotted, fill=gray!50!white] (0,0) circle (1.5cm);
  \filldraw[dotted, fill = white] (0, 0) circle (1cm);
\draw (90:1) -- (90:0.5).. controls (180:0.5) and (180:0.5) .. (90:-0.5) -- (90:-1);
\draw (90:0.5).. controls (180:-0.5) and (180:-0.5) .. (90:-0.5);
\fill[white] (0, 0.75) circle (0.1);
\draw (180:1) .. controls (0,1) and (0,1) .. (0:1);
\end{scope}
\begin{scope}[xshift= 10cm, yshift=2cm]
  \filldraw[dotted, fill=gray!50!white] (0,0) circle (1.5cm);
  \filldraw[dotted, fill = white] (0, 0) circle (1cm);
\draw (90:-1) -- (90:1);
\fill[white] (0, 0) circle (0.1);
\draw (180:1) .. controls (0,0) and (0,0) .. (0:1);
\end{scope} 
    \end{tikzpicture}}}$\quad and \quad
$\vcenter{\hbox{\begin{tikzpicture}[scale=0.42]
      \begin{scope}
\draw[draw=  gray, line width = 2mm] (0,0) --  (12,0); 
\draw [draw = gray, line width = 2mm] (0,4) -- (12,4); 
\draw[draw=  gray, very thick] (0,0) -- (0,4);
\draw[draw=  gray, very thick] (4,0) -- (4,4);
\draw[draw=  gray, very thick] (8,0) -- (8,4);
\draw[draw=  gray, very thick] (12,0) -- (12,4);
\draw[draw= white, dotted, line width =1.2mm] (0,0) -- (12,0);
\draw[draw= white, dotted, line width = 1.2mm] (0,4) -- (12,4);
\end{scope}
\begin{scope}[xshift= 2cm, yshift=2cm]
  \filldraw[dotted, fill=gray!50!white] (0,0) circle (1.5cm);
  \filldraw[dotted, fill = white] (0, 0) circle (1cm);
\draw (90:1) -- (90:0.5).. controls (180:0.5) and (180:0.5) .. (90:-0.5) -- (90:-1);
\draw (90:0.5).. controls (180:-0.5) and (180:-0.5) .. (90:-0.5);
\fill[white] (-0.38, 0) circle (0.1);
\fill[white] (0.38, 0) circle (0.1);
\draw (180:1) --(0:1);
\end{scope}
\begin{scope}[xshift= 6cm, yshift=2cm]
  \filldraw[dotted, fill=gray!50!white] (0,0) circle (1.5cm);
  \filldraw[dotted, fill = white] (0, 0) circle (1cm);
\draw (90:1) -- (90:0.5).. controls (180:0.5) and (180:0.5) .. (90:-0.5) -- (90:-1);
\draw (90:0.5).. controls (180:-0.5) and (180:-0.5) .. (90:-0.5);
\fill[white] (0, -0.75) circle (0.1);
\draw (180:1) .. controls (0,-1) and (0,-1) .. (0:1);
\end{scope}
\begin{scope}[xshift= 10cm, yshift=2cm]
  \filldraw[dotted, fill=gray!50!white] (0,0) circle (1.5cm);
  \filldraw[dotted, fill = white] (0, 0) circle (1cm);
\draw (90:-1) -- (90:1);
\fill[white] (0, 0) circle (0.1);
\draw (180:1) .. controls (0,0) and (0,0) .. (0:1);
\end{scope}
    \end{tikzpicture}}}$.
\vspace{0.5cm}
  \end{center}
\end{lem}

To prove more invariance results, we will use Bar-Natan's trick \cite{MR2174270} in order to avoid long and complicated computations. This trick yields ``up-to-a-sign'' invariance results, but this is enough for our purposes.

\begin{dfn}[Bar-Natan, adapted]
A knotted web diagram $w$ is \emph{BN-simple} if $\G(w)$ admits only two automorphisms (\ie homotopy equivalences): $+\id$ and $-\id$.
\end{dfn}

\begin{lem}\label{lem:treesimple}
  If a knotted web diagram $w$ is an unknotted collection of trees, then $w$ is BN-simple
\end{lem}
Of course we only consider trivalent trees. Note that any tree can be oriented in such a manner that it becomes a web.
\begin{proof}
The knotted web diagram $w$ having no crossing, $\G(w)$ is concentrated in degree zero. If $w$ consists of $t$ trees, and has $v$ vertices. One easily shows (thanks to remarks~\ref{rmk:web2sl3} and \ref{rmk:dimHOM}) that $\END{\G(w)}\simeq \ZZ\cdot(q^{2t+v}[3]^t[2]^v)$. 
This proves that $\mathrm{END}_0{\G(w)}\simeq \ZZ$, and finally that the only possible automorphisms of $\G(w)$ are $+\id$ and $-\id$.
\end{proof}

\begin{lem}[Bar-Natan, adapted]\label{lem:fe2simple}
  If $w_1$ and $w_2$ are framed equivalent, then $w_1$ is BN-simple if and only if $w_2$ is BN-simple.
\end{lem}

\begin{proof}
  If $w_1$ and $w_2$ are framed equivalent, then $\G(w_1)$ and $\G(w_2)$ are homotopy equivalent. 
Hence $\END{\G{w_1}}$ and $\END{\G{w_2}}$ are isomorphic, this proves the statement.
\end{proof}

\begin{cor}\label{cor:eqBNsimple} Let $w_1$ and $w_2$ two BN-simple framed equivalent web diagrams. Suppose that $h_1$ and $h_2$ two homotopy equivalences between $\G(w_1)$ and $\G(w_2)$. Then $h_1=\pm h_2$.
\end{cor}

\begin{lem}[Bar-Natan, {\cite[Lemmas 8.8 and 8.9]{MR2174270}}] \label{lem:BNX}
  Let $w$ be a knotted web diagram, and let $wX$ be a knotted web diagram obtained from $w$ by adding a crossing $X$ (positive or negative) somewhere on the boundary of $w$ so that exactly two (adjacent) legs of $X$ are connected to $w$ and two remain free. Then $w$ is BN-simple if and only if $wX$ is BN-simple.
\end{lem}
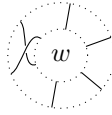
\begin{figure}[ht]
  \centering
  \begin{tikzpicture}[scale=0.7]
    \begin{scope}
  \draw (20: 1) -- (0,0);
  \draw (80: 1) -- (0,0);
  \draw (140: 1) .. controls (140:0.8) and (200:0.8).. (200:0.5);
  \fill[white] (170: 0.67) circle (2pt);
  \draw (140: 0.5) .. controls (140:0.8) and (200:0.8).. (200:1);
  \draw (260: 1) -- (0,0);
  \draw (320: 1) -- (0,0);
  \filldraw[fill =white, dotted] (0,0) circle (0.5cm);
  \draw[dotted] (0,0) circle (1cm);
  \node at (0,0) {$w$};
\end{scope}
  \end{tikzpicture}
  \caption{The web $wX$ is obtained by gluing a crossing on two ends of $w$. }
  \label{fig:wX}
\end{figure}

From lemmas \ref{lem:treesimple}, \ref{lem:fe2simple}, \ref{lem:BNX} and corollary~\ref{cor:eqBNsimple}, we deduce:

\begin{prop}
  \label{prop:importantmoves}
The images of the following pairs of morphisms are equal up to a sign two by two:
\[
\vspace{0.3cm}
\begin{tikzpicture}[scale=0.4] 
   \begin{scope}
\begin{scope}
\draw[draw=  gray, line width = 2mm] (0,0) --  (12,0); 
\draw [draw = gray, line width = 2mm] (12,3) -- (0,3); 
\draw[draw=  gray, very thick] (0,0) -- (0,3);
\draw[draw=  gray, very thick] (4,0) -- (4,3);
\draw[draw=  gray, very thick] (8,0) -- (8,3);
\draw[draw=  gray, very thick] (12,0) -- (12,3);
\draw[draw= white, dotted, line width =1.2mm] (0,0) -- (12,0);
\draw[draw= white, dotted, line width = 1.2mm] (0,3) -- (12,3);
\end{scope}
\begin{scope}
\coordinate (D) at (3.5, 1.5);
\coordinate (A) at (0.5, 2.5);
\coordinate (B) at (0.5, 1.5);
\coordinate (C) at (0.5, 0.5);
 \draw (A) .. controls +( 0.5,0) and +(0,0.5) .. (D);
 \draw (B) .. controls +( 0.5,0) and +(-0.5,0) .. (D);
 \draw (C) .. controls +( 0.5,0) and +(0,-0.5) .. (D);
 \draw[draw= white, line width=1mm] (2, 2.5) .. controls +(0, 0)  and +(0,0) .. (2, 0.5);
 \draw[draw= black] (2, 2.5) .. controls +(0, 0)  and +(0,0) .. (2, 0.5);
\end{scope}
\begin{scope}[xshift= 4cm]
\coordinate (D) at (1.5, 1);
\coordinate (A) at (0.5, 2.5);
\coordinate (B) at (0.5, 1.5);
\coordinate (C) at (0.5, 0.5);
 \draw (D) .. controls +(2.5, 0) and +(0.5,0) .. (A);
 \draw (D) .. controls +(-0.5, 0.5) and +(0.5,0) .. (B);
 \draw (D) .. controls +(-0.5, -0.5) and +(0.5,0) .. (C);
 \draw[draw= white, line width=1mm] (2, 2.5) .. controls +(0, 0)  and +(0,0) .. (2, 0.5);
 \draw[draw= black, ] (2, 2.5) .. controls +(0, 0)  and +(0,0) .. (2, 0.5);
\end{scope}
\begin{scope}[xshift=8cm]
\coordinate (D) at (1.5, 1.5);
\coordinate (A) at (0.5, 2.5);
\coordinate (B) at (0.5, 1.5);
\coordinate (C) at (0.5, 0.5);
 \draw (A) .. controls +( 0.5,0) and +(0,0.5) .. (D);
 \draw (B) .. controls +( 0.5,0) and +(-0.5,0) .. (D);
 \draw (C) .. controls +( 0.5,0) and +(0,-0.5) .. (D);
 \draw[draw= white, line width=1mm] (2, 2.5) .. controls  +(0, 0)  and +(0,0) .. (2, 0.5);
 \draw[draw= black] (2, 2.5) .. controls +(0, 0)  and +(0,0) .. (2, 0.5);
\end{scope}
\end{scope}
\node at (14, 1.5) {$\leftrightsquigarrow$};
\begin{scope}[xshift= 16cm]
\begin{scope}
\draw[draw=  gray, line width = 2mm] (0,0) --  (12,0); 
\draw [draw = gray, line width = 2mm] (12,3) -- (0,3); 
\draw[draw=  gray, very thick] (0,0) -- (0,3);
\draw[draw=  gray, very thick] (4,0) -- (4,3);
\draw[draw=  gray, very thick] (8,0) -- (8,3);
\draw[draw=  gray, very thick] (12,0) -- (12,3);
\draw[draw= white, dotted, line width =1.2mm] (0,0) -- (12,0);
\draw[draw= white, dotted, line width = 1.2mm] (0,3) -- (12,3);
\end{scope}
\begin{scope}
\coordinate (D) at (3.5, 1.5);
\coordinate (A) at (0.5, 2.5);
\coordinate (B) at (0.5, 1.5);
\coordinate (C) at (0.5, 0.5);
 \draw (A) .. controls +( 0.5,0) and +(0,0.5) .. (D);
 \draw (B) .. controls +( 0.5,0) and +(-0.5,0) .. (D);
 \draw (C) .. controls +( 0.5,0) and +(0,-0.5) .. (D);
 \draw[draw= white, line width=1mm] (2, 2.5) .. controls +(0, 0)  and +(0,0) .. (2, 0.5);
 \draw[draw= black] (2, 2.5) .. controls +(0, 0)  and +(0,0) .. (2, 0.5);
\end{scope}
\begin{scope}[xshift= 4cm]
\coordinate (D) at (1.5, 2);
\coordinate (A) at (0.5, 2.5);
\coordinate (B) at (0.5, 1.5);
\coordinate (C) at (0.5, 0.5);
 \draw (D) .. controls +(2.5, 0) and +(0.5,0) .. (C);
 \draw (D) .. controls +(-0.5, -0.5) and +(0.5,0) .. (B);
 \draw (D) .. controls +(-0.5, +0.5) and +(0.5,0) .. (A);
 \draw[draw= white, line width=1mm] (2, 2.5) .. controls +(0, 0)  and +(0,0) .. (2, 0.5);
 \draw[draw= black, ] (2, 2.5) .. controls +(0, 0)  and +(0,0) .. (2, 0.5);
\end{scope}
\begin{scope}[xshift=8cm]
\coordinate (D) at (1.5, 1.5);
\coordinate (A) at (0.5, 2.5);
\coordinate (B) at (0.5, 1.5);
\coordinate (C) at (0.5, 0.5);
 \draw (A) .. controls +( 0.5,0) and +(0,0.5) .. (D);
 \draw (B) .. controls +( 0.5,0) and +(-0.5,0) .. (D);
 \draw (C) .. controls +( 0.5,0) and +(0,-0.5) .. (D);
 \draw[draw= white, line width=1mm] (2, 2.5) .. controls  +(0, 0)  and +(0,0) .. (2, 0.5);
 \draw[draw= black] (2, 2.5) .. controls +(0, 0)  and +(0,0) .. (2, 0.5);
\end{scope}
\end{scope}
 \end{tikzpicture}\]
\[
\vspace{0.3cm}
\begin{tikzpicture}[scale=0.4] 
   \begin{scope}
\begin{scope}
\draw[draw=  gray, line width = 2mm] (0,0) --  (12,0); 
\draw [draw = gray, line width = 2mm] (12,3) -- (0,3); 
\draw[draw=  gray, very thick] (0,0) -- (0,3);
\draw[draw=  gray, very thick] (4,0) -- (4,3);
\draw[draw=  gray, very thick] (8,0) -- (8,3);
\draw[draw=  gray, very thick] (12,0) -- (12,3);
\draw[draw= white, dotted, line width =1.2mm] (0,0) -- (12,0);
\draw[draw= white, dotted, line width = 1.2mm] (0,3) -- (12,3);
\end{scope}
\begin{scope}
\coordinate (D) at (3.5, 1.5);
\coordinate (A) at (0.5, 2.5);
\coordinate (B) at (0.5, 1.5);
\coordinate (C) at (0.5, 0.5);
 \draw[draw= black] (2, 2.5) .. controls +(0, 0)  and +(0,0) .. (2, 0.5);
 \draw[draw= white, line width=1mm] (A) .. controls +( 0.5,0) and +(0,0.5) .. (D);
 \draw[draw= white, line width=1mm] (B) .. controls +( 0.5,0) and +(-0.5,0) .. (D);
 \draw[draw= white, line width=1mm] (C) .. controls +( 0.5,0) and +(0,-0.5) .. (D);
 \draw (A) .. controls +( 0.5,0) and +(0,0.5) .. (D);
 \draw (B) .. controls +( 0.5,0) and +(-0.5,0) .. (D);
 \draw (C) .. controls +( 0.5,0) and +(0,-0.5) .. (D);
\end{scope}
\begin{scope}[xshift= 4cm]
\coordinate (D) at (1.5, 1);
\coordinate (A) at (0.5, 2.5);
\coordinate (B) at (0.5, 1.5);
\coordinate (C) at (0.5, 0.5);
 \draw[draw= black, ] (2, 2.5) .. controls +(0, 0)  and +(0,0) .. (2, 0.5);
 \draw[draw= white, line width=1mm] (D) .. controls +(2.5, 0) and +(0.5,0) .. (A);
 \draw[draw= white, line width=1mm] (D) .. controls +(-0.5, 0.5) and +(0.5,0) .. (B);
 \draw[draw= white, line width=1mm] (D) .. controls +(-0.5, -0.5) and +(0.5,0) .. (C);
 \draw (D) .. controls +(2.5, 0) and +(0.5,0) .. (A);
 \draw (D) .. controls +(-0.5, 0.5) and +(0.5,0) .. (B);
 \draw (D) .. controls +(-0.5, -0.5) and +(0.5,0) .. (C);
\end{scope}
\begin{scope}[xshift=8cm]
\coordinate (D) at (1.5, 1.5);
\coordinate (A) at (0.5, 2.5);
\coordinate (B) at (0.5, 1.5);
\coordinate (C) at (0.5, 0.5);
 \draw[draw= black] (2, 2.5) .. controls +(0, 0)  and +(0,0) .. (2, 0.5);
 \draw[draw= white, line width=1mm] (A) .. controls +( 0.5,0) and +(0,0.5) .. (D);
 \draw[draw= white, line width=1mm] (B) .. controls +( 0.5,0) and +(-0.5,0) .. (D);
 \draw[draw= white, line width=1mm] (C) .. controls +( 0.5,0) and +(0,-0.5) .. (D);
 \draw (A) .. controls +( 0.5,0) and +(0,0.5) .. (D);
 \draw (B) .. controls +( 0.5,0) and +(-0.5,0) .. (D);
 \draw (C) .. controls +( 0.5,0) and +(0,-0.5) .. (D);
\end{scope}
\end{scope}
\node at (14, 1.5) {$\leftrightsquigarrow$};
\begin{scope}[xshift= 16cm]
\begin{scope}
\draw[draw=  gray, line width = 2mm] (0,0) --  (12,0); 
\draw [draw = gray, line width = 2mm] (12,3) -- (0,3); 
\draw[draw=  gray, very thick] (0,0) -- (0,3);
\draw[draw=  gray, very thick] (4,0) -- (4,3);
\draw[draw=  gray, very thick] (8,0) -- (8,3);
\draw[draw=  gray, very thick] (12,0) -- (12,3);
\draw[draw= white, dotted, line width =1.2mm] (0,0) -- (12,0);
\draw[draw= white, dotted, line width = 1.2mm] (0,3) -- (12,3);
\end{scope}
\begin{scope}
\coordinate (D) at (3.5, 1.5);
\coordinate (A) at (0.5, 2.5);
\coordinate (B) at (0.5, 1.5);
\coordinate (C) at (0.5, 0.5);
 \draw[draw= black] (2, 2.5) .. controls +(0, 0)  and +(0,0) .. (2, 0.5);
 \draw[draw= white, line width=1mm] (A) .. controls +( 0.5,0) and +(0,0.5) .. (D);
 \draw[draw= white, line width=1mm] (B) .. controls +( 0.5,0) and +(-0.5,0) .. (D);
 \draw[draw= white, line width=1mm] (C) .. controls +( 0.5,0) and +(0,-0.5) .. (D);
 \draw (A) .. controls +( 0.5,0) and +(0,0.5) .. (D);
 \draw (B) .. controls +( 0.5,0) and +(-0.5,0) .. (D);
 \draw (C) .. controls +( 0.5,0) and +(0,-0.5) .. (D);
\end{scope}
\begin{scope}[xshift= 4cm]
\coordinate (D) at (1.5, 2);
\coordinate (A) at (0.5, 2.5);
\coordinate (B) at (0.5, 1.5);
\coordinate (C) at (0.5, 0.5);
 \draw[draw= black, ] (2, 2.5) .. controls +(0, 0)  and +(0,0) .. (2, 0.5);
\draw[draw= white, line width=1mm] (D) .. controls +(2.5, 0) and +(0.5,0) .. (C);
\draw[draw= white, line width=1mm] (D) .. controls +(-0.5, -0.5) and +(0.5,0) .. (B);
\draw[draw= white, line width=1mm] (D) .. controls +(-0.5, +0.5) and +(0.5,0) .. (A);
\draw (D) .. controls +(2.5, 0) and +(0.5,0) .. (C);
\draw (D) .. controls +(-0.5, -0.5) and +(0.5,0) .. (B);
\draw (D) .. controls +(-0.5, +0.5) and +(0.5,0) .. (A);
\end{scope}
\begin{scope}[xshift=8cm]
\coordinate (D) at (1.5, 1.5);
\coordinate (A) at (0.5, 2.5);
\coordinate (B) at (0.5, 1.5);
\coordinate (C) at (0.5, 0.5);
 \draw[draw= black] (2, 2.5) .. controls +(0, 0)  and +(0,0) .. (2, 0.5);
 \draw[draw= white, line width=1mm] (A) .. controls +( 0.5,0) and +(0,0.5) .. (D);
 \draw[draw= white, line width=1mm] (B) .. controls +( 0.5,0) and +(-0.5,0) .. (D);
 \draw[draw= white, line width=1mm] (C) .. controls +( 0.5,0) and +(0,-0.5) .. (D);
 \draw (A) .. controls +( 0.5,0) and +(0,0.5) .. (D);
 \draw (B) .. controls +( 0.5,0) and +(-0.5,0) .. (D);
 \draw (C) .. controls +( 0.5,0) and +(0,-0.5) .. (D);
\end{scope}
\end{scope}
 \end{tikzpicture}\]
\[
\vspace{0.3cm}
\begin{tikzpicture}[scale=0.35] 
   \begin{scope}
\begin{scope}
\draw[draw=  gray, line width = 2mm] (0,0) --  (16,0); 
\draw [draw = gray, line width = 2mm] (16,3) -- (0,3); 
\draw[draw=  gray, very thick] (0,0) -- (0,3);
\draw[draw=  gray, very thick] (4,0) -- (4,3);
\draw[draw=  gray, very thick] (8,0) -- (8,3);
\draw[draw=  gray, very thick] (12,0) -- (12,3);
\draw[draw=  gray, very thick] (16,0) -- (16,3);
\draw[draw= white, dotted, line width =1.2mm] (0,0) -- (16,0);
\draw[draw= white, dotted, line width = 1.2mm] (0,3) -- (16,3);
\end{scope}
\begin{scope}
\coordinate (D1) at (3, 2);
\coordinate (A1) at (3.5, 2.5);
\coordinate (B1) at (0.5, 0.5);
\coordinate (C1) at (1, 0.5);
\coordinate (D2) at (3, 1);
\coordinate (A2) at (3.5, 0.5);
\coordinate (B2) at (0.5, 2.5);
\coordinate (C2) at (1, 2.5);
 \draw (A1) .. controls +( 0,0) and +(0.5,0.5) .. (D1);
 \draw (B1) .. controls +( 0,0) and +(-0.5,0) .. (D1);
 \draw (C1) .. controls +( 0,0) and +(0,-0.5) .. (D1);
 \draw[draw= white, line width=1mm] (A2) .. controls +( 0,0) and +(0.5, - 0.5) .. (D2);
 \draw[draw= white, line width=1mm] (B2) .. controls +( 0,0) and +(-0.5,0) .. (D2);
 \draw[draw= white, line width=1mm] (C2) .. controls +( 0,0) and +(0,0.5) .. (D2);
 \draw (A2) .. controls +( 0,0) and +(0.5, -0.5) .. (D2);
 \draw (B2) .. controls +( 0,0) and +(-0.5,0) .. (D2);
 \draw (C2) .. controls +( 0,0) and +(0,0.5) .. (D2);
\end{scope}
\begin{scope}[xshift= 4cm]
\coordinate (D1) at (2.15, 1.65);
\coordinate (A1) at (3.5, 2.5);
\coordinate (B1) at (0.5, 0.5);
\coordinate (C1) at (1, 0.5);
\coordinate (D2) at (3, 1);
\coordinate (A2) at (3.5, 0.5);
\coordinate (B2) at (0.5, 2.5);
\coordinate (C2) at (1, 2.5);
 \draw (A1) .. controls +( 0,0) and +(0.5,0.5) .. (D1);
 \draw (B1) .. controls +( 0,0) and +(-0.5,0) .. (D1);
 \draw (C1) .. controls +( 0,0) and +(0,-0.5) .. (D1);
 \draw[draw= white, line width=1mm] (A2) .. controls +( 0,0) and +(0.5, - 0.5) .. (D2);
 \draw[draw= white, line width=1mm] (B2) .. controls +( 0,0) and +(-0.5,0) .. (D2);
 \draw[draw= white, line width=1mm] (C2) .. controls +( 0,0) and +(0,0.5) .. (D2);
 \draw (A2) .. controls +( 0,0) and +(0.5, -0.5) .. (D2);
 \draw (B2) .. controls +( 0,0) and +(-0.5,0) .. (D2);
 \draw (C2) .. controls +( 0,0) and +(0,0.5) .. (D2);
\end{scope}
\begin{scope}[xshift=8cm]
\coordinate (D1) at (1.1, 1);
\coordinate (A1) at (3.5, 2.5);
\coordinate (B1) at (0.5, 0.5);
\coordinate (C1) at (1, 0.5);
\coordinate (D2) at (3, 1);
\coordinate (A2) at (3.5, 0.5);
\coordinate (B2) at (0.5, 2.5);
\coordinate (C2) at (1, 2.5);
 \draw (A1) .. controls +( 0,0) and +(0.5,0.5) .. (D1);
 \draw (B1) .. controls +( 0,0) and +(-0.5,0) .. (D1);
 \draw (C1) .. controls +( 0,0) and +(0,-0.5) .. (D1);
 \draw[draw= white, line width=1mm] (A2) .. controls +( 0,0) and +(0.5, - 0.5) .. (D2);
 \draw[draw= white, line width=1mm] (B2) .. controls +( 0,0) and +(-0.5,0) .. (D2);
 \draw[draw= white, line width=1mm] (C2) .. controls +( 0,0) and +(0,0.5) .. (D2);
 \draw (A2) .. controls +( 0,0) and +(0.5, -0.5) .. (D2);
 \draw (B2) .. controls +( 0,0) and +(-0.5,0) .. (D2);
 \draw (C2) .. controls +( 0,0) and +(0,0.5) .. (D2);
\end{scope}
\begin{scope}[xshift=12cm]
\coordinate (D1) at (1.1, 1);
\coordinate (A1) at (3.5, 2.5);
\coordinate (B1) at (0.5, 0.5);
\coordinate (C1) at (1, 0.5);
\coordinate (D2) at (1.1, 2);
\coordinate (A2) at (3.5, 0.5);
\coordinate (B2) at (0.5, 2.5);
\coordinate (C2) at (1, 2.5);
 \draw (A1) .. controls +( 0,0) and +(0.5,0.5) .. (D1);
 \draw (B1) .. controls +( 0,0) and +(-0.5,0) .. (D1);
 \draw (C1) .. controls +( 0,0) and +(0,-0.5) .. (D1);
 \draw[draw= white, line width=1mm] (A2) .. controls +( 0,0) and +(0.5, - 0.5) .. (D2);
 \draw[draw= white, line width=1mm] (B2) .. controls +( 0,0) and +(-0.5,0) .. (D2);
 \draw[draw= white, line width=1mm] (C2) .. controls +( 0,0) and +(0,0.5) .. (D2);
 \draw (A2) .. controls +( 0,0) and +(0.5, -0.5) .. (D2);
 \draw (B2) .. controls +( 0,0) and +(-0.5,0) .. (D2);
 \draw (C2) .. controls +( 0,0) and +(0,0.5) .. (D2);
\end{scope}
\end{scope}


\node at (18, 1.5) {$\leftrightsquigarrow$};
\begin{scope}[xshift= 20cm]

\begin{scope}
\draw[draw=  gray, line width = 2mm] (0,0) --  (16,0); 
\draw [draw = gray, line width = 2mm] (16,3) -- (0,3); 
\draw[draw=  gray, very thick] (0,0) -- (0,3);
\draw[draw=  gray, very thick] (4,0) -- (4,3);
\draw[draw=  gray, very thick] (8,0) -- (8,3);
\draw[draw=  gray, very thick] (12,0) -- (12,3);
\draw[draw=  gray, very thick] (16,0) -- (16,3);
\draw[draw= white, dotted, line width =1.2mm] (0,0) -- (16,0);
\draw[draw= white, dotted, line width = 1.2mm] (0,3) -- (16,3);
\end{scope}
\begin{scope}
\coordinate (D1) at (3, 2);
\coordinate (A1) at (3.5, 2.5);
\coordinate (B1) at (0.5, 0.5);
\coordinate (C1) at (1, 0.5);
\coordinate (D2) at (3, 1);
\coordinate (A2) at (3.5, 0.5);
\coordinate (B2) at (0.5, 2.5);
\coordinate (C2) at (1, 2.5);
 \draw (A1) .. controls +( 0,0) and +(0.5,0.5) .. (D1);
 \draw (B1) .. controls +( 0,0) and +(-0.5,0) .. (D1);
 \draw (C1) .. controls +( 0,0) and +(0,-0.5) .. (D1);
 \draw[draw= white, line width=1mm] (A2) .. controls +( 0,0) and +(0.5, - 0.5) .. (D2);
 \draw[draw= white, line width=1mm] (B2) .. controls +( 0,0) and +(-0.5,0) .. (D2);
 \draw[draw= white, line width=1mm] (C2) .. controls +( 0,0) and +(0,0.5) .. (D2);
 \draw (A2) .. controls +( 0,0) and +(0.5, -0.5) .. (D2);
 \draw (B2) .. controls +( 0,0) and +(-0.5,0) .. (D2);
 \draw (C2) .. controls +( 0,0) and +(0,0.5) .. (D2);
\end{scope}
\begin{scope}[xshift= 4cm]
\coordinate (D1) at (3, 2);
\coordinate (A1) at (3.5, 2.5);
\coordinate (B1) at (0.5, 0.5);
\coordinate (C1) at (1, 0.5);
\coordinate (D2) at (2.15, 1.35);
\coordinate (A2) at (3.5, 0.5);
\coordinate (B2) at (0.5, 2.5);
\coordinate (C2) at (1, 2.5);
 \draw (A1) .. controls +( 0,0) and +(0.5,0.5) .. (D1);
 \draw (B1) .. controls +( 0,0) and +(-0.5,0) .. (D1);
 \draw (C1) .. controls +( 0,0) and +(0,-0.5) .. (D1);
 \draw[draw= white, line width=1mm] (A2) .. controls +( 0,0) and +(0.5, - 0.5) .. (D2);
 \draw[draw= white, line width=1mm] (B2) .. controls +( 0,0) and +(-0.5,0) .. (D2);
 \draw[draw= white, line width=1mm] (C2) .. controls +( 0,0) and +(0,0.5) .. (D2);
 \draw (A2) .. controls +( 0,0) and +(0.5, -0.5) .. (D2);
 \draw (B2) .. controls +( 0,0) and +(-0.5,0) .. (D2);
 \draw (C2) .. controls +( 0,0) and +(0,0.5) .. (D2);
\end{scope}
\begin{scope}[xshift=8cm]
\coordinate (D1) at (3, 2);
\coordinate (A1) at (3.5, 2.5);
\coordinate (B1) at (0.5, 0.5);
\coordinate (C1) at (1, 0.5);
\coordinate (D2) at (1.1, 2);
\coordinate (A2) at (3.5, 0.5);
\coordinate (B2) at (0.5, 2.5);
\coordinate (C2) at (1, 2.5);
 \draw (A1) .. controls +( 0,0) and +(0.5,0.5) .. (D1);
 \draw (B1) .. controls +( 0,0) and +(-0.5,0) .. (D1);
 \draw (C1) .. controls +( 0,0) and +(0,-0.5) .. (D1);
 \draw[draw= white, line width=1mm] (A2) .. controls +( 0,0) and +(0.5, - 0.5) .. (D2);
 \draw[draw= white, line width=1mm] (B2) .. controls +( 0,0) and +(-0.5,0) .. (D2);
 \draw[draw= white, line width=1mm] (C2) .. controls +( 0,0) and +(0,0.5) .. (D2);
 \draw (A2) .. controls +( 0,0) and +(0.5, -0.5) .. (D2);
 \draw (B2) .. controls +( 0,0) and +(-0.5,0) .. (D2);
 \draw (C2) .. controls +( 0,0) and +(0,0.5) .. (D2);
\end{scope}
\begin{scope}[xshift=12cm]
\coordinate (D1) at (1.1, 1);
\coordinate (A1) at (3.5, 2.5);
\coordinate (B1) at (0.5, 0.5);
\coordinate (C1) at (1, 0.5);
\coordinate (D2) at (1.1, 2);
\coordinate (A2) at (3.5, 0.5);
\coordinate (B2) at (0.5, 2.5);
\coordinate (C2) at (1, 2.5);
 \draw (A1) .. controls +( 0,0) and +(0.5,0.5) .. (D1);
 \draw (B1) .. controls +( 0,0) and +(-0.5,0) .. (D1);
 \draw (C1) .. controls +( 0,0) and +(0,-0.5) .. (D1);
 \draw[draw= white, line width=1mm] (A2) .. controls +( 0,0) and +(0.5, - 0.5) .. (D2);
 \draw[draw= white, line width=1mm] (B2) .. controls +( 0,0) and +(-0.5,0) .. (D2);
 \draw[draw= white, line width=1mm] (C2) .. controls +( 0,0) and +(0,0.5) .. (D2);
 \draw (A2) .. controls +( 0,0) and +(0.5, -0.5) .. (D2);
 \draw (B2) .. controls +( 0,0) and +(-0.5,0) .. (D2);
 \draw (C2) .. controls +( 0,0) and +(0,0.5) .. (D2);
\end{scope}
\end{scope}
 \end{tikzpicture}
\]
\[
\vspace{0.3cm}
\begin{tikzpicture}[scale=0.4] 
   \begin{scope}
\begin{scope}
\draw[draw=  gray, line width = 2mm] (0,0) --  (16,0); 
\draw [draw = gray, line width = 2mm] (16,3) -- (0,3); 
\draw[draw=  gray, very thick] (0,0) -- (0,3);
\draw[draw=  gray, very thick] (4,0) -- (4,3);
\draw[draw=  gray, very thick] (8,0) -- (8,3);
\draw[draw=  gray, very thick] (12,0) -- (12,3);
\draw[draw=  gray, very thick] (16,0) -- (16,3);
\draw[draw= white, dotted, line width =1.2mm] (0,0) -- (16,0);
\draw[draw= white, dotted, line width = 1.2mm] (0,3) -- (16,3);
\end{scope}
\begin{scope}
\coordinate (A1) at (3.5, 2.5);
\coordinate (D1) at (3, 2);
\coordinate (B1) at (0.5, 0.5);
\coordinate (C1) at (1, 0.5);
\coordinate (A2) at (0.5, 2);
\coordinate (B2) at (3.5, 1.5);
\coordinate (A3) at (0.5, 2.5); 
\coordinate (B3) at (3.5, 0.5);
 \draw (A1) .. controls +( -0.2,-0.2) and +(0.2,0.2) .. (D1);
 \draw (B1) .. controls +( 0.5,0.3) and +(-0.4,-0.2) .. (D1);
 \draw (C1) .. controls +( 0.5,0.3) and +(-0.2,-0.4) .. (D1);
 \draw[draw= white, line width=1mm] (A2) -- (B2);
 \draw (A2) -- (B2);
 \draw[draw= white, line width=1mm] (A3) -- (B3);
 \draw (A3) -- (B3);
 \end{scope}
\begin{scope}[xshift=4cm]
\coordinate (A1) at (3.5, 2.5);
\coordinate (D1) at (2.5, 1.5);
\coordinate (B1) at (0.5, 0.5);
\coordinate (C1) at (1, 0.5);
\coordinate (A2) at (0.5, 2);
\coordinate (B2) at (3.5, 1.5);
\coordinate (A3) at (0.5, 2.5); 
\coordinate (B3) at (3.5, 0.5);
 \draw (A1) .. controls +( -0.2,-0.2) and +(0.2,0.2) .. (D1);
 \draw (B1) .. controls +( 0.5,0.3) and +(-0.4,-0.2) .. (D1);
 \draw (C1) .. controls +( 0.5,0.3) and +(-0.2,-0.4) .. (D1);
 \draw[draw= white, line width=1mm] (A2) -- (B2);
 \draw (A2) -- (B2);
 \draw[draw= white, line width=1mm] (A3) -- (B3);
 \draw (A3) -- (B3);
\end{scope}
\begin{scope}[xshift=8cm]
\coordinate (A1) at (3.5, 2.5);
\coordinate (D1) at (1.3, 1);
\coordinate (B1) at (0.5, 0.5);
\coordinate (C1) at (1, 0.5);
\coordinate (A2) at (0.5, 2);
\coordinate (B2) at (3.5, 1.5);
\coordinate (A3) at (0.5, 2.5); 
\coordinate (B3) at (3.5, 0.5);
 \draw (A1) .. controls +( -0.2,-0.2) and +(0.2,0.2) .. (D1);
 \draw (B1) .. controls +( 0,0) and +(-0.4,-0.2) .. (D1);
 \draw (C1) .. controls +( 0,0) and +(-0.2,-0.4) .. (D1);
 \draw[draw= white, line width=1mm] (A2) -- (B2);
 \draw (A2) -- (B2);
 \draw[draw= white, line width=1mm] (A3) -- (B3);
 \draw (A3) -- (B3);
\end{scope}
\begin{scope}[xshift=12cm]
\coordinate (A1) at (3.5, 2.5);
\coordinate (D1) at (1.3, 1);
\coordinate (B1) at (0.5, 0.5);
\coordinate (C1) at (1, 0.5);
\coordinate (A2) at (0.5, 2);
\coordinate (B2) at (3.5, 1.5);
\coordinate (A3) at (0.5, 2.5); 
\coordinate (B3) at (3.5, 0.5);
 \draw (A1) .. controls +( -0.2,-0.2) and +(0.2,0.2) .. (D1);
 \draw (B1) .. controls +( 0,0) and +(-0.4,-0.2) .. (D1);
 \draw (C1) .. controls +( 0,0) and +(-0.2,-0.4) .. (D1);
 \draw[draw= white, line width=1mm] (A2) .. controls +(0.5, -1) and +(-1,-0.5 ) .. (B2);
 \draw (A2) .. controls +(0.5, -1) and +(-1,-0.5 ) .. (B2);
 \draw[draw= white, line width=1mm] (A3) -- (B3);
 \draw (A3) -- (B3);
\end{scope}
\end{scope}
\node[rotate= 90] at (8, -1.5) {${\leftrightsquigarrow}$};

\begin{scope}[yshift =-6cm, xshift = -2cm]
\begin{scope}
\draw[draw=  gray, line width = 2mm] (0,0) --  (20,0); 
\draw [draw = gray, line width = 2mm] (20,3) -- (0,3); 
\draw[draw=  gray, very thick] (0,0) -- (0,3);
\draw[draw=  gray, very thick] (4,0) -- (4,3);
\draw[draw=  gray, very thick] (8,0) -- (8,3);
\draw[draw=  gray, very thick] (12,0) -- (12,3);
\draw[draw=  gray, very thick] (16,0) -- (16,3);
\draw[draw=  gray, very thick] (20,0) -- (20,3);
\draw[draw= white, dotted, line width =1.2mm] (0,0) -- (20,0);
\draw[draw= white, dotted, line width = 1.2mm] (0,3) -- (20,3);
\end{scope}
\begin{scope}
\coordinate (A1) at (3.5, 2.5);
\coordinate (D1) at (3, 2);
\coordinate (B1) at (0.5, 0.5);
\coordinate (C1) at (1, 0.5);
\coordinate (A2) at (0.5, 2);
\coordinate (B2) at (3.5, 1.5);
\coordinate (A3) at (0.5, 2.5); 
\coordinate (B3) at (3.5, 0.5);
 \draw (A1) .. controls +( -0.2,-0.2) and +(0.2,0.2) .. (D1);
 \draw (B1) .. controls +( 0.5,0.3) and +(-0.4,-0.2) .. (D1);
 \draw (C1) .. controls +( 0.5,0.3) and +(-0.2,-0.4) .. (D1);
 \draw[draw= white, line width=1mm] (A2) -- (B2);
 \draw (A2) -- (B2);
 \draw[draw= white, line width=1mm] (A3) -- (B3);
 \draw (A3) -- (B3);
\end{scope}
\begin{scope}[xshift=4cm]
\coordinate (A1) at (3.5, 2.5);
\coordinate (D1) at (3, 2);
\coordinate (B1) at (0.5, 0.5);
\coordinate (C1) at (1, 0.5);
\coordinate (A2) at (0.5, 2);
\coordinate (B2) at (3.5, 1.5);
\coordinate (A3) at (0.5, 2.5); 
\coordinate (B3) at (3.5, 0.5);
 \draw (A1) .. controls +( -0.2,-0.2) and +(0.2,0.2) .. (D1);
 \draw (B1) .. controls +(0.1,1.5) and +(-2.5,+0.3) .. (D1);
 \draw (C1) .. controls +( 0.5,0.3) and +(-0.2,-0.4) .. (D1);
 \draw[draw= white, line width=1mm] (A2) -- (B2);
 \draw (A2) -- (B2);
 \draw[draw= white, line width=1mm] (A3) -- (B3);
 \draw (A3) -- (B3);
 \end{scope}
\begin{scope}[xshift=8cm]
\coordinate (A1) at (3.5, 2.5);
\coordinate (D1) at (3, 2);
\coordinate (B1) at (0.5, 0.5);
\coordinate (C1) at (1, 0.5);
\coordinate (A2) at (0.5, 2);
\coordinate (B2) at (3.5, 1.5);
\coordinate (A3) at (0.5, 2.5); 
\coordinate (B3) at (3.5, 0.5);
 \draw (A1) .. controls +( -0.2,-0.2) and +(0.2,0.2) .. (D1);
 \draw (B1) .. controls +(0.1,1.5) and +(-2.5,+0.3) .. (D1);
 \draw (C1) .. controls +( 0,1.5) and +(-0.2,-0.4) .. (D1);
 \draw[draw= white, line width=1mm] (A2) ..controls +(0.5,-0.5) and +(-0.5,-0.5).. (B2);
 \draw (A2)  ..controls +(0.5,-0.5) and +(-0.5,-0.5).. (B2);
 \draw[draw= white, line width=1mm] (A3) ..controls +(1, 0) and  +(-0,1).. (B3);
 \draw (A3) ..controls +(1, 0) and  +(-0,1).. (B3);
\end{scope}
\begin{scope}[xshift=12cm]
\coordinate (A1) at (3.5, 2.5);
\coordinate (D1) at (1.8, 1.9);
\coordinate (B1) at (0.5, 0.5);
\coordinate (C1) at (1, 0.5);
\coordinate (A2) at (0.5, 2);
\coordinate (B2) at (3.5, 1.5);
\coordinate (A3) at (0.5, 2.5); 
\coordinate (B3) at (3.5, 0.5);
 \draw (A1) .. controls +( -0.2,-0.2) and +(0.2,0.2) .. (D1);
 \draw (B1) .. controls +(0.1,1.5) and +(-0.5,+0) .. (D1);
 \draw (C1) .. controls +( 0,1.5) and +(-0.2,-0.4) .. (D1);
 \draw[draw= white, line width=1mm] (A2) ..controls +(0.5,-0.5) and +(-0.5,-0.5).. (B2);
 \draw (A2)  ..controls +(0.5,-0.5) and +(-0.5,-0.5).. (B2);
 \draw[draw= white, line width=1mm] (A3) ..controls +(1, 0) and  +(-0,1).. (B3);
 \draw (A3) ..controls +(1, 0) and  +(-0,1).. (B3);
\end{scope}
\begin{scope}[xshift=16cm]
\coordinate (A1) at (3.5, 2.5);
\coordinate (D1) at (1.3, 1);
\coordinate (B1) at (0.5, 0.5);
\coordinate (C1) at (1, 0.5);
\coordinate (A2) at (0.5, 2);
\coordinate (B2) at (3.5, 1.5);
\coordinate (A3) at (0.5, 2.5); 
\coordinate (B3) at (3.5, 0.5);
 \draw (A1) .. controls +( -0.2,-0.2) and +(0.2,0.2) .. (D1);
 \draw (B1) .. controls +( 0,0) and +(-0.4,-0.2) .. (D1);
 \draw (C1) .. controls +( 0,0) and +(-0.2,-0.4) .. (D1);
 \draw[draw= white, line width=1mm] (A2) .. controls +(0.5, -1) and +(-1,-0.5 ) .. (B2);
 \draw (A2) .. controls +(0.5, -1) and +(-1,-0.5 ) .. (B2);
 \draw[draw= white, line width=1mm] (A3) -- (B3);
 \draw (A3) -- (B3);
\end{scope}
\end{scope}

 \end{tikzpicture}
\]
\[
\vspace{0.3cm}
\begin{tikzpicture}[scale=0.4] 
   \begin{scope}
\begin{scope}
\draw[draw=  gray, line width = 2mm] (0,0) --  (16,0); 
\draw [draw = gray, line width = 2mm] (16,3) -- (0,3); 
\draw[draw=  gray, very thick] (0,0) -- (0,3);
\draw[draw=  gray, very thick] (4,0) -- (4,3);
\draw[draw=  gray, very thick] (8,0) -- (8,3);
\draw[draw=  gray, very thick] (12,0) -- (12,3);
\draw[draw=  gray, very thick] (16,0) -- (16,3);
\draw[draw= white, dotted, line width =1.2mm] (0,0) -- (16,0);
\draw[draw= white, dotted, line width = 1.2mm] (0,3) -- (16,3);
\end{scope}
\begin{scope}
\coordinate (A1) at (3.5, 2.5);
\coordinate (D1) at (3, 2);
\coordinate (B1) at (0.5, 0.5);
\coordinate (C1) at (1, 0.5);
\coordinate (A2) at (0.5, 2);
\coordinate (B2) at (3.5, 1.5);
\coordinate (A3) at (0.5, 2.5); 
\coordinate (B3) at (3.5, 0.5);
 \draw (A1) .. controls +( -0.2,-0.2) and +(0.2,0.2) .. (D1);
 \draw (B1) .. controls +( 0.5,0.3) and +(-0.4,-0.2) .. (D1);
 \draw (C1) .. controls +( 0.5,0.3) and +(-0.2,-0.4) .. (D1);
 \draw[draw= white, line width=1mm] (A3) -- (B3);
 \draw (A3) -- (B3);
 \draw[draw= white, line width=1mm] (A2) -- (B2);
 \draw (A2) -- (B2);
 \end{scope}
\begin{scope}[xshift=4cm]
\coordinate (A1) at (3.5, 2.5);
\coordinate (D1) at (2.5, 1.5);
\coordinate (B1) at (0.5, 0.5);
\coordinate (C1) at (1, 0.5);
\coordinate (A2) at (0.5, 2);
\coordinate (B2) at (3.5, 1.5);
\coordinate (A3) at (0.5, 2.5); 
\coordinate (B3) at (3.5, 0.5);
 \draw (A1) .. controls +( -0.2,-0.2) and +(0.2,0.2) .. (D1);
 \draw (B1) .. controls +( 0.5,0.3) and +(-0.4,-0.2) .. (D1);
 \draw (C1) .. controls +( 0.5,0.3) and +(-0.2,-0.4) .. (D1);
 \draw[draw= white, line width=1mm] (A3) -- (B3);
 \draw (A3) -- (B3);
 \draw[draw= white, line width=1mm] (A2) -- (B2);
 \draw (A2) -- (B2);
\end{scope}
\begin{scope}[xshift=8cm]
\coordinate (A1) at (3.5, 2.5);
\coordinate (D1) at (1.3, 1);
\coordinate (B1) at (0.5, 0.5);
\coordinate (C1) at (1, 0.5);
\coordinate (A2) at (0.5, 2);
\coordinate (B2) at (3.5, 1.5);
\coordinate (A3) at (0.5, 2.5); 
\coordinate (B3) at (3.5, 0.5);
 \draw (A1) .. controls +( -0.2,-0.2) and +(0.2,0.2) .. (D1);
 \draw (B1) .. controls +( 0,0) and +(-0.4,-0.2) .. (D1);
 \draw (C1) .. controls +( 0,0) and +(-0.2,-0.4) .. (D1);
 \draw[draw= white, line width=1mm] (A3) -- (B3);
 \draw (A3) -- (B3);
 \draw[draw= white, line width=1mm] (A2) -- (B2);
 \draw (A2) -- (B2);
\end{scope}
\begin{scope}[xshift=12cm]
\coordinate (A1) at (3.5, 2.5);
\coordinate (D1) at (1.3, 1);
\coordinate (B1) at (0.5, 0.5);
\coordinate (C1) at (1, 0.5);
\coordinate (A2) at (0.5, 2);
\coordinate (B2) at (3.5, 1.5);
\coordinate (A3) at (0.5, 2.5); 
\coordinate (B3) at (3.5, 0.5);
 \draw (A1) .. controls +( -0.2,-0.2) and +(0.2,0.2) .. (D1);
 \draw (B1) .. controls +( 0,0) and +(-0.4,-0.2) .. (D1);
 \draw (C1) .. controls +( 0,0) and +(-0.2,-0.4) .. (D1);
 \draw[draw= white, line width=1mm] (A3) -- (B3);
 \draw (A3) -- (B3);
 \draw[draw= white, line width=1mm] (A2) .. controls +(0.5, -1) and +(-1,-0.5 ) .. (B2);
 \draw (A2) .. controls +(0.5, -1) and +(-1,-0.5 ) .. (B2);
\end{scope}
\end{scope}
\node[rotate= 90] at (8, -1.5) {${\leftrightsquigarrow}$};
\begin{scope}[yshift =-6cm, xshift = -2cm]
\begin{scope}
\draw[draw=  gray, line width = 2mm] (0,0) --  (20,0); 
\draw [draw = gray, line width = 2mm] (20,3) -- (0,3); 
\draw[draw=  gray, very thick] (0,0) -- (0,3);
\draw[draw=  gray, very thick] (4,0) -- (4,3);
\draw[draw=  gray, very thick] (8,0) -- (8,3);
\draw[draw=  gray, very thick] (12,0) -- (12,3);
\draw[draw=  gray, very thick] (16,0) -- (16,3);
\draw[draw=  gray, very thick] (20,0) -- (20,3);
\draw[draw= white, dotted, line width =1.2mm] (0,0) -- (20,0);
\draw[draw= white, dotted, line width = 1.2mm] (0,3) -- (20,3);
\end{scope}
\begin{scope}
\coordinate (A1) at (3.5, 2.5);
\coordinate (D1) at (3, 2);
\coordinate (B1) at (0.5, 0.5);
\coordinate (C1) at (1, 0.5);
\coordinate (A2) at (0.5, 2);
\coordinate (B2) at (3.5, 1.5);
\coordinate (A3) at (0.5, 2.5); 
\coordinate (B3) at (3.5, 0.5);
 \draw (A1) .. controls +( -0.2,-0.2) and +(0.2,0.2) .. (D1);
 \draw (B1) .. controls +( 0.5,0.3) and +(-0.4,-0.2) .. (D1);
 \draw (C1) .. controls +( 0.5,0.3) and +(-0.2,-0.4) .. (D1);
 \draw[draw= white, line width=1mm] (A3) -- (B3);
 \draw (A3) -- (B3);
 \draw[draw= white, line width=1mm] (A2) -- (B2);
 \draw (A2) -- (B2);
\end{scope}
\begin{scope}[xshift=4cm]
\coordinate (A1) at (3.5, 2.5);
\coordinate (D1) at (3, 2);
\coordinate (B1) at (0.5, 0.5);
\coordinate (C1) at (1, 0.5);
\coordinate (A2) at (0.5, 2);
\coordinate (B2) at (3.5, 1.5);
\coordinate (A3) at (0.5, 2.5); 
\coordinate (B3) at (3.5, 0.5);
 \draw (A1) .. controls +( -0.2,-0.2) and +(0.2,0.2) .. (D1);
 \draw (B1) .. controls +(0.1,1.5) and +(-2.5,+0.3) .. (D1);
 \draw (C1) .. controls +( 0.5,0.3) and +(-0.2,-0.4) .. (D1);
 \draw[draw= white, line width=1mm] (A3) -- (B3);
 \draw (A3) -- (B3);
 \draw[draw= white, line width=1mm] (A2) -- (B2);
 \draw (A2) -- (B2);
 \end{scope}
\begin{scope}[xshift=8cm]
\coordinate (A1) at (3.5, 2.5);
\coordinate (D1) at (3, 2);
\coordinate (B1) at (0.5, 0.5);
\coordinate (C1) at (1, 0.5);
\coordinate (A2) at (0.5, 2);
\coordinate (B2) at (3.5, 1.5);
\coordinate (A3) at (0.5, 2.5); 
\coordinate (B3) at (3.5, 0.5);
 \draw (A1) .. controls +( -0.2,-0.2) and +(0.2,0.2) .. (D1);
 \draw (B1) .. controls +(0.1,1.5) and +(-2.5,+0.3) .. (D1);
 \draw (C1) .. controls +( 0,1.5) and +(-0.2,-0.4) .. (D1);
 \draw[draw= white, line width=1mm] (A3) ..controls +(1, 0) and  +(-0,1).. (B3);
 \draw (A3) ..controls +(1, 0) and  +(-0,1).. (B3);
 \draw[draw= white, line width=1mm] (A2) ..controls +(0.5,-0.5) and +(-0.5,-0.5).. (B2);
 \draw (A2)  ..controls +(0.5,-0.5) and +(-0.5,-0.5).. (B2);
\end{scope}
\begin{scope}[xshift=12cm]
\coordinate (A1) at (3.5, 2.5);
\coordinate (D1) at (1.8, 1.9);
\coordinate (B1) at (0.5, 0.5);
\coordinate (C1) at (1, 0.5);
\coordinate (A2) at (0.5, 2);
\coordinate (B2) at (3.5, 1.5);
\coordinate (A3) at (0.5, 2.5); 
\coordinate (B3) at (3.5, 0.5);
 \draw (A1) .. controls +( -0.2,-0.2) and +(0.2,0.2) .. (D1);
 \draw (B1) .. controls +(0.1,1.5) and +(-0.5,+0) .. (D1);
 \draw (C1) .. controls +( 0,1.5) and +(-0.2,-0.4) .. (D1);
 \draw[draw= white, line width=1mm] (A3) ..controls +(1, 0) and  +(-0,1).. (B3);
 \draw (A3) ..controls +(1, 0) and  +(-0,1).. (B3);
 \draw[draw= white, line width=1mm] (A2) ..controls +(0.5,-0.5) and +(-0.5,-0.5).. (B2);
 \draw (A2)  ..controls +(0.5,-0.5) and +(-0.5,-0.5).. (B2);
\end{scope}
\begin{scope}[xshift=16cm]
\coordinate (A1) at (3.5, 2.5);
\coordinate (D1) at (1.3, 1);
\coordinate (B1) at (0.5, 0.5);
\coordinate (C1) at (1, 0.5);
\coordinate (A2) at (0.5, 2);
\coordinate (B2) at (3.5, 1.5);
\coordinate (A3) at (0.5, 2.5); 
\coordinate (B3) at (3.5, 0.5);
 \draw (A1) .. controls +( -0.2,-0.2) and +(0.2,0.2) .. (D1);
 \draw (B1) .. controls +( 0,0) and +(-0.4,-0.2) .. (D1);
 \draw (C1) .. controls +( 0,0) and +(-0.2,-0.4) .. (D1);
 \draw[draw= white, line width=1mm] (A3) -- (B3);
 \draw (A3) -- (B3);
 \draw[draw= white, line width=1mm] (A2) .. controls +(0.5, -1) and +(-1,-0.5 ) .. (B2);
 \draw (A2) .. controls +(0.5, -1) and +(-1,-0.5 ) .. (B2);
\end{scope}
\end{scope}

 \end{tikzpicture}
\]
\[
\vspace{0.3cm}
\begin{tikzpicture}[scale=0.4] 
   \begin{scope}
\begin{scope}
\draw[draw=  gray, line width = 2mm] (0,0) --  (16,0); 
\draw [draw = gray, line width = 2mm] (16,3) -- (0,3); 
\draw[draw=  gray, very thick] (0,0) -- (0,3);
\draw[draw=  gray, very thick] (4,0) -- (4,3);
\draw[draw=  gray, very thick] (8,0) -- (8,3);
\draw[draw=  gray, very thick] (12,0) -- (12,3);
\draw[draw=  gray, very thick] (16,0) -- (16,3);
\draw[draw= white, dotted, line width =1.2mm] (0,0) -- (16,0);
\draw[draw= white, dotted, line width = 1.2mm] (0,3) -- (16,3);
\end{scope}
\begin{scope}
\coordinate (A1) at (3.5, 2.5);
\coordinate (D1) at (3, 2);
\coordinate (B1) at (0.5, 0.5);
\coordinate (C1) at (1, 0.5);
\coordinate (A2) at (0.5, 2);
\coordinate (B2) at (3.5, 1.5);
\coordinate (A3) at (0.5, 2.5); 
\coordinate (B3) at (3.5, 0.5);
 \draw[draw=white, line width=1mm] (A1) .. controls +( -0.2,-0.2) and +(0.2,0.2) .. (D1);
 \draw[draw=white, line width=1mm] (B1) .. controls +( 0.5,0.3) and +(-0.4,-0.2) .. (D1);
 \draw[draw=white, line width=1mm] (C1) .. controls +( 0.5,0.3) and +(-0.2,-0.4) .. (D1);
 \draw (A1) .. controls +( -0.2,-0.2) and +(0.2,0.2) .. (D1);
 \draw (B1) .. controls +( 0.5,0.3) and +(-0.4,-0.2) .. (D1);
 \draw (C1) .. controls +( 0.5,0.3) and +(-0.2,-0.4) .. (D1);
 \draw[draw= white, line width=1mm] (A2) -- (B2);
 \draw (A2) -- (B2);
 \draw[draw= white, line width=1mm] (A3) -- (B3);
 \draw (A3) -- (B3);
 \end{scope}
\begin{scope}[xshift=4cm]
\coordinate (A1) at (3.5, 2.5);
\coordinate (D1) at (2.5, 1.5);
\coordinate (B1) at (0.5, 0.5);
\coordinate (C1) at (1, 0.5);
\coordinate (A2) at (0.5, 2);
\coordinate (B2) at (3.5, 1.5);
\coordinate (A3) at (0.5, 2.5); 
\coordinate (B3) at (3.5, 0.5);
 \draw[draw=white, line width=1mm] (A1) .. controls +( -0.2,-0.2) and +(0.2,0.2) .. (D1);
 \draw[draw=white, line width=1mm] (B1) .. controls +( 0.5,0.3) and +(-0.4,-0.2) .. (D1);
 \draw[draw=white, line width=1mm] (C1) .. controls +( 0.5,0.3) and +(-0.2,-0.4) .. (D1);
 \draw (A1) .. controls +( -0.2,-0.2) and +(0.2,0.2) .. (D1);
 \draw (B1) .. controls +( 0.5,0.3) and +(-0.4,-0.2) .. (D1);
 \draw (C1) .. controls +( 0.5,0.3) and +(-0.2,-0.4) .. (D1);
 \draw[draw= white, line width=1mm] (A2) -- (B2);
 \draw (A2) -- (B2);
 \draw[draw= white, line width=1mm] (A3) -- (B3);
 \draw (A3) -- (B3);
\end{scope}
\begin{scope}[xshift=8cm]
\coordinate (A1) at (3.5, 2.5);
\coordinate (D1) at (1.3, 1);
\coordinate (B1) at (0.5, 0.5);
\coordinate (C1) at (1, 0.5);
\coordinate (A2) at (0.5, 2);
\coordinate (B2) at (3.5, 1.5);
\coordinate (A3) at (0.5, 2.5); 
\coordinate (B3) at (3.5, 0.5);
 \draw[draw=white, line width=1mm] (A1) .. controls +( -0.2,-0.2) and +(0.2,0.2) .. (D1);
 \draw[draw=white, line width=1mm] (B1) .. controls +( 0,0) and +(-0.4,-0.2) .. (D1);
 \draw[draw=white, line width=1mm] (C1) .. controls +( 0,0) and +(-0.2,-0.4) .. (D1);
 \draw (A1) .. controls +( -0.2,-0.2) and +(0.2,0.2) .. (D1);
 \draw (B1) .. controls +( 0,0) and +(-0.4,-0.2) .. (D1);
 \draw (C1) .. controls +( 0,0) and +(-0.2,-0.4) .. (D1);
 \draw[draw= white, line width=1mm] (A2) -- (B2);
 \draw (A2) -- (B2);
 \draw[draw= white, line width=1mm] (A3) -- (B3);
 \draw (A3) -- (B3);
\end{scope}
\begin{scope}[xshift=12cm]
\coordinate (A1) at (3.5, 2.5);
\coordinate (D1) at (1.3, 1);
\coordinate (B1) at (0.5, 0.5);
\coordinate (C1) at (1, 0.5);
\coordinate (A2) at (0.5, 2);
\coordinate (B2) at (3.5, 1.5);
\coordinate (A3) at (0.5, 2.5); 
\coordinate (B3) at (3.5, 0.5);
 \draw[draw=white, line width=1mm] (A1) .. controls +( -0.2,-0.2) and +(0.2,0.2) .. (D1);
 \draw[draw=white, line width=1mm] (B1) .. controls +( 0,0) and +(-0.4,-0.2) .. (D1);
 \draw[draw=white, line width=1mm] (C1) .. controls +( 0,0) and +(-0.2,-0.4) .. (D1); 
 \draw (A1) .. controls +( -0.2,-0.2) and +(0.2,0.2) .. (D1);
 \draw (B1) .. controls +( 0,0) and +(-0.4,-0.2) .. (D1);
 \draw (C1) .. controls +( 0,0) and +(-0.2,-0.4) .. (D1);
 \draw[draw= white, line width=1mm] (A2) .. controls +(0.5, -1) and +(-1,-0.5 ) .. (B2);
 \draw (A2) .. controls +(0.5, -1) and +(-1,-0.5 ) .. (B2);
 \draw[draw= white, line width=1mm] (A3) -- (B3);
 \draw (A3) -- (B3);
\end{scope}
\end{scope}
\node[rotate= 90] at (8, -1.5) {${\leftrightsquigarrow}$};
\begin{scope}[yshift =-6cm, xshift = -2cm]
\begin{scope}
\draw[draw=  gray, line width = 2mm] (0,0) --  (20,0); 
\draw [draw = gray, line width = 2mm] (20,3) -- (0,3); 
\draw[draw=  gray, very thick] (0,0) -- (0,3);
\draw[draw=  gray, very thick] (4,0) -- (4,3);
\draw[draw=  gray, very thick] (8,0) -- (8,3);
\draw[draw=  gray, very thick] (12,0) -- (12,3);
\draw[draw=  gray, very thick] (16,0) -- (16,3);
\draw[draw=  gray, very thick] (20,0) -- (20,3);
\draw[draw= white, dotted, line width =1.2mm] (0,0) -- (20,0);
\draw[draw= white, dotted, line width = 1.2mm] (0,3) -- (20,3);
\end{scope}
\begin{scope}
\coordinate (A1) at (3.5, 2.5);
\coordinate (D1) at (3, 2);
\coordinate (B1) at (0.5, 0.5);
\coordinate (C1) at (1, 0.5);
\coordinate (A2) at (0.5, 2);
\coordinate (B2) at (3.5, 1.5);
\coordinate (A3) at (0.5, 2.5); 
\coordinate (B3) at (3.5, 0.5);
 \draw[draw=white, line width=1mm] (A1) .. controls +( -0.2,-0.2) and +(0.2,0.2) .. (D1);
 \draw[draw=white, line width=1mm] (B1) .. controls +( 0.5,0.3) and +(-0.4,-0.2) .. (D1);
 \draw[draw=white, line width=1mm] (C1) .. controls +( 0.5,0.3) and +(-0.2,-0.4) .. (D1);
 \draw (A1) .. controls +( -0.2,-0.2) and +(0.2,0.2) .. (D1);
 \draw (B1) .. controls +( 0.5,0.3) and +(-0.4,-0.2) .. (D1);
 \draw (C1) .. controls +( 0.5,0.3) and +(-0.2,-0.4) .. (D1);
 \draw[draw= white, line width=1mm] (A2) -- (B2);
 \draw (A2) -- (B2);
 \draw[draw= white, line width=1mm] (A3) -- (B3);
 \draw (A3) -- (B3);
\end{scope}
\begin{scope}[xshift=4cm]
\coordinate (A1) at (3.5, 2.5);
\coordinate (D1) at (3, 2);
\coordinate (B1) at (0.5, 0.5);
\coordinate (C1) at (1, 0.5);
\coordinate (A2) at (0.5, 2);
\coordinate (B2) at (3.5, 1.5);
\coordinate (A3) at (0.5, 2.5); 
\coordinate (B3) at (3.5, 0.5);
 \draw[draw=white, line width=1mm] (A1) .. controls +( -0.2,-0.2) and +(0.2,0.2) .. (D1);
 \draw[draw=white, line width=1mm] (B1) .. controls +(0.1,1.5) and +(-2.5,+0.3) .. (D1);
 \draw[draw=white, line width=1mm] (C1) .. controls +( 0.5,0.3) and +(-0.2,-0.4) .. (D1);
 \draw (A1) .. controls +( -0.2,-0.2) and +(0.2,0.2) .. (D1);
 \draw (B1) .. controls +(0.1,1.5) and +(-2.5,+0.3) .. (D1);
 \draw (C1) .. controls +( 0.5,0.3) and +(-0.2,-0.4) .. (D1);
 \draw[draw= white, line width=1mm] (A2) -- (B2);
 \draw (A2) -- (B2);
 \draw[draw= white, line width=1mm] (A3) -- (B3);
 \draw (A3) -- (B3);
 \end{scope}
\begin{scope}[xshift=8cm]
\coordinate (A1) at (3.5, 2.5);
\coordinate (D1) at (3, 2);
\coordinate (B1) at (0.5, 0.5);
\coordinate (C1) at (1, 0.5);
\coordinate (A2) at (0.5, 2);
\coordinate (B2) at (3.5, 1.5);
\coordinate (A3) at (0.5, 2.5); 
\coordinate (B3) at (3.5, 0.5);
 \draw[draw=white, line width=1mm] (A1) .. controls +( -0.2,-0.2) and +(0.2,0.2) .. (D1);
 \draw[draw=white, line width=1mm] (B1) .. controls +(0.1,1.5) and +(-2.5,+0.3) .. (D1);
 \draw[draw=white, line width=1mm] (C1) .. controls +( 0,1.5) and +(-0.2,-0.4) .. (D1);
 \draw (A1) .. controls +( -0.2,-0.2) and +(0.2,0.2) .. (D1);
 \draw (B1) .. controls +(0.1,1.5) and +(-2.5,+0.3) .. (D1);
 \draw (C1) .. controls +( 0,1.5) and +(-0.2,-0.4) .. (D1);
 \draw[draw= white, line width=1mm] (A2) ..controls +(0.5,-0.5) and +(-0.5,-0.5).. (B2);
 \draw (A2)  ..controls +(0.5,-0.5) and +(-0.5,-0.5).. (B2);
 \draw[draw= white, line width=1mm] (A3) ..controls +(1, 0) and  +(-0,1).. (B3);
 \draw (A3) ..controls +(1, 0) and  +(-0,1).. (B3);
\end{scope}
\begin{scope}[xshift=12cm]
\coordinate (A1) at (3.5, 2.5);
\coordinate (D1) at (1.8, 1.9);
\coordinate (B1) at (0.5, 0.5);
\coordinate (C1) at (1, 0.5);
\coordinate (A2) at (0.5, 2);
\coordinate (B2) at (3.5, 1.5);
\coordinate (A3) at (0.5, 2.5); 
\coordinate (B3) at (3.5, 0.5);
 \draw[draw=white, line width=1mm] (A1) .. controls +( -0.2,-0.2) and +(0.2,0.2) .. (D1);
 \draw[draw=white, line width=1mm] (B1) .. controls +(0.1,1.5) and +(-0.5,+0) .. (D1);
 \draw[draw=white, line width=1mm] (C1) .. controls +( 0,1.5) and +(-0.2,-0.4) .. (D1);
 \draw (A1) .. controls +( -0.2,-0.2) and +(0.2,0.2) .. (D1);
 \draw (B1) .. controls +(0.1,1.5) and +(-0.5,+0) .. (D1);
 \draw (C1) .. controls +( 0,1.5) and +(-0.2,-0.4) .. (D1);
 \draw[draw= white, line width=1mm] (A2) ..controls +(0.5,-0.5) and +(-0.5,-0.5).. (B2);
 \draw (A2)  ..controls +(0.5,-0.5) and +(-0.5,-0.5).. (B2);
 \draw[draw= white, line width=1mm] (A3) ..controls +(1, 0) and  +(-0,1).. (B3);
 \draw (A3) ..controls +(1, 0) and  +(-0,1).. (B3);
\end{scope}
\begin{scope}[xshift=16cm]
\coordinate (A1) at (3.5, 2.5);
\coordinate (D1) at (1.3, 1);
\coordinate (B1) at (0.5, 0.5);
\coordinate (C1) at (1, 0.5);
\coordinate (A2) at (0.5, 2);
\coordinate (B2) at (3.5, 1.5);
\coordinate (A3) at (0.5, 2.5); 
\coordinate (B3) at (3.5, 0.5);
 \draw[draw=white, line width=1mm] (A1) .. controls +( -0.2,-0.2) and +(0.2,0.2) .. (D1);
 \draw[draw=white, line width=1mm] (B1) .. controls +( 0,0) and +(-0.4,-0.2) .. (D1);
 \draw[draw=white, line width=1mm] (C1) .. controls +( 0,0) and +(-0.2,-0.4) .. (D1);
 \draw (A1) .. controls +( -0.2,-0.2) and +(0.2,0.2) .. (D1);
 \draw (B1) .. controls +( 0,0) and +(-0.4,-0.2) .. (D1);
 \draw (C1) .. controls +( 0,0) and +(-0.2,-0.4) .. (D1);
 \draw[draw= white, line width=1mm] (A2) .. controls +(0.5, -1) and +(-1,-0.5 ) .. (B2);
 \draw (A2) .. controls +(0.5, -1) and +(-1,-0.5 ) .. (B2);
 \draw[draw= white, line width=1mm] (A3) -- (B3);
 \draw (A3) -- (B3);
\end{scope}
\end{scope}
 \end{tikzpicture}
\]
\[
\vspace{0.3cm}
\begin{tikzpicture}[scale=0.4] 
   \input{\imagesfolder/c3_movieWR5}
 \end{tikzpicture}
\]
\end{prop}


\subsection{Two movies}
\label{sec:two-foamy-cobordisms}

\subsubsection{The movie $\lambda$ }
\label{sec:cobordism-y}
We consider an oriented link diagram $L$ in the standard disk $B_0$ and $L_{//}$ a cabling of $L$ along one component $l$ with two parallel copies $l_1$ and $l_2$  both oriented conversely to the orientation of $l$. The links $L$ and $L_{//}$ are considered as object of $\mathcal{TD}_{\emptyset}$. Let $x$ be a regular (not a crossing) point of $l$.

\begin{dfn}
  \label{dfn:Ymorphism}
  The movie $\lambda(L,x, l_1, l_2)$ in $\mathrm{HOM}_{\mathcal{TD_{\emptyset}}}(L_{//}, L)$ is defined by the following sequence:
\begin{itemize}
\item A zip (B2) at $x$. We obtain a knotted web diagram $w$,
\item A sequence of moves of type B8 and B9 obtained by pushing one of the two singularities of $w$ along the strands following the orientation of $l$.
\item A digon cap (A4) in a neighborhood of $x$.
\end{itemize}
An example is given in figure~\ref{fig:examp7trefoil}
\begin{figure}[ht]
  \centering
\includegraphics{\imagesfolder/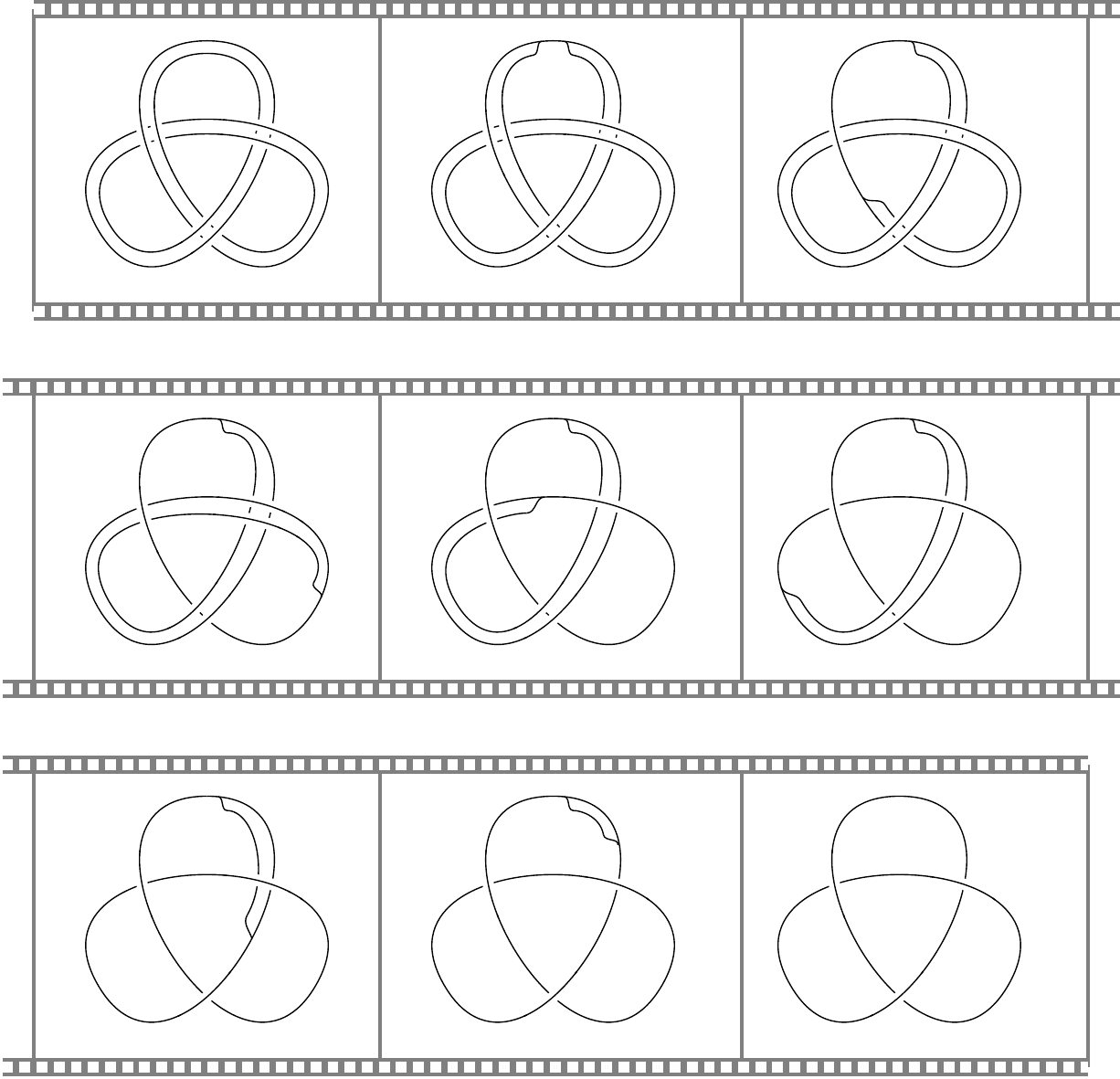}
  \caption{The morphism $\lambda(L,x)$ performed on the cabled trefoil. The picture is to be read from left to right and from top to bottom. Note that for simplicity, between the second and the fifth frames, there are two elementary movies instead of one between two frames. }
  \label{fig:examp7trefoil}
\end{figure}
\end{dfn}

\begin{lem}\label{lem:startpointinvariance}
  With the same notations and with $y$ another regular point of $l$, we have:
  \[\G(\lambda(L,x, l_1, l_2)) = \pm \G(\lambda(L,y, l_1, l_2)).\]
\end{lem}

\begin{proof}
  It is enough to consider the case where $x$ and $y$ are separated by only one crossing. In this case,
  this is an easy consequence of lemmas \ref{lem:zipcrossing}, \ref{lem:zipdigon} and proposition \ref{prop:importantmoves}.
\end{proof}

\begin{lem}\label{lem:ReidemeisterYinvariance}
With the same notations, if $L$ and $L'$ are two oriented link diagrams which are framed equivalent via a movie $m$, then the following diagram homotopy-commutes up to a sign:
\[
\begin{tikzpicture}[scale=2]
  \node (A) at (0,0) {$\G(L)$};
  \node (B) at (2,0) {$\G(L')$};
  \node (C) at (0,1) {$\G(L_{//})$};
  \node (D) at (2,1) {$\G(L'_{//})$};
\draw [->] (A) -- (B) node [midway, below] {$\G(m)$};
\draw [->] (C) -- (D) node [midway, above] {$\G(m_{//})$};
\draw [->] (C) -- (A) node [midway, left] {$\G(\lambda(L,x, l_1, l_2))$};
\draw [->] (D) -- (B) node [midway, right] {$\G(\lambda(L', x', l_1, l_2))$};
\end{tikzpicture}
\]
Where $m_{//}$ is the movie $m$ adapted to the cabling. 
\end{lem}

\begin{proof}
  This is enough to consider the case where $L$ and $L'$ are related by a Reidemeister move (A5-A7 and B4-B7). Thanks to lemma~\ref{lem:startpointinvariance}, we can suppose that the points $x$ and $x'$ are disjoint from the ball where the Reidemeister move takes place. With this setting, the result is an easy consequence of lemmas \ref{lem:zipcrossing}, \ref{lem:zipdigon} and proposition \ref{prop:importantmoves}.
\end{proof}
 
Consider $\lambda(L,x, l_1, l_2)$ in the other direction: it becomes a movie from $L$ to $L_{//}$ and let us denote it by $\mu(L, x, l_1, l_2)$. Composing these two movies, we obtain a movie $\alpha(L,x, l_1, l_2)$ from $L$ to $L$, 

\begin{lem}
  \label{lem:lambda2timesid}
   The $\G(\alpha(L,x, l_1, l_2))$ is homotopic to $2\cdot\id_{\G(L)}$.
\end{lem}
\begin{proof}[Sketch of the proof]
In a neighborhood of $x$, the movie $\alpha(L,x, l_1, l_2)$  reads as follows:
\[
\begin{tikzpicture}[scale=0.4]
  \begin{scope}[scale=1]
\begin{scope}
\draw[draw=  gray, line width = 2mm] (-0.50,0) --  (12.5,0); 
\draw [draw = gray, line width = 2mm] (-0.50,4) -- (12.5,4); 
\draw[draw=  gray, very thick] (0,0) -- (0,4);
\draw[draw=  gray, very thick] (4,0) -- (4,4);
\draw[draw=  gray, very thick] (8,0) -- (8,4);
\draw[draw=  gray, very thick] (12,0) -- (12,4);
\draw[draw= white, dotted, line width =1.2mm] (-0.50,0) -- (12.5,0);
\draw[draw= white, dotted, line width = 1.2mm] (-0.50,4) -- (12.5,4);
\draw[draw=  gray, line width = 2mm] (-10,0) --  (-1.5,0); 
\draw [draw = gray, line width = 2mm] (-10,4) -- (-1.5,4); 
\draw[draw=  gray, very thick] (-10,0) -- (-10,4);
\draw[draw=  gray, very thick] (-6,0) -- (-6,4);
\draw[draw=  gray, very thick] (-2,0) -- (-2,4);
\draw[draw= white, dotted, line width =1.2mm] (-10,0) -- (-1.5,0);
\draw[draw= white, dotted, line width = 1.2mm] (-10,4) -- (-1.5,4);
\draw[draw=  gray, line width = 2mm] (13.5,0) --  (22,0); 
\draw [draw = gray, line width = 2mm] (13.5,4) -- (22,4); 
\draw[draw=  gray, very thick] (22,0) -- (22,4);
\draw[draw=  gray, very thick] (18,0) -- (18,4);
\draw[draw=  gray, very thick] (14,0) -- (14,4);
\draw[draw= white, dotted, line width =1.2mm] (13.5,0) -- (22,0);
\draw[draw= white, dotted, line width = 1.2mm] (13.5,4) -- (22,4);

\node at ( -1, 2) {$\dots$};
\node at ( 13,2) {$\dots$};

\end{scope}
\begin{scope}[xshift= -8cm, yshift=2cm]
  \filldraw[dotted, fill=gray!50!white] (0,0) circle (1.5cm);
  \filldraw[dotted, fill = white] (0, 0) circle (1cm);
\draw (90:-1) -- (90:1);
\end{scope}
\begin{scope}[xshift= -4cm, yshift=2cm]
  \filldraw[dotted, fill=gray!50!white] (0,0) circle (1.5cm);
  \filldraw[dotted, fill = white] (0, 0) circle (1cm);
\draw (90:1) -- (90:0);
\draw (90:-1) -- (90:-0.6);
\draw (90:0) .. controls +(-0.3,-0.3) and +(-0.3, 0.3) .. (90: -0.6);
\draw (90:0) .. controls +(0.3,-0.3) and +(0.3, 0.3) .. (90: -0.6);
\end{scope}

\begin{scope}[xshift= 2cm, yshift=2cm]
  \filldraw[dotted, fill=gray!50!white] (0,0) circle (1.5cm);
  \filldraw[dotted, fill = white] (0, 0) circle (1cm);
\draw (60:1) -- (90:0.6) -- (90:0) -- (60:-1);
\draw (90:0.6) -- (120:1);
\draw (90:0) -- (120:-1);
\fill[white] (-0.38, 0.56) circle (0.1);
\fill[white] (0.38, 0.56) circle (0.1);
\end{scope}
\begin{scope}[xshift= 6cm, yshift=2cm]
  \filldraw[dotted, fill=gray!50!white] (0,0) circle (1.5cm);
  \filldraw[dotted, fill = white] (0, 0) circle (1cm);
\draw (60:1).. controls +(-120:0.5) and +(120:0.5) .. (-60:1);
\draw (120:1).. controls +(-60:0.5) and +(60:0.5) .. (-120:1);
\end{scope}
\begin{scope}[xshift= 10cm, yshift=2cm]
  \filldraw[dotted, fill=gray!50!white] (0,0) circle (1.5cm);
  \filldraw[dotted, fill = white] (0, 0) circle (1cm);
\draw (60:1) -- (90:0.6) -- (90:0) -- (60:-1);
\draw (90:0.6) -- (120:1);
\draw (90:0) -- (120:-1);
\end{scope}

\begin{scope}[xshift= 20cm, yshift=2cm]
  \filldraw[dotted, fill=gray!50!white] (0,0) circle (1.5cm);
  \filldraw[dotted, fill = white] (0, 0) circle (1cm);
\draw (90:-1) -- (90:1);
\end{scope}
\begin{scope}[xshift= 16cm, yshift=2cm]
  \filldraw[dotted, fill=gray!50!white] (0,0) circle (1.5cm);
  \filldraw[dotted, fill = white] (0, 0) circle (1cm);
\draw (90:1) -- (90:0);
\draw (90:-1) -- (90:-0.6);
\draw (90:0) .. controls +(-0.3,-0.3) and +(-0.3, 0.3) .. (90: -0.6);
\draw (90:0) .. controls +(0.3,-0.3) and +(0.3, 0.3) .. (90: -0.6);
\end{scope}

\end{scope} 
\end{tikzpicture}
\]

In the middle of the movie, there is a vertical digon in $\G(\alpha(L,x, l_1, l_2))$, we now use the digon relation (see definition~\ref{pd:functor}). Hence, $\G(\alpha(L,x, l_1, l_2))$ can be seen as the difference of two morphisms. Using the fact that the images by $\G$ of the elementary moves A$i$ and B$i$ (for $i$ equal to $8$ or $9$) are inverse one from the other, we obtain that $\G(\alpha(L,x, l_1, l_2))$ is homotopic to the identity everywhere  but in a neighborhood of $x$ where we have:
\[
 \G(\alpha(L,x, l_1, l_2)) = \foamgid[0.3] \circ \foamdigd[0.3] - \foamgidd[0.3] \circ \foamdig[0.3]. 
\]

Using the surgery relation, the evaluation of dotted theta-foams and dotted sphere, we obtain:
\[
\foamgid[0.3] \circ \foamdigd[0.3] =\vcenter{\hbox{\! \tikz[scale=0.3]{\draw(-2,0) --(2,0) --(2,3) --(-2,3) --(-2,0); }}}  \quad \textrm{and}\quad \foamgidd[0.3] \circ \foamdig[0.3]=- \vcenter{\hbox{\! \tikz[scale=0.3]{\draw(-2,0) --(2,0) --(2,3) --(-2,3) --(-2,0); }}}  .
\]
Hence $\G(\alpha(L,x, l_1, l_2))$ is homotopic to $2\cdot\id_{\G(L)}$.
\end{proof}

\subsubsection{The movie $\nu$}
\label{sec:cobordism-nu}
We consider an oriented  link diagram $L$ in the standard disk $B_0$ and $L_{//}$ a cabling of $L$ along one component $l$ with two parallel copies $l_1$ and $l_2$  one oriented likes $l$ the other one with the opposite orientation. We consider as well $L_0$, this is the link diagram obtained by removing $l$ from $L$.
Both $L_0$ and $L_{//}$ are considered as objects of $\mathcal{TD}_{\emptyset}$. Let $x$ be a regular point of $l$.

\begin{dfn}
  \label{dfn:numorphism}
  The movie $\nu(L,x, l_1, l_2)$ in $\mathrm{HOM}_{\mathcal{TD_{\emptyset}}}(L_{//}, L_0)$ is defined by the following sequence:
\begin{itemize}
\item A saddle (B1) at $x$.
\item A sequence of moves of type B6 by pushing the strand along itself.
\item A cap (A1) in a neighborhood of $x$. 
\end{itemize}
An example is given in figure~\ref{fig:examp7trefoil2}
\begin{figure}[ht]
  \centering
\includegraphics{\imagesfolder/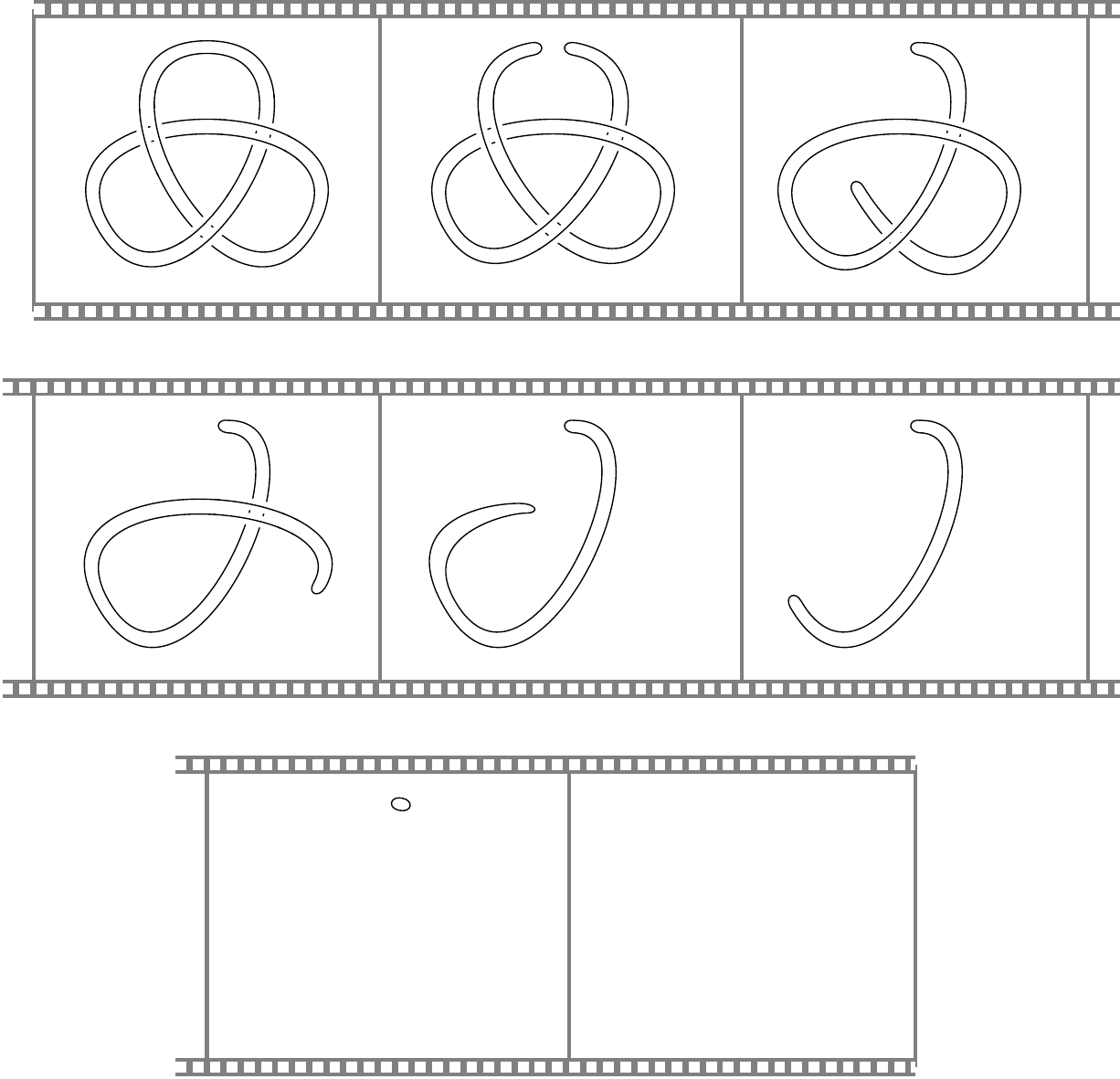} 
  \caption{The morphism $\nu(L,x)$ performed on the cabled trefoil.}
  \label{fig:examp7trefoil2}
\end{figure}
\end{dfn}

The same arguments as for the movie $\lambda$ gives the following two lemmas:

\begin{lem}\label{lem:startpointinvariancenu}
  With the same notations and with $y$ another regular point of $l$, we have:
  \[\G(\nu(L,x, l_1, l_2)) = \pm \G(\nu(L,y, l_1, l_2)).\]
\end{lem}

\begin{lem}\label{lem:Reidemeisternuinvariance}
With the same notations, if $L$ and $L'$ are two oriented link diagrams which are framed equivalent via a movie $m$, then the following diagram homotopy-commutes up to a sign:
\[
\begin{tikzpicture}[scale=2]
  \node (A) at (0,0) {$\G(L_{0})$};
  \node (B) at (2,0) {$\G(L'_{0})$};
  \node (C) at (0,1) {$\G(L_{//})$};
  \node (D) at (2,1) {$\G(L'_{//})$};
\draw [->] (A) -- (B) node [midway, below] {$\G(m_0)$};
\draw [->] (C) -- (D) node [midway, above] {$\G(m_{//})$};
\draw [->] (C) -- (A) node [midway, left] {$\G(\nu(L,x, l_1, l_2))$};
\draw [->] (D) -- (B) node [midway, right] {$\G(\nu(L', x', l_1, l_2))$};
\end{tikzpicture}
\]

Where $m_{//}$ is the movie $m$ adapted to $L_{//}$ and $m_{0}$, is the movie $m$ adapted to $L_{0}$. 
\end{lem}

\begin{rmk}
  \label{rmk:clarkwouldhavedoneit}
Note that theorem~\ref{thm:clark} implies the two previous lemma since the movie $\nu$ describe a surface.
\end{rmk}

We have an analogue to lemma~\ref{lem:lambda2timesid}: we consider $\nu(L,x,l_1,l_2)$ read in the other direction and denote this movie $\oldnu(L,x, l_1,l_2)$.  Composing these two movies, we obtain a movie $\beta(L,x, l_1, l_2)$ from $L_0$ to $L_0$, 
\begin{lem}
  \label{lem:nu3timesid}
   This endomorphism $\G(\beta(L_{0},x, l_1, l_2))$ is homotopic to $3\cdot\id_{\G(L_{0})}$.
\end{lem}

The proof is similar. We use the surgery relation instead of the digon relation. One could as well use theorem~\ref{thm:clark}: in terms of surface $\beta(L_{0},x, l_1, l_2)$ is the disjoint union of the identity of $L_0$ with a knotted torus, hence  we clearly have that $\G(\beta(L_{0},x, l_1, l_2))$ is homotopic to $3\cdot\id_{\G(L_{0})}$.


\subsection{The colored $\sll_3$-homology in the real world}
\label{sec:color-sll_3-homol-1}

We now work with the canopolis $\mathcal{KW}_{/h}$, which define just like $\mathcal{KW}$ except that morphisms are considered up to homotopy. We now consider $\G$ as a morphism between $\mathcal{FT}$ and $\mathcal{KW}_{/h}$ and the complexes $C(D)$ (defined in section~\ref{sec:sll_3-homology})  are seen as objects of $\mathcal{KW}_{/h}(\emptyset)$.  The aim of this part is to redefine the complex $\C_\bullet(D, V_{m,n})$ introduced in section~\ref{sec:overview} and to re-state the theorem~\ref{thm:coloredhomwithcjc}.

\begin{dfn}
 \label{dfn:cobincomplex2}
Let $D$ be an oriented link diagram with $k$ components, $\mathbf{m}$ and $\mathbf{n}$ be two $k$-tuples of non-negative integers and $e: P_1 \To P_2$ an edge in $\Gamma_{\mathbf{m},\mathbf{n}}$. We associate with $e$ and $D$ a movie $f_D(e)$ between $D_{P_1}$ and $D_{P_2}$: It is the identity on all strands not concerned with the edge $e$ and on the strands concerned by $e$, we use the following table:
\begin{center}
  \begin{tabular}[ht]{|c||c|}
    \hline $e$ & $f_D(e)$ \\ \hline \hline
\NB{\dops{(0,0)} \doms{(1,0)} \draw[->](1.4,0) -- (2.6,0); \draw[very thick, gray](3,0) -- +(1,0); \dops{(3,0)} \doms{(4,0)} } & $\nu$ \\ \hline
\NB{\dops{(0,0)} \dops{(1,0)} \draw[->](1.4,0) -- (2.6,0); \draw[very thick, gray](3,0) -- +(1,0); \dops{(3,0)} \dops{(4,0)} } & $\lambda $ \\ \hline
\NB{\doms{(0,0)} \doms{(1,0)} \draw[->](1.4,0) -- (2.6,0); \draw[very thick, gray](3,0) -- +(1,0); \doms{(3,0)} \doms{(4,0)} } & $\lambda$ \\ \hline
\NB{\draw[very thick, gray](0,0) -- +(1,0);\dops{(0,0)} \dops{(1,0)} \dops{(2,0)} \draw[->](2.4,0) -- (3.6,0); \draw[very thick, gray](4,0) -- +(2,0); \dops{(4,0)} \dops{(5,0)} \dops{(6,0)} } & $\nu$ \\ \hline
\NB{\draw[very thick, gray](1,0) -- +(1,0);\dops{(0,0)} \dops{(1,0)} \dops{(2,0)} \draw[->](2.4,0) -- (3.6,0); \draw[very thick, gray](4,0) -- +(2,0); \dops{(4,0)} \dops{(5,0)} \dops{(6,0)} } & $\nu$ \\ \hline
\NB{\draw[very thick, gray](0,0) -- +(1,0);\doms{(0,0)} \doms{(1,0)} \doms{(2,0)} \draw[->](2.4,0) -- (3.6,0); \draw[very thick, gray](4,0) -- +(2,0); \doms{(4,0)} \doms{(5,0)} \doms{(6,0)} } & $\nu$ \\ \hline
\NB{\draw[very thick, gray](1,0) -- +(1,0);\doms{(0,0)} \doms{(1,0)} \doms{(2,0)} \draw[->](2.4,0) -- (3.6,0); \draw[very thick, gray](4,0) -- +(2,0); \doms{(4,0)} \doms{(5,0)} \doms{(6,0)} } & $\nu$ \\ \hline
  \end{tabular}
\end{center}
\end{dfn}

\begin{propdfn}
  \label{pd:complex2}
Let $\mathbf{m}$ and $\mathbf{n}$ be two $k$-tuples of non-negative integers.
We consider an oriented link diagram $D$ colored with $V_{\mathbf{m}, \mathbf{n}}$. 
We can find some adequate signs such that the following formula define a complex $\C_\bullet(D, V_{m,n})$: 
\begin{align*}
  \C_i(D, V_{m,n})&= \bigoplus_{\textrm{$P$ of degree $i$}} C(D_{P}) \quad \textrm{and}\\
  d_i&: \C_i\To \C_{i+1},\\
  d_i&=\sum_{\textrm{$P$ of degree $i$}} \sum_{P_1\stackrel{e}{\To} P_2} d_e \quad \textrm{with}\quad d_e=(-1)^\bullet \G(f_e(D)).
\end{align*}
Furthermore, two satisfactory choices of signs yield isomorphic complexes.
\end{propdfn}

\begin{proof} 
We want to use lemma~\ref{lem:precom2com}.
 $\C_\bullet(D, V_{m,n})$ can  be seen as $S$-shaped space with $S$ strongly inductive. Hence, we only need to show that:
  \begin{enumerate}
  \item\label{it:commute} every square either commutes or anti-commutes,
  \item\label{it:nonzero} the product by $2$ of any composition of 3 maps $d_\bullet$ is not equal to zero.
  \end{enumerate}
Property~(\ref{it:commute}) follows from proposition~\ref{prop:importantmoves} and property~(\ref{it:nonzero}) follows from lemmas~\ref{lem:lambda2timesid} and \ref{lem:nu3timesid}.
\end{proof}

\begin{thm}
  \label{thm:main}
   The isomorphism type of $\C_\bullet(D, V_{\mathbf{m},\mathbf{n}})$ depends only on $(\mathbf{m}$, $\mathbf{n})$ and on the framed oriented isotopy type of $D$ and we have:
\[
\chi(\C_\bullet(D, V_{\mathbf{m},\mathbf{n}})) = e^{i\pi\frac23(n_+-n_-)s(\mathbf{m}-\mathbf{n})}\kup{(D, V_{\mathbf{m},\mathbf{n}})},
\]
where: 
\begin{itemize}
\item $(D, V_{\mathbf{m},\mathbf{n}})$ denotes the diagram $D$ colored with the $\Usl$-modules $V_{\mathbf{m},\mathbf{n}}$;
\item $n_+$ and $n_-$ are respectively the number of positive and negative crossings
in $D$;
\item $s(\mathbf{m}- \mathbf{n})$ is the sum over all coordinate of $(\mathbf{m}-\mathbf{n})$. 
\end{itemize}
\end{thm}

\begin{proof}
  The invariance follows directly from the fact that cabling is a well-defined operation on framed links and from lemmas~\ref{lem:ReidemeisterYinvariance} and \ref{lem:ReidemeisterYinvariance}. The Euler characteristic part is a consequence of proposition~\ref{prop:formulakup}.
\end{proof}

\begin{rmk}
  \label{rmk:diffmatters}
  One could define another complex by declaring all the differential to be zero. This would be a substantial simplification. However this would lead to an invariant completely determine by the $\sll_3$-homology of cables of links ignoring the relation between the different cables, this does not seem very satifactory.
\end{rmk}


\subsection{Cobordisms}
\label{sec:cobordisms}

A cobordism in $\RR^4$ between two colored framed links $L_1$ and $L_2$ can be
presented by a movie (the link diagrams in the movie should be colored). We can decompose this movie into elementary movies: Reidemeister moves (A5, A6, A7, B4, B5, B6 and B7) or Morse moves (A1, A2 or B1, note that for the saddle move (A2) the involved strands should have the same color).  One may wonder if like in the classical $\sll_3$-homology, the colored $\sll_3$-homology would extend to cobordisms. To address this question, one should define  morphisms associated with each elementary movies and prove that the movies move hold (see \cite{MR1905687}, \cite{MR2482322} and \cite{MR2462446} for the framed case). In this section we define some morphisms (we leave all the proofs to the reader) associated with Morse moves (this is very similar to \cite{MR2124557} and \cite{MR2462446}). We do not pretend to have any statement regarding the movie moves.

For colored cups and caps it is pretty simple: the associated morphism is induced by the unit and the counit of the Frobenius algebra $\mathcal{A}$ associated to the circle. Let us be a little more explicit. Let $D$ be a colored link diagram and $D_o$ with the same colored link diagram with an additional disjoint circle colored with $V_{m,n}$. If $P$ be an admissible multi-partition for $D$, we denote by $P_o$ is the multi-partition for $D_o$, which is determined by $P$ and by the canonical partition on the additional circle. Then the morphism associated with a the colored cup from $D$ to $D_o$ is $\phi=\sum_P \phi_P$, where $\phi_P: C(D_P) \to C((D_o)_{P_o})$ is the image by $\G$ of a succession of $m+n$ cups (oriented in the appropriate way). Here is what it looks like for $m=2$ and $n=1$:
\[
\begin{tikzpicture}[scale=0.45]
  \begin{scope}
\begin{scope}
\draw[draw=  gray, line width = 2mm] (0,0) --  (8,0); 
\draw [draw = gray, line width = 2mm] (8,4) -- (0,4); 
\draw[draw=  gray, very thick] (0,0) -- +(0,4);
\draw[draw=  gray, very thick] (4,0) -- +(0,4);
\draw[draw=  gray, very thick] (8,0) -- +(0,4);
\draw[draw= white, dotted, line width =1.2mm] (0,0) -- (8,0);
\draw[draw= white, dotted, line width = 1.2mm] (0,4) -- (8,4);
\end{scope}
\begin{scope}[xshift = 6cm, yshift = 2cm, decoration={markings, mark=at
     position 0.5 with {\arrow{>}}},postaction={decorate}]
  \draw[postaction = {decorate}] (-0.3, 0) circle (-1cm);
  \node at (1.2, 1) {$V_{2,1}$};
\end{scope}
\end{scope}
\node at (9,2) {$\mapsto$};
\begin{scope}[xshift = 10cm]
\begin{scope}
\draw[draw=  gray, line width = 2mm] (0,0) --  (16,0); 
\draw[draw = gray, line width = 2mm] (16,4) -- (0,4); 
\draw[draw=  gray, very thick] (0,0) -- +(0,4);
\draw[draw=  gray, very thick] (4,0) -- +(0,4);
\draw[draw=  gray, very thick] (8,0) -- +(0,4);
\draw[draw=  gray, very thick] (12,0) -- +(0,4);
\draw[draw=  gray, very thick] (16,0) -- +(0,4);
\draw[draw= white, dotted, line width =1.2mm] (0,0) -- (16,0);
\draw[draw= white, dotted, line width = 1.2mm] (0,4) -- (16,4);
\end{scope}
\begin{scope}[decoration={markings, mark=at
     position 0.5 with {\arrow{>}}},postaction={decorate}]
  \draw[postaction={decorate}] (10,2) circle (-1.3cm);
  \draw[postaction={decorate}] (14,2) circle (-1cm);
  \draw[postaction={decorate}] (14,2) circle (-1.3cm); 
\end{scope}

\begin{scope}[decoration={markings, mark=at
     position 0.5 with {\arrow{<}}},postaction={decorate}]
  \draw[postaction={decorate}] (6,2) circle (-1.6cm);
  \draw[postaction={decorate}] (10,2) circle (-1.6cm);
  \draw[postaction={decorate}] (14,2) circle (-1.6cm); 
\end{scope}
\end{scope}
\end{tikzpicture}
\]
For the cap this is the same read in the other direction, in particular the map induced by a cap is equal to equal to zero on every space $C((D_o)_{P})$ when $P$ does not induce the canonical partition on the additional circle.
 \marginpar{introduce multipartition}

\begin{claim}
   \label{cl:capandcup}
   The maps associated to the  cup and the cap colored with $V_{m,n}$ is a chain map. The $q$-degrees of these maps are both equal to $-2(m+n)$.
\end{claim}

A saddle can either merge two link components or split one link component. We define two different maps associated to these two different situations.


Let $D$ be a colored link diagram and $D'$ obtain from $D$ by a merging saddle move. Let us denote by $l_1$ and $l_2$ the two components of $D$ concerned by the saddle, and $l_0$ the component of $D'$ concerned with the saddle.
Let $P$ be a multi-partition for $D$, it induces a partition $P_1$ of $l_1$ and a partition $P_2$ for $l_2$. The morphism $\sigma$ is defined to be zero on $C(D_P)$ unless $P_1\pitchfork P_2$ and in these cases
it maps  $C(D_P)$ into $C(D'_{P'})$, where $P'$ is induced by $P$ for all components of $D'$ except $l_0$. The partition $P_0$ of $l_0$ is $P_1 \vee P_2$.

Suppose $P_1\pitchfork P_2$, and let us choose
\begin{align*}
&  P_1=P_1^1 \stackrel{e_1^1}{\To} P_1^2 \stackrel{e_1^2}{\To}  \dots \stackrel{e_1^{k-2}}{\To}P_1^{k-1} \stackrel{e_1^{k-1}}{\To} P_1^k= P_1\vee P_2,\quad \textrm{and} \\
&  P_2=P_2^1 \stackrel{e_2^1}{\To} P_2^2 \stackrel{e_2^2}{\To}  \dots \stackrel{e_2^{r-2}}{\To}P_2^{r-1} \stackrel{e_2^{r-1}}{\To} P_2^r= P_1\vee P_2. \\
\end{align*}
We define
\[\sigma_{|C(D_P)} = (-1)^\bullet\left(j_{P_1^k \stackrel{e_1^{k-1}}{\To} P_1^{k-1}}\!\! \circ\, 
\dots\, \circ j_{P_1^{1} \stackrel{e_1^{1}}{\To} P_1^2}\right) \circ  \left(j_{P_2^k \stackrel{e_2^{k-1}}{\To} P_2^{k-1}} \!\!\circ 
\,\dots\, \circ j_{P_2 \stackrel{e_2^{1}}{\To} P_2^2}\right) \circ s,
\]
where $s$ is the image by $\G$ of a composition of saddles, and $j_{P_a \stackrel{e}{\To} P_b}$ is the image by $\G$ of $f_D(e)$ precomposed with an endomorphism of $C(D_{P_a})$ which add dots on the strands concerned with $e$:
\begin{itemize}
\item We add one dot on each strand if $f_D(e)$ is a movie of type $\nu$,
\item We add a dot on the left strand (with respect to the orientation of the strands) if $f_D(e)$ is a movie of type $\lambda$,
\end{itemize}
On figure~\ref{fig:exsaddle}, we schematize\footnote{The second movie has no rigorous meaning in the canopolis $\mathcal{FT}$.} this construction for $m=3$ and $n=1$ when $P_1$ and $P_2$ are both the canonical partition and when they are equal to the partitions $P_1$ and $P_2$ of example~\ref{exa:lup}.
\begin{figure}
\centering
\begin{tikzpicture}[scale=0.6, yscale = -1]
  \begin{scope}
\draw[draw=  gray, line width = 2mm] (0,0) --  (20,0); 
\draw[draw = gray, line width = 2mm] (20,4) -- (0,4); 
\draw[draw=  gray, very thick] (0,0) -- +(0,4);
\draw[draw=  gray, very thick] (4,0) -- +(0,4);
\draw[draw=  gray, very thick] (8,0) -- +(0,4);
\draw[draw=  gray, very thick] (12,0) -- +(0,4);
\draw[draw=  gray, very thick] (16,0) -- +(0,4);
\draw[draw=  gray, very thick] (20,0) -- +(0,4);
\draw[draw= white, dotted, line width =1.2mm] (0,0) -- (20,0);
\draw[draw= white, dotted, line width = 1.2mm] (0,4) -- (20,4);
\end{scope}

\begin{scope}[xshift=2cm, yshift =2cm, decoration={markings, mark=at
     position 0.5 with {\arrow{>}}},postaction={decorate}]
  \draw[postaction={decorate}, xshift = 0.0cm] (-1.5, 1.5) .. controls +(-70:1)  and   +(70:1)   ..(-1.5, -1.5);
  \draw[postaction={decorate}, xshift = 0.3cm] (-1.5, 1.5) .. controls +(-70:1)  and  +(70:1)   ..(-1.5, -1.5);
  \draw[postaction={decorate}, xshift = 0.6cm] (-1.5, 1.5) .. controls +(-70:1)  and  +(70:1)   ..(-1.5, -1.5);
  \draw[postaction={decorate}, xshift = 0.9cm, yscale=-1] (-1.5, 1.5) .. controls +(-70:1)  and  +(70:1)   ..(-1.5, -1.5);
  \begin{scope}[xscale= -1, yscale= -1]
    \draw[postaction={decorate}, xshift = 0.0cm] (-1.5, 1.5) .. controls +(-70:1)  and   +(70:1)   ..(-1.5, -1.5);
  \draw[postaction={decorate}, xshift = 0.3cm] (-1.5, 1.5) .. controls +(-70:1)  and  +(70:1)   ..(-1.5, -1.5);
  \draw[postaction={decorate}, xshift = 0.6cm] (-1.5, 1.5) .. controls +(-70:1)  and  +(70:1)   ..(-1.5, -1.5);
  \draw[postaction={decorate}, xshift = 0.9cm, yscale = -1] (-1.5, 1.5) .. controls +(-70:1)  and  +(70:1)   ..(-1.5, -1.5);
\end{scope}
\end{scope}

\begin{scope}[xshift=6cm, yshift =2cm, decoration={markings, mark=at
     position 0.5 with {\arrow{>}}},postaction={decorate}]
  \draw[postaction={decorate}, xshift = 0.0cm] (-1.5, 1.5) .. controls +(-70:1)  and   +(70:1)   ..(-1.5, -1.5);
  \draw[postaction={decorate}, xshift = 0.3cm] (-1.5, 1.5) .. controls +(-70:1)  and  +(70:1)   ..(-1.5, -1.5);
  \draw[postaction={decorate}, xshift = 0.6cm] (-1.5, 1.5) .. controls +(-70:1)  and  +(70:1)   ..(-1.5, -1.5);
  \draw[postaction={decorate}, xshift = 0.cm, yscale=-1] (-0.6, 1.5) .. controls +(-70:1)  and  +(-110:1)   ..(0.6, 1.5);
  \draw[postaction={decorate}, xshift = 0.cm, xscale = -1] (-0.6, 1.5) .. controls +(-70:1)  and  +(-110:1)   ..(0.6, 1.5);
  \begin{scope}[xscale= -1, yscale= -1]
    \draw[postaction={decorate}, xshift = 0.0cm] (-1.5, 1.5) .. controls +(-70:1)  and   +(70:1)   ..(-1.5, -1.5);
  \draw[postaction={decorate}, xshift = 0.3cm] (-1.5, 1.5) .. controls +(-70:1)  and  +(70:1)   ..(-1.5, -1.5);
  \draw[postaction={decorate}, xshift = 0.6cm] (-1.5, 1.5) .. controls +(-70:1)  and  +(70:1)   ..(-1.5, -1.5);
\end{scope}
\end{scope}
\begin{scope}[xshift=10cm, yshift =2cm, decoration={markings, mark=at
     position 0.5 with {\arrow{>}}},postaction={decorate}]
  \draw[postaction={decorate}, xshift = 0.0cm] (-1.5, 1.5) .. controls +(-70:1)  and   +(70:1)   ..(-1.5, -1.5);
  \draw[postaction={decorate}, xshift = 0.3cm] (-1.5, 1.5) .. controls +(-70:1)  and  +(70:1)   ..(-1.5, -1.5);
  \draw[postaction={decorate}, xshift = 0.cm, yscale=1] (-0.9, 1.5) .. controls +(-70:1.3)  and  +(-110:1.3)   ..(0.9, 1.5);
  \draw[postaction={decorate}, xshift = 0.cm, rotate=180] (-0.9, 1.5) .. controls +(-70:1.3)  and  +(-110:1.3)   ..(0.9, 1.5);
  \draw[postaction={decorate}, xshift = 0.cm, yscale=-1] (-0.6, 1.5) .. controls +(-70:1)  and  +(-110:1)   ..(0.6, 1.5);
  \draw[postaction={decorate}, xshift = 0.cm, xscale = -1] (-0.6, 1.5) .. controls +(-70:1)  and  +(-110:1)   ..(0.6, 1.5);
  \begin{scope}[xscale= -1, yscale= -1]
    \draw[postaction={decorate}, xshift = 0.0cm] (-1.5, 1.5) .. controls +(-70:1)  and   +(70:1)   ..(-1.5, -1.5);
  \draw[postaction={decorate}, xshift = 0.3cm] (-1.5, 1.5) .. controls +(-70:1)  and  +(70:1)   ..(-1.5, -1.5);
\end{scope}
\end{scope}
\begin{scope}[xshift=14cm, yshift =2cm, decoration={markings, mark=at
     position 0.5 with {\arrow{>}}},postaction={decorate}]
  \draw[postaction={decorate}, xshift = 0.0cm] (-1.5, 1.5) .. controls +(-70:1)  and   +(70:1)   ..(-1.5, -1.5);
  \draw[postaction={decorate}, xshift = 0.cm, yscale=1] (-1.2, 1.5) .. controls +(-70:1.6)  and  +(-110:1.6)   ..(1.2, 1.5);
  \draw[postaction={decorate}, xshift = 0.cm, rotate=180] (-1.2, 1.5) .. controls +(-70:1.6)  and  +(-110:1.6)   ..(1.2, 1.5);
  \draw[postaction={decorate}, xshift = 0.cm, yscale=1] (-0.9, 1.5) .. controls +(-70:1.3)  and  +(-110:1.3)   ..(0.9, 1.5);
  \draw[postaction={decorate}, xshift = 0.cm, rotate=180] (-0.9, 1.5) .. controls +(-70:1.3)  and  +(-110:1.3)   ..(0.9, 1.5);
  \draw[postaction={decorate}, xshift = 0.cm, yscale=-1] (-0.6, 1.5) .. controls +(-70:1)  and  +(-110:1)   ..(0.6, 1.5);
  \draw[postaction={decorate}, xshift = 0.cm, xscale = -1] (-0.6, 1.5) .. controls +(-70:1)  and  +(-110:1)   ..(0.6, 1.5);
  \begin{scope}[xscale= -1, yscale= -1]
    \draw[postaction={decorate}, xshift = 0.0cm] (-1.5, 1.5) .. controls +(-70:1)  and   +(70:1)   ..(-1.5, -1.5);
\end{scope}
\end{scope}
\begin{scope}[xshift=18cm, yshift =2cm, decoration={markings, mark=at
     position 0.5 with {\arrow{>}}},postaction={decorate}]
  \draw[postaction={decorate}, xshift = 0.cm, yscale=1] (-1.5, 1.5) .. controls +(-70:1.9)  and  +(-110:1.9)   ..(1.5, 1.5);
  \draw[postaction={decorate}, xshift = 0.cm, rotate=180] (-1.5, 1.5) .. controls +(-70:1.9)  and  +(-110:1.9)   ..(1.5, 1.5);
  \draw[postaction={decorate}, xshift = 0.cm, yscale=1] (-1.2, 1.5) .. controls +(-70:1.6)  and  +(-110:1.6)   ..(1.2, 1.5);
  \draw[postaction={decorate}, xshift = 0.cm, rotate=180] (-1.2, 1.5) .. controls +(-70:1.6)  and  +(-110:1.6)   ..(1.2, 1.5);
  \draw[postaction={decorate}, xshift = 0.cm, yscale=1] (-0.9, 1.5) .. controls +(-70:1.3)  and  +(-110:1.3)   ..(0.9, 1.5);
  \draw[postaction={decorate}, xshift = 0.cm, rotate=180] (-0.9, 1.5) .. controls +(-70:1.3)  and  +(-110:1.3)   ..(0.9, 1.5);
  \draw[postaction={decorate}, xshift = 0.cm, yscale=-1] (-0.6, 1.5) .. controls +(-70:1)  and  +(-110:1)   ..(0.6, 1.5);
  \draw[postaction={decorate}, xshift = 0.cm, xscale = -1] (-0.6, 1.5) .. controls +(-70:1)  and  +(-110:1)   ..(0.6, 1.5);
\end{scope} 
\end{tikzpicture} \\ \vspace{0.3cm}
\begin{tikzpicture}[scale=0.6, yscale = 1]
  \begin{scope}
\draw[draw=  gray, line width = 2mm] (0,0) --  (16,0); 
\draw[draw = gray, line width = 2mm] (16,4) -- (0,4); 
\draw[draw=  gray, very thick] (0,0) -- +(0,4);
\draw[draw=  gray, very thick] (4,0) -- +(0,4);
\draw[draw=  gray, very thick] (8,0) -- +(0,4);
\draw[draw=  gray, very thick] (12,0) -- +(0,4);
\draw[draw=  gray, very thick] (16,0) -- +(0,4);
\draw[draw= white, dotted, line width =1.2mm] (0,0) -- (16,0);
\draw[draw= white, dotted, line width = 1.2mm] (0,4) -- (16,4);
\end{scope}

\begin{scope}[xshift=2cm, yshift =2cm, decoration={markings, mark=at
     position 0.5 with {\arrow{>}}},postaction={decorate}]
  \draw[postaction={decorate}, xshift = 0.0cm] (-1.5, -1.5) .. controls +(70:1)  and   +(-70:1)  ..(-1.5, 1.5);
  \draw[postaction={decorate}, xshift = 0.3cm] (-1.5, -1.5) .. controls +(70:1)  and  +(-70:1)   ..(-1.5, 1.5);
  \draw[dotted, xshift = 0.6cm] (-1.5, 1.5) .. controls +(-70:1)  and  +(70:1)   ..(-1.5, -1.5);
  \draw[dotted, xshift = 0.9cm] (-1.5, 1.5) .. controls +(-70:1)  and  +(70:1)   ..(-1.5, -1.5);
  \begin{scope}[xscale= -1, yscale= -1]
  \draw[dotted, xshift = 0.0cm] (-1.5, -1.5) .. controls +(70:1)  and   +(-70:1)   ..(-1.5, 1.5);
  \draw[postaction={decorate}, xshift = 0.15cm, yscale = -1] (-1.5, -1.5) .. controls +(70:1)  and  +(-70:1)   ..(-1.5, 1.5);
  \draw[dotted, xshift = 0.3cm] (-1.5, 1.5) .. controls +(-70:1)  and  +(70:1)   ..(-1.5, -1.5);
  \draw[postaction={decorate}, xshift = 0.6cm] (-1.5, -1.5) .. controls +(70:1)  and  +(-70:1)   ..(-1.5, 1.5);
  \draw[postaction={decorate}, xshift = 0.9cm, yscale = -1] (-1.5, -1.5) .. controls +(70:1)  and  +(-70:1)   ..(-1.5, 1.5);
\end{scope}
\end{scope}
\begin{scope}[xshift=6cm, yshift =2cm, decoration={markings, mark=at
     position 0.5 with {\arrow{>}}},postaction={decorate}]
  \draw[postaction={decorate}, xshift = 0.0cm] (-1.5, -1.5) .. controls +(70:1)  and   +(-70:1)   ..(-1.5, 1.5) node[pos=0.3] {$\bullet$} ;
  \draw[postaction={decorate}, xshift = 0.3cm] (-1.5, -1.5) .. controls +(70:1)  and  +(-70:1)   ..(-1.5, 1.5) ;
  \begin{scope}[xscale= -1, yscale= -1]
  \draw[postaction={decorate}, xshift = 0.15cm, yscale = -1] (-1.5, -1.5) .. controls +(70:1)  and  +(-70:1)   ..(-1.5, 1.5);
  \draw[postaction={decorate}, xshift = 0.6cm] (-1.5, -1.5) .. controls +(70:1)  and  +(-70:1)   ..(-1.5, 1.5) node[pos=0.7] {$\bullet$} ;
  \draw[postaction={decorate}, xshift = 0.9cm, yscale = -1] (-1.5, -1.5) .. controls +(70:1)  and  +(-70:1)   ..(-1.5, 1.5) node[pos=0.3] {$\bullet$} ;
\end{scope}
\end{scope}

\begin{scope}[xshift=10cm, yshift =2cm, decoration={markings, mark=at
     position 0.5 with {\arrow{>}}},postaction={decorate}]
  \draw[dotted, xshift = 0.0cm] (-1.5, -1.5) .. controls +(70:1)  and   +(-70:1)   ..(-1.5, 1.5);
  \draw[postaction={decorate}, xshift = 0.15cm, yscale = -1] (-1.5, -1.5) .. controls +(70:1)  and  +(-70:1)   ..(-1.5, 1.5);
   \draw[dotted, xshift = 0.3cm] (-1.5, -1.5) .. controls +(70:1)  and  +(-70:1)   ..(-1.5, 1.5) ;
  \begin{scope}[xscale= -1, yscale= -1]
  \draw[postaction={decorate}, xshift = 0.15cm, yscale = -1] (-1.5, -1.5) .. controls +(70:1)  and  +(-70:1)   ..(-1.5, 1.5);
  \draw[dotted, xshift = 0.6cm] (-1.5, -1.5) .. controls +(70:1)  and  +(-70:1)   ..(-1.5, 1.5);
  \draw[dotted, xshift = 0.9cm, yscale = -1] (-1.5, -1.5) .. controls +(70:1)  and  +(-70:1)   ..(-1.5, 1.5);
\end{scope}
\end{scope}

\begin{scope}[xshift=14cm, yshift =2cm, decoration={markings, mark=at
     position 0.5 with {\arrow{>}}},postaction={decorate}]
  \draw[postaction={decorate}, xshift =0, yscale = -1] (-1.35, -1.5) .. controls +(70:1)  and  +(110:1)   ..(1.35, -1.5);
  \draw[postaction={decorate}, xscale=-1 =, yscale = 1] (-1.35, -1.5) .. controls +(70:1)  and  +(110:1)   ..(1.35, -1.5);
\end{scope} 
\end{tikzpicture}
\caption{On the second movie we added some dashed lines to clarify: on the first frame they indicate the position of strands for the canonical partition. On the third, they indicate the positions of the strands of the second frame which disappeared between the second and the third frame. The $\bullet$'s on the second indicate that we add a dot.}
\label{fig:exsaddle}
\end{figure}
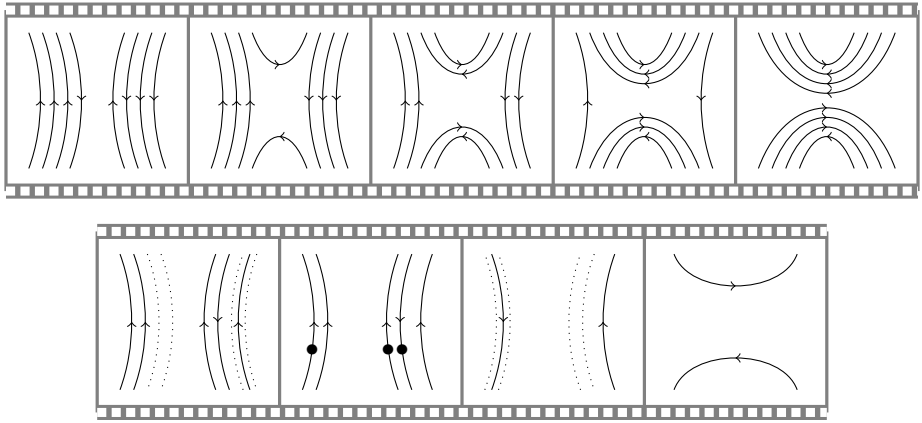

Note that we can choose different path between $P_i$ and $P_1\vee P_2$ ($i=1$ or $2$), but the composition of the $j$'s is well-defined up to a sign.

If the saddle split a component, we do the same construction in the opposite direction. We let the details to the reader.
\begin{claim}
   \label{cl:saddle}
One can fix the signs in the definition of the morphism associated with a saddle colored with $V_{m,n}$ to turn it into a chain map. The $q$-degree of this map is $2(m+n)$.
\end{claim}



\subsection{Problems}
\label{sec:problems}

We would like to suggest a list of problems which arise quite naturally from the problem we dealt with:
\begin{enumerate}
\item Relate this work with the categorification of the quantum $\sll_3$ projectors defined by Rose in \cite{MR3029720}.
\item Extends this construction to the $\sll_n$-homology defined by Queffelec and Rose in \cite{1212.6076}.
\item\label{item:3} Define geometrically the category of framed webs and framed foams, find a complete list of Reidemeister moves. Find a complete list of movies and of movie moves. One can start with the 2-dimensional moves introduced by Carter \cite{doi:10.1142/9789814630627_0001} and with the framed movie moves listed by Beliakova and Wehrli in \cite{MR2462446}.
\item Once (\ref{item:3}) is done, continue the work of Clark \cite{MR2482322} and prove that the framed $\sll_3$-homology is functorial for foam cobordisms.
\item Make section~\ref{sec:cobordisms} more explicit and establish the functoriality of the colored $\sll_3$-homology.
\item Following Turner and Mackaay\cite{MR2276518}, find a new formulation of the colored $\sll_3$-homology as a space of invariant of the homology of the cable under a certain action. In the $\sll_2$-case this is this is an action of the symmetric group. We suspect that in the $\sll_3$-case, we need to consider an algebra.
\item Extend algebraically this construction to tangles, like Caprau \cite{MR3220473} did for the $\sll_2$-case.
\end{enumerate}

\appendix 

\section{Complexes and cubes}
\label{sec:complexes-cubes}

\subsection{Cubical sets}
\label{sec:cubical-sets}

\begin{dfn}
  If $n$ is a non-negative integer, the \emph{$n$th dimensional cube} $B_n$ is the $\ZZ/2\ZZ$-vector space generated by all maps $\llbracket 1, n\rrbracket$ to $\{0,*, 1\}$ (which are called \emph{faces}).  The \emph{dimension} of a face is  given by the formula:
\[
\dim f = \# f^{-1}(\{*\}).
\]
The faces of dimension 0, 1, 2 and 3 are called vertices, edges, squares and cubes respectively. A \emph{cubical set} is a subspace of an $n$th dimensional cube generated by some faces of $B_n$. Let $f$ be a face of dimension $k$ of the $n$th dimensional cube. \emph{The boundary of a face $f$}, denoted  by $\partial f$ is the sum of all faces of $B_n$ of dimension $k-1$ which are equal to $f$ on exactly $n-1$ elements of $\llbracket 1, n \rrbracket$. A cubical set $S$ is \emph{consistent} if for every face $f$ of positive dimension, $f$ is in $S$ if and only if $\partial f$ is in $S$.
\end{dfn}

\begin{rmk}
  If $f$ has dimension $k$, $\partial f$ is a sum of $2k$ faces.
  If a cubical set is consistent it is determined by its vertices. Hence, if $X$ is a set of vertices the $n$th dimensional cube, we may speak of the consistent cubical set associated with $X$. 
\end{rmk} 

The set $\{0,*,1\}$ is endowed with the total order $0<*<1$. This induces a partial order on the set of faces.  

\begin{dfn}
 A cubical set $S\subset B_n$  is \emph{inductive} if for every face $f$ of $S$ and every face $f'$ of $B_n$, $f'<f$ implies $f'\in S$. It is \emph{strong-inductive} if furthermore all its maximal faces have dimension equal to 0.
\end{dfn}
\begin{rmk}
  An inductive cubical set is entirely determined by its set of maximal faces.
\end{rmk}
\begin{lem}
  If a cubical set $S$ is strong-inductive, it is consistent.
\end{lem}
\begin{proof}
  Let $f$ be a face of $S$, and let $h$ be a $0$ dimensional maximal element of $S$ such that $f<h$. Then $\partial f$ is a sum of faces which are all smaller than or equal than $h$, hence they are all in $S$. If $\partial f$ is in $S$, all the faces in the sum of $\partial f$ are in $S$. Half of them are greater than $f$, hence $f$ is in $S$.
\end{proof}
\begin{lem}
  Let $S$ be a consistent cubical set, then the map $\partial: S\to S$ is a differential (the homological degree is given by the dimension). 
\end{lem}
\begin{proof}
  From the definition of $\partial$, we see that it is a map of degree -1. It is clear that $\partial ^2= 0_S$ since we work over $\ZZ/2\ZZ$. 
\end{proof}
\begin{prop}\label{prop:homology-cubical-set}
  Let $S\subset B_n$ be a (non-trivial) strong-inductive cubical set, then we have:
  \[ H_0(S) = \ZZ/2\ZZ \quad \textrm{and} \quad H_i(S)= 0 \quad \textrm{for all $i > 0$.}
\]
\end{prop}
\begin{lem} For every positive integer $n$ we have:
  \[ H_0(B_n) = \ZZ/2\ZZ \quad \textrm{and} \quad H_i(B_n)= 0 \quad \textrm{for all $i > 0$.}
\]
\end{lem}

\begin{proof}
  This is an easy induction using the fact that $B_{n+1}$ is isomorphic as a complex to the cone of $B_n \stackrel{\id \oplus \id}\longrightarrow B_n \oplus B_n$.
\end{proof}
\begin{proof}[Proof of proposition\ref{prop:homology-cubical-set}]
  This is an easy induction on $n$ and on the number of maximal element in $S$: If there is one maximal element, $S$ is isomorphic to a $k$-th dimensional cube. If it has $i+1$ maximal faces of dimension 0, it is a union of two cubical sets: $S'$ with $i$ maximal faces of dimension 0 and $S''$ with 1 maximal face of dimension 0. 
Furthermore $S'\cap S''$ is isomorphic to a cubical set in $B_m$ with $m<n$.
We have the following short exact sequence:
\[
0\to S'\cap S'' \to S' \oplus S'' \to S \to 0.
\]
We conclude using the associated long exact sequence and the induction hypothesis.
\end{proof}

\begin{dfn}
  If $B_n$ is the $n$th dimensional cube, we denote by $B^*_n$ its $\ZZ/2\ZZ$-dual (as a complexe) and call it \emph{the $n$th dimensional cocube}. Similarly a \emph{cocubical set} is the $\ZZ/2\ZZ$-dual of a cubical set. If $S$ is a consistent cubical set we denote $H^\bullet(S)$ the homology of $S^*$.
\end{dfn}

Thanks to the universal coefficient theorem and proposition~\ref{prop:homology-cubical-set}, we have:
\begin{prop}\label{prop:cohomology-cubical-set}
  Let $S$ a (non-trivial) strong-inductive cubical set, then we have:
  \[ H^0(S) = \ZZ/2\ZZ \quad \textrm{and} \quad H^i(S)= 0 \quad \textrm{for all $i > 0$.}
\]
\end{prop}

\subsection{Cubical sets and (pre-)complexes}
\label{sec:cubic-pre-compl}
In this section $R$ is a unital ring.

\begin{dfn}
  Let $S$ cubical set, an \emph{$S$-shaped space $(C,d)$} (or simply $C$, when this is not ambiguous) consists of the following data:
  \begin{itemize}
  \item for every vertex $v$ of $S$, an $R$-module $C_v$,
  \item for every edge $e$ of $S$, an $R$-module map $d_e: C_{e_0} \to C_{e_1}$ (where $e_0$ and $e_1$ are the vertices of $S$ obtained by replacing the only '$*$' of $e$ by a '$0$' or '$1$' respectively).
  \end{itemize}
If $S$ is strong-inductive and if $s$ is a square in $S$. There are two maps we can consider: the two possible compositions of $d_\bullet$ which makes sense. We say that $C$ is an \emph{$S$-shaped pre-complex}:
\begin{itemize}
\item for every square $s$ the two maps are equal up to a sign (if they are equal we say that the square commutes, else that it anti-commutes),
\item any (compatible) composition of 3 maps $d_\bullet$ multiplied by $2$ is not equal to zero.
\end{itemize}
\end{dfn}
\begin{exa}
The hypercube in the definition of $\sll_3$ homology is a $B_n$-shaped pre-complex, while the hypercube defining the odd Khovanov homology (see~\cite{0710.4300}) is not a pre-complex.
\end{exa}
\begin{rmk}
  If $C$ is an $S$-shaped pre-complex, such that all its squares anti-commute (we say that $C$ \emph{anti-commutes}), then it is naturally endowed  a structure of complex: the space  of homological degree $i$ is then given by:
\[
C_i = \bigoplus_{\substack{ v\in S, \\ \# f^{-1}(\{1\})=i}}C_v.
\]
\end{rmk}
\begin{dfn}
  Let $S$ be a cubical set, a \emph{sign assignment} $\delta$ on $S$ is an application from the set of edges of $S$ to $\ZZ/2\ZZ$ (Hence it can be seen as an element of $S^*$). If $C$ is an $S$-shaped pre-complex and $\delta$ a sign assignment on $S$, the \emph{$\delta$-deformation} of $(C, d)$, denoted $C^\delta, d^\delta)$ (or simply $C^\delta$) is the $S$-shaped pre-complex defined by 
  \begin{itemize}
  \item for every vertex $v$ of $S$, $C^\delta _v:= C_v$,
  \item for every edge $e$ of $S$, $d^\delta_e :=(-1)^{\delta(e)}d_e$ where $\ZZ/2\ZZ$ is identified with $\{0,1\}$.
  \end{itemize}
\end{dfn}

\begin{lem}\label{lem:precom2com}
Let $S$ be a strong-inductive cubical set and $C$ an $S$-shaped pre-complex. Then there exists a sign assignment $\delta$ such that $C^\delta$ anti-commutes. Furthermore, if $\delta_1$ and $\delta_2$ are two such sign assignments, then the complexes $C^{\delta_1}$ and $C^{\delta_2}$ are isomorphic.
\end{lem}

\begin{proof}
Let us define a $\ZZ/2\ZZ$-linear form $\gamma$ on the $\ZZ/2\ZZ$-vector space generated by squares in $S$:
\[
\gamma(s) =
\begin{cases}
  1 & \textrm{if $s$ commutes,} \\
  0 & \textrm{if $s$ anti-commutes.}
\end{cases}
\]
This map is meant to encode the defect of anti-commutativity of $C$. We claim that this map (seen as an element the $S^*$) satisfies $\partial^*\gamma =0$. The map  $\partial^*\gamma$ is a linear form on the $\ZZ/2\ZZ$-vector space generated by cubes of $S$. We consider a cube $c$ in $S$ (that is a map from $\llbracket 1, n \rrbracket$ to $\{0, *, 1\}$ such that $*$ has exactly three pre-images) and the associated data in $C$ (for simplicity the indices only tracks the preimages of $*$ of the cube):
\[
\begin{tikzpicture}
\node (C000) at (-6, 0)  {$C_{000}$};
\node (C001) at (-2, 2)  {$C_{001}$};
\node (C010) at (-2, 0)  {$C_{010}$};
\node (C100) at (-2,-2)  {$C_{100}$};
\node (C011) at ( 2, 2)  {$C_{011}$};
\node (C101) at ( 2, 0)  {$C_{101}$};
\node (C110) at ( 2, -2)  {$C_{110}$};
\node (C111) at ( 6, 0)  {$C_{111}$};
\draw [->] (C000) -- (C001) node [midway, above] {$d_{00*}$};
\draw [->] (C000) -- (C010) node [midway, above] {$d_{0*0}$};
\draw [->] (C000) -- (C100) node [midway, above] {$d_{*00}$};
\draw [->] (C001) -- (C011) node [midway, above] {$d_{0*1}$};
\draw [->] (C001) -- (C101) node [near end, below] {$d_{*01}$};
\draw [->] (C010) -- (C110) node [near start, above] {$d_{*10}$};
\draw [->] (C010) -- (C011) node [near start, below] {$d_{01*}$};
\draw [->] (C100) -- (C110) node [midway, above] {$d_{1*0}$};
\draw [->] (C100) -- (C101) node [near end, above] {$d_{10*}$};
\draw [->] (C110) -- (C111) node [midway, above] {$d_{11*}$};
\draw [->] (C101) -- (C111) node [midway, above] {$d_{1*1}$};
\draw [->] (C011) -- (C111) node [midway, above] {$d_{*11}$};
\draw [line width = 0.8cm, rounded corners, red, opacity =0.3, line cap= round]  (C000.center)-- (C001.center) -- (C011.center) -- (C111.center);
\draw [line width = 0.8cm, rounded corners, green, opacity =0.3, line cap= round] (C000.center)-- (C100.center) -- (C110.center) -- (C111.center);
\end{tikzpicture}
\]
There are two ways to compare the two highlighted maps (the squares of $S$ are denoted $s_\bullet$): 
\begin{align*}
d_{11*}\circ d_{1*0} \circ d_{*00} 
&= (-1)^{1+ \gamma(s_{1**})} d_{1*1}\circ d_{10*} \circ d_{*00} \\
&= (-1)^{2+ \gamma(s_{*0*})+\gamma(s_{1**}) } d_{1*1}\circ d_{*01} \circ d_{00*} \\
&= (-1)^{3+ \gamma(s_{**1})+\gamma(s_{*0*})+\gamma(s_{1**})} d_{*11}\circ d_{0*1} \circ d_{00*}, \\
\end{align*}
and 
\begin{align*}
d_{11*}\circ d_{1*0} \circ d_{*00} 
&= (-1)^{1+ \gamma(s_{**0})} d_{11*}\circ d_{*10} \circ d_{0*0} \\
&= (-1)^{2+ \gamma(s_{*1*}) + \gamma(s_{**0})} d_{*11}\circ d_{01*} \circ d_{0*0} \\
&= (-1)^{3+ \gamma(s_{0**}) + \gamma(s_{*1*}) + \gamma(s_{**0})} d_{*11}\circ d_{0*1} \circ d_{00*}, \\
\end{align*}
$C$ being a pre-complex $2(d_{11*}\circ d_{1*0} \circ d_{*00})$ is not equal to zero, we have:
\[\gamma(s_{**1})+\gamma(s_{*0*})+\gamma(s_{1**}) + \gamma(s_{0**}) + \gamma(s_{*1*}) + \gamma(s_{**0}) = 0 \quad \textrm{in $\ZZ/2\ZZ$.}
\]
This means precisely that $\partial^*\gamma(c) =0$. Since this holds for all cube $c$, we have $\partial^*\gamma =0$. Thanks to proposition \ref{prop:cohomology-cubical-set}, we know that there exists $\delta$ a sign assignment such that $\partial^*\delta = \gamma$. We claim that $C^\delta$ anti-commute. Let $s$ be a square in $S$ we have:
\begin{align*}
d_{1*}^\delta \circ d^{\delta}_{*0} &= (-1)^{\delta(1*) + \delta(*0)} d_{1*} \circ d_{*0} \\
&=(-1)^{1+ \gamma(s) +\delta(1*) + \delta(*0)} d_{*1} \circ d_{0*} \\
&=(-1)^{1+ (\delta(0*) + \delta(*0) + \delta(1*) +\delta(*1)) +\delta(1*) + \delta(*0)} d_{*1} \circ d_{0} \\
&= (-1)^{1+ *\delta(0*)  +\delta(*1) } d_{*1} \circ d_{0*} \\
&= - d_{*1}^\delta \circ d_{0*}^\delta.
\end{align*}
This proves the claim. We consider $\delta_1$ and $\delta_2$ two sign assignments such that $C^{\delta_1}$ and $C^{\delta_2}$ anti-commutes. This means that $\partial^* \delta_1 =\partial^* \delta_2 = \gamma$. Hence $\partial^*(\delta_1+\delta_2) =0$. Thanks to proposition~\ref{prop:cohomology-cubical-set}, there exists a $\ZZ/2\ZZ$-linear form $\kappa$ on the $\ZZ/2\ZZ$-vector space generated by vertices of $S$ such that $\partial^* \kappa = \delta_1 +\delta_2$.
The map $f_\kappa: C\to C$ defined on $C_v$ by $(f_\kappa)_{|C_v}:=(-1)^{\kappa(v)} \id_{C_v}$ satisfies:
$f_\kappa\circ d^{\delta_1} = d^{\delta_2} \circ f_\kappa$ and provides an isomorphism between the complexes $C^{\delta_1}$ and $C^{\delta_2}$.
\end{proof}


\bibliographystyle{alpha}
\bibliography{../../Latex/biblio}

\begin{thebibliography}{Rob13b}

\bibitem[BN05]{MR2174270}
Dror Bar-Natan.
\newblock Khovanov's homology for tangles and cobordisms.
\newblock {\em Geom. Topol.}, 9:1443--1499 (electronic), 2005.

\bibitem[BW08]{MR2462446}
Anna Beliakova and Stephan Wehrli.
\newblock Categorification of the colored {J}ones polynomial and {R}asmussen
  invariant of links.
\newblock {\em Canad. J. Math.}, 60(6):1240--1266, 2008.

\bibitem[Cap14]{MR3220473}
Carmen Caprau.
\newblock A cohomology theory for colored tangles.
\newblock In {\em Knots in {P}oland. {III}. {P}art 1}, volume 100 of {\em
  Banach Center Publ.}, pages 13--25. Polish Acad. Sci. Inst. Math., Warsaw,
  2014.

\bibitem[Car]{doi:10.1142/9789814630627_0001}
J.~Scott Carter.
\newblock {\em Reidemeister/Roseman-Type Moves to Embedded Foams in
  4-Dimensional Space}, chapter~1, pages 1--30.

\bibitem[CKS02]{MR1905687}
J.~Scott Carter, Seiichi Kamada, and Masahico Saito.
\newblock Stable equivalence of knots on surfaces and virtual knot cobordisms.
\newblock {\em J. Knot Theory Ramifications}, 11(3):311--322, 2002.
\newblock Knots 2000 Korea, Vol. 1 (Yongpyong).

\bibitem[Cla09]{MR2482322}
David Clark.
\newblock Functoriality for the {${\mathfrak{su}_3}$} {K}hovanov homology.
\newblock {\em Algebr. Geom. Topol.}, 9(2):625--690, 2009.

\bibitem[Kho04]{MR2100691}
Mikhail Khovanov.
\newblock sl(3) link homology.
\newblock {\em Algebr. Geom. Topol.}, 4:1045--1081, 2004.

\bibitem[Kho05]{MR2124557}
Mikhail Khovanov.
\newblock Categorifications of the colored {J}ones polynomial.
\newblock {\em J. Knot Theory Ramifications}, 14(1):111--130, 2005.

\bibitem[KRT97]{kassel97:_quant}
Christian Kassel, Marc Rosso, and Vladimir Turaev.
\newblock {\em Quantum groups and knot invariants}.
\newblock Number~5 in Panoramas et Synthèse. Société Mathématique de
  France, 1997.

\bibitem[Kup96]{MR1403861}
Greg Kuperberg.
\newblock Spiders for rank {$2$} {L}ie algebras.
\newblock {\em Comm. Math. Phys.}, 180(1):109--151, 1996.

\bibitem[Lew13a]{MR3248745}
Lukas Lewark.
\newblock {$\mathfrak{sl}_3$}-foam homology calculations.
\newblock {\em Algebr. Geom. Topol.}, 13(6):3661--3686, 2013.

\bibitem[Lew13b]{LewarkThese}
Lukas Lewark.
\newblock {\em Les homologies de Khovanov-Rozansky, toiles nouées pondérées
  et le genre lisse}.
\newblock PhD thesis, Univeristé Paris 7 -- Denis Diderot, June 2013.

\bibitem[LQR12]{1212.6076}
Aaron~D. Lauda, Hoel Queffelec, and David E.~V. Rose.
\newblock Khovanov homology is a skew howe 2--representation of categorified
  quantum sl(m), 2012.

\bibitem[MN08]{MR2457839}
Scott Morrison and Ari Nieh.
\newblock On {K}hovanov's cobordism theory for {$\mathfrak{su}_3$} knot
  homology.
\newblock {\em J. Knot Theory Ramifications}, 17(9):1121--1173, 2008.

\bibitem[MT07]{MR2276518}
Marco Mackaay and Paul Turner.
\newblock Bar-{N}atan's {K}hovanov homology for coloured links.
\newblock {\em Pacific J. Math.}, 229(2):429--446, 2007.

\bibitem[MV07]{MR2336253}
Marco Mackaay and Pedro Vaz.
\newblock The universal {${\rm sl}_3$}--link homology.
\newblock {\em Algebr. Geom. Topol.}, 7:1135--1169, 2007.

\bibitem[MV08]{MR2443231}
Marco Mackaay and Pedro Vaz.
\newblock The foam and the matrix factorization {$\rm sl_3$} link homologies
  are equivalent.
\newblock {\em Algebr. Geom. Topol.}, 8(1):309--342, 2008.

\bibitem[Oht02]{ohtsuki02:_quant}
Tomotada Ohtsuki.
\newblock {\em Quantum invariants}.
\newblock Number~29 in Series on knots and everything. World Scientific, 2002.

\bibitem[ORS07]{0710.4300}
Peter Ozsvath, Jacob Rasmussen, and Zoltan Szabo.
\newblock Odd khovanov homology, 2007.

\bibitem[Rob13a]{LHR2}
Louis-Hadrien Robert.
\newblock {A characterisation of indecomposable web-modules over
  {K}hovanov--{K}uperberg algebras}.
\newblock {\em ArXiv e-prints}, September 2013.

\bibitem[Rob13b]{LHRThese}
Louis-Hadrien Robert.
\newblock {\em Sur l'homologie $\mathfrak{sl}_3$ des enchevêtrements; algèbre
  de {K}hovanov--{K}uperberg}.
\newblock PhD thesis, Université Paris 7 -- Denis Diderot, July 2013.

\bibitem[Ros12]{MR3029720}
David Emile~Vatcher Rose.
\newblock {\em Categorification of {Q}uantum sl3 {P}rojectors and the sl3
  {R}eshetikhin-{T}uraev {I}nvariant of {F}ramed {T}angles}.
\newblock PhD thesis, 2012.
\newblock Thesis (Ph.D.)--Duke University.

\bibitem[RT90]{MR1036112}
Nicolai~Yu. Reshetikhin and Vladimir~G. Turaev.
\newblock Ribbon graphs and their invariants derived from quantum groups.
\newblock {\em Comm. Math. Phys.}, 127(1):1--26, 1990.

\end{thebibliography}

\end{document}